\documentclass[11pt]{article}

\usepackage[margin=1in]{geometry}
\usepackage{amsmath,amssymb,amsthm,mathtools}
\usepackage[T1]{fontenc}
\usepackage{bm}
\usepackage[round,authoryear]{natbib}
\usepackage{authblk}
\usepackage{hyperref}
\usepackage{xcolor}
\usepackage{enumitem}
\usepackage{array,booktabs}

\hypersetup{colorlinks=true,linkcolor=blue,citecolor=blue,urlcolor=blue}

\newtheorem{theorem}{Theorem}[section]
\newtheorem{proposition}[theorem]{Proposition}
\newtheorem{corollary}[theorem]{Corollary}
\newtheorem{lemma}[theorem]{Lemma}
\newtheorem{remark}[theorem]{Remark}
\newtheorem{definition}[theorem]{Definition}
\newtheorem{example}[theorem]{Example}

\newcommand{\R}{\mathbb{R}}
\newcommand{\E}{\mathbb{E}}
\newcommand{\Pbb}{\mathbb{P}}
\newcommand{\norm}[1]{\left\lVert #1\right\rVert}
\newcommand{\abs}[1]{\left\lvert #1\right\rvert}
\newcommand{\ip}[2]{\left\langle #1,#2\right\rangle}
\newcommand{\supp}{\operatorname{supp}}
\newcommand{\prog}{\operatorname{prog}}
\newcommand{\Ortho}{\operatorname{Ortho}}
\newcommand{\one}{\mathbf{1}}
\newcommand{\gap}{\Delta}

\title{Sharp First-Order Lower Bounds under \(\alpha\)-Polyak-\L{}ojasiewicz Conditions}
\author[1]{Saeed Masiha}
\author[1]{Negar Kiyavash}
\author[2]{Patrick Thiran}
\affil[1]{EPFL School of Management of Technology}
\affil[2]{EPFL Department of Computer and Communications Sciences}
\date{}

\begin{document}
\maketitle

\begin{abstract}
We study the optimal complexity of first-order methods under the \(\alpha\)-P\L{} condition with
\(\alpha\in[1,2)\). The condition requires the suboptimality gap to be bounded by the gradient norm raised to the
power \(\alpha\); the endpoint \(\alpha=2\) recovers the classical Polyak--\L ojasiewicz inequality, \(\alpha=1\)
is the H\"older error-bound regime, and intermediate values \(\alpha\in(1,2)\) arise
near degenerate critical points of analytic objectives. We begin with a
structural impossibility: requiring both global \(L\)-smoothness and global \(\alpha\)-P\L{} on \(\R^d\) forces the
function to be constant for every \(\alpha<2\). To recover a nontrivial class, we keep global smoothness but require the
\(\alpha\)-P\L{} inequality only on the initial sublevel set, the set of points at which the objective is no larger
than its initial value, and which contains every iterate of monotone descent methods.

On this \emph{sublevel-\(\alpha\)-P\L{}} class, we prove sharp first-order minimax lower bounds for every
\(\alpha\in[1,2)\). In the deterministic oracle, any first-order method needs
\(\Omega\bigl(L\tau^{2/\alpha}\varepsilon^{-(2-\alpha)/\alpha}\bigr)\) queries, matched exactly by plain gradient
descent. In the bounded-variance stochastic-gradient oracle, any stochastic first-order method needs $\Omega\left(
L\tau^{2/\alpha}\varepsilon^{-(2-\alpha)/\alpha}
 + L\sigma^2\tau^{4/\alpha}\varepsilon^{-(4-\alpha)/\alpha}
\right)$ queries, matching the corresponding SGD upper
rates under standard trajectory containment \citep{FatkhullinEtesamiHeKiyavash2022}.

At the classical P\L{} endpoint \(\alpha=2\), we further prove that every bounded-variance stochastic
first-order method requires \(\Omega(L\tau^2\sigma^2/\varepsilon)\) oracle calls even when the objective is globally
smooth and satisfies the P\L{} inequality globally. This matches the variance-dependent term in known SGD upper bounds
under global P\L{} geometry~\citep{FatkhullinEtesamiHeKiyavash2022}. For comparison, a \(\mu\)-strongly convex
objective satisfies our P\L{} inequality with \(\tau=(2\mu)^{-1}\). Its sharp variance-dependent complexity is
therefore \(\Theta(\tau\sigma^2/\varepsilon)\), whereas the worst case over the global-P\L{} class is larger by a
factor \(L\tau=L/(2\mu)\).
\end{abstract}
\section{Introduction}

The Polyak--\L ojasiewicz (P\L{}) inequality~\citep{Polyak1963,Lojasiewicz1963} is a geometric
condition that allows one to prove fast convergence of first-order methods in nonconvex problems. For a differentiable objective
$f:\R^d\to\R$ with finite infimum $f_\star:=\inf_{x\in\R^d} f(x)$, it requires
\begin{equation*}
f(x)-f_\star\;\le\;\tau\,\norm{\nabla f(x)}^2
\end{equation*}
for some constant $\tau>0$ and every $x\in\R^d$. In the absence of convexity, this inequality suffices for a
stationary point of $f$ to be a global minimizer and yields a linear convergence rate for gradient descent (GD) methods \citep{Polyak1963,KarimiNutiniSchmidt2016}. We study the \emph{$\alpha$-P\L{}
condition}, an exponent-parametrized generalization of the P\L{} inequality:
\begin{equation}
\label{eq:alphaPL-intro}
f(x)-f_\star\;\le\;\tau\,\norm{\nabla f(x)}^\alpha,\qquad \alpha\in[1,2],
\end{equation}
which recovers the classical P\L{} inequality at $\alpha=2$ and reduces to a H\"older error-bound condition at
$\alpha=1$. By Kurdyka--\L ojasiewicz
theory~\citep{Lojasiewicz1963,Kurdyka1998,BolteDaniilidisLewis2007,AttouchBolteRedontSoubeyran2010}, intermediate
exponents $\alpha\in(1,2)$ arise naturally near degenerate local minima of analytic or \emph{tame}\footnote{Here
``tame'' refers to functions definable in an o-minimal structure; this includes semialgebraic and globally subanalytic
functions. Such classes rule out pathological oscillations and satisfy the Kurdyka--\L ojasiewicz property under mild
regularity assumptions.} objectives. In this case, $\alpha$ measures the local flatness of the objective around the
minimum.

\paragraph{Where intermediate exponents arise.}
The connection with K\L{} theory is local and is stated relative to the relevant critical value. Near a critical point
\(\bar x\), the K\L{} gradient inequality gives constants \(c>0\), \(\theta\in(0,1)\), and a neighborhood
\(\mathcal U\) such that $\norm{\nabla f(x)} \ge c\, |f(x)-f(\bar x)|^\theta,$ for every $x\in\mathcal U$.
When \(\bar x\) is a local minimizer, this inequality can be rewritten on \(\mathcal U\) as $f(x)-f(\bar x) \le c^{-1/\theta}\norm{\nabla f(x)}^{1/\theta}.$
Thus K\L{} geometry near a minimizer gives a local \(\alpha\)-P\L{} inequality with exponent \(\alpha=1/\theta\). The
classical P\L{} case \(\alpha=2\) corresponds to \(\theta=1/2\), as at nondegenerate Morse minima. Flatter degenerate
minima can have \(\theta>1/2\), and hence \(\alpha\in(1,2)\). For example, if the first nonzero Taylor term of
\(f(x)-f(\bar x)\) has order \(p>2\), then the corresponding local exponent is \(\alpha=p/(p-1)\).

This mechanism applies broadly: real-analytic functions and, more generally, tame objectives such as semialgebraic,
subanalytic, or definable functions satisfy K\L{} inequalities locally
\citep{Lojasiewicz1963,Kurdyka1998,BolteDaniilidisLewis2007,AttouchBolteRedontSoubeyran2010,LiPong2018}. These classes
include many standard semialgebraic or definable machine-learning objectives, such as quadratic and logistic
regression, ReLU networks with \(\ell_1/\ell_2\) regularization, matrix factorization, and deep linear
models~\citep{BolteDaniilidisLewis2007,AttouchBolteRedontSoubeyran2010,LiuZhuBelkin2022}. One concrete example,
analyzed by \citet{MasihaSalehkaleybarHeKiyavashThiran2026}, is the squared loss of a depth-\(d\) one-dimensional
linear network, \(f(x)=\bigl((\prod_{i=1}^d x_i)-1\bigr)^2\), which satisfies a local \(\alpha\)-P\L{} inequality with
\(\alpha=2d/(2d-1)\in(1,2)\). 

\paragraph{The minimax question and the missing lower bounds.}
For a target accuracy $\varepsilon$, the complexity question of interest is to determine, in the
worst case, how many first-order oracle calls are required to find an output $x$ with \(f(x)-f_\star\le\varepsilon\). The relevant parameters are the smoothness
constant $L$, the $\alpha$-P\L{} constant $\tau$, the initial gap $\Delta_0:=f(x_0)-f_\star$, and, in the
bounded-variance stochastic-gradient model, the variance bound $\sigma^2$.

At the P\L{} endpoint $\alpha=2$, the deterministic minimax picture is sharp:
$\Theta\left(L\tau\log(\Delta_0/\varepsilon)\right)$ queries are necessary and sufficient for deterministic
first-order methods on smooth P\L{} functions
\citep{KarimiNutiniSchmidt2016,YueFangLin2023}. In the bounded-variance stochastic setting, known SGD analyses under a
global P\L{} condition, or under a guarantee that the stochastic trajectory remains in the P\L{} region, give
\[
O\left(
L\tau\log\frac{\Delta_0}{\varepsilon}
+\frac{L\tau^2\sigma^2}{\varepsilon}
\right)
\]
oracle calls~\citep{FatkhullinEtesamiHeKiyavash2022}. This differs from the classical strongly convex complexity.
Indeed, if $f$ is $L$-smooth and $\mu$-strongly convex and $\kappa:=L/\mu$, optimal stochastic first-order methods
achieve
\[
O\left(
\sqrt{\kappa}\log\frac{\Delta_0}{\varepsilon}
+\frac{\sigma^2}{\mu\varepsilon}
\right),
\]
whereas plain SGD has $\kappa$ rather than $\sqrt{\kappa}$ in the logarithmic term; both have the sharp noise
dependence $\Theta(\sigma^2/(\mu\varepsilon))$
\citep{GhadimiLan2013,NemirovskiJuditskyLanShapiro2009,AgarwalBartlettRavikumarWainwright2012}.
Since a differentiable $\mu$-strongly convex function satisfies
$\frac12\norm{\nabla f(x)}^2\ge\mu\bigl(f(x)-f_\star\bigr)$, the corresponding P\L{} parameter in our notation is
$\tau=(2\mu)^{-1}$. Thus the strongly convex
noise term is $\Theta(\tau\sigma^2/\varepsilon)$, while the general P\L{} upper bound contains the additional factor
$L\tau=\Theta(\kappa)$. We prove that this factor is unavoidable even within the globally smooth,
globally P\L{} class by establishing the variance-dependent lower bound
$\Omega(L\tau^2\sigma^2/\varepsilon)$, equivalently
$\Omega(L\sigma^2/(\mu^2\varepsilon))$ up to numerical constants. The result concerns only the variance-dependent term;
it does not establish the logarithmic exact-gradient contribution for randomized methods.

For $\alpha\in[1,2)$, known upper bounds predict polynomial rather than logarithmic complexity. In the deterministic
case, the standard descent analysis for gradient descent gives
$O\left(L\tau^{2/\alpha}\varepsilon^{-(2-\alpha)/\alpha}\right)$
\citep{AttouchBolteRedontSoubeyran2010,BolteSabachTeboulle2014},
and in the bounded-variance stochastic-gradient setting, existing local or trajectory-contained SGD analyses achieve
the sum
\[
O\left(
L\tau^{2/\alpha}\varepsilon^{-(2-\alpha)/\alpha}
 + L\sigma^2\tau^{4/\alpha}\varepsilon^{-(4-\alpha)/\alpha}
\right)
\]
under $\alpha$-P\L{} or K\L{} assumptions
\citep{FatkhullinEtesamiHeKiyavash2022}. However, the matching
oracle lower bound had remained elusive: it was not known whether the aforementioned dependences are unavoidable for
the whole function class. This paper provides \textit{matching} lower
bounds for both deterministic and bounded-variance stochastic-gradient oracle models.
This requires specifying the function class. The obvious choice, globally smooth and globally $\alpha$-P\L{} on $\mathbb{R}^d$, turns out to be vacuous for $\alpha < 2$, as the next paragraph discusses.


\paragraph{The function class: global smoothness with localized $\alpha$-P\L{}.}
We consider objective functions that are globally $L$-smooth (i.e., differentiable everywhere with globally Lipschitz
gradient), but require the $\alpha$-P\L{} inequality only on the \emph{initial sublevel set}
\[
S_0(f):=\{x\in\R^d:\ f(x)\le f(x_0)\}.
\]
The reason for this choice as opposed to imposing the $\alpha$-P\L{}
inequality~\eqref{eq:alphaPL-intro} globally is that for $\alpha<2$, together with global smoothness, the class of objectives does not admit any
nontrivial functions. More precisely, we will prove
(see Theorem~\ref{thm:global-trivial}) that
if $f:\R^d\to\R$ is continuously differentiable, globally $L$-smooth, bounded below by $f_\star>-\infty$, and
satisfies $f(x)-f_\star\le\tau\,\norm{\nabla f(x)}^\alpha$ at every $x\in\R^d$ for some $\alpha\in[1,2)$ and
$\tau<\infty$, then $f\equiv f_\star$.

Returning to our local condition, note that for any deterministic descent method whose iterates satisfy \(f(x_{k+1})\le f(x_k)\), all iterates remain
inside \(S_0(f)\) when the algorithm is initialized at \(x_0\). Hence the \(\alpha\)-P\L{} inequality only needs to
hold on \(S_0(f)\) for the standard descent analysis. For stochastic gradient methods, existing upper
bounds use the same local viewpoint together with a trajectory-containment condition ensuring that the iterates remain
in the region where the K\L/\(\alpha\)-P\L{} inequality is valid. The resulting class is rich: it contains the classical
P\L{} endpoint $\alpha=2$, smooth convex minimization over a bounded sublevel set as an $\alpha=1$ subclass, and local
K\L\ landscapes from optimization, statistics, and
learning~\citep{KarimiNutiniSchmidt2016,FatkhullinEtesamiHeKiyavash2022,MasihaSalehkaleybarHeKiyavashThiran2026,LiuZhuBelkin2022}.
We refer to this class as the \emph{globally smooth, sublevel-$\alpha$-P\L{} class} (see Definition~\ref{def:function-class}). The question this paper addresses is:
\begin{quote}\itshape
Are the deterministic and bounded-variance stochastic upper rates displayed above minimax unavoidable on the globally
smooth, sublevel-\(\alpha\)-P\L{} class, or can some first-order method improve either dependence uniformly over the
class?
\end{quote}
We prove matching deterministic and
bounded-variance stochastic lower bounds, which settle the above question. Our precise contributions are summarized below.

\paragraph{Contributions.}
Throughout, \(T_\varepsilon\) denotes the minimum number of deterministic first-order oracle calls needed to guarantee
objective error at most \(\varepsilon\) within the entire class, and \(T_\varepsilon^{\mathrm{BV}}\) denotes the minimum
number of bounded-variance stochastic-gradient oracle calls needed to guarantee expected objective error at most
\(\varepsilon\). The implicit constants depend only on \(\alpha\). The paper makes four contributions.

\begin{enumerate}[label={\bf (\arabic*)},leftmargin=2.3em,itemsep=0.5em,topsep=0.3em]

\item \emph{Structural impossibility} (Theorem~\ref{thm:global-trivial}). For every $\alpha\in[1,2)$, the only
$C^1$ functions on $\R^d$ that are simultaneously globally $L$-smooth, bounded below, and globally $\alpha$-P\L{} are
constants. This motivates the sublevel localization adopted in
Definition~\ref{def:function-class}.

\item \emph{Sharp deterministic lower bound} (Theorem~\ref{thm:matching}). Within the class of globally smooth,
sublevel-$\alpha$-P\L{} class with initial gap at most $\Delta_0$, every deterministic first-order method must make
\[
T_\varepsilon
\;\ge\;
\Omega\bigl(L\,\tau^{2/\alpha}\,\varepsilon^{-(2-\alpha)/\alpha}\bigr),
\qquad \alpha\in[1,2),
\] deterministic first-order oracle calls.
This matches the standard gradient-descent upper bound for $\alpha$-P\L{}/K\L\ geometry. The lower bound is proved via a
rescaled version of the zero-chain hard instance in \citet{YueFangLin2023} which simultaneously achieves the target objective gap and preserves the target sublevel
$\alpha$-P\L{} geometry.

\item \emph{Sharp stochastic lower bound} (Theorem~\ref{thm:stochastic-matching}). In the bounded-variance (BV)
stochastic-gradient model, each query returns an unbiased stochastic gradient whose variance is at most
\(\sigma^2\); see Definition~\ref{def:stochastic-model}. We prove the following lower bound on the number of oracle
calls, in the nontrivial accuracy range
\(\varepsilon\lesssim\min\{\Delta_0,L^{\alpha/(2-\alpha)}\tau^{2/(2-\alpha)}\}\):
\[
T_\varepsilon^{\mathrm{BV}}
\;\ge\;
\Omega\Bigl(
L\tau^{2/\alpha}\varepsilon^{-(2-\alpha)/\alpha}
 + L\,\sigma^2\,\tau^{4/\alpha}\,\varepsilon^{-(4-\alpha)/\alpha}
\Bigr),
\qquad \alpha\in[1,2).
\]
The second term is the stochastic-noise contribution; it dominates the deterministic term when
\(\varepsilon\lesssim\tau\sigma^\alpha\). The bound matches the corresponding SGD upper rate of
\citet{FatkhullinEtesamiHeKiyavash2022} whenever the whole SGD trajectory remains in the region where the
$\alpha$-P\L{}/K\L\ inequality is valid. Our construction uses a
probabilistic zero-chain in the style of
\citet{ArjevaniEtAl2019}, in which the next hidden coordinate is revealed with probability tied to the variance
budget; this yields the sharp exponent $(4-\alpha)/\alpha$.

\item
\emph{A sharp variance-dependent lower bound at the P\L{} endpoint}
(Theorem~\ref{thm:stochastic-pl-endpoint}). At the classical endpoint $\alpha=2$, every randomized first-order method
accessing only unbiased gradients with variance at most $\sigma^2$ requires
\[
T_\varepsilon^{\mathrm{BV}}
\;\ge\;
\Omega\left(L\tau^2\sigma^2\varepsilon^{-1}\right)
\]
oracle calls even on the globally smooth, globally P\L{} class. This matches the variance-dependent term in known SGD
upper bounds under global P\L{} geometry~\citep{FatkhullinEtesamiHeKiyavash2022}. In contrast, smooth strongly convex
optimization has sharp variance-dependent complexity $\Theta(\tau\sigma^2/\varepsilon)$, so the global-P\L{} class
can be harder by a factor $L\tau$. The proof uses the same sequential coordinate-hiding idea as for $\alpha<2$, but
rescales the hard instance differently at $\alpha=2$.

\end{enumerate}
\noindent
Table~\ref{tab:complexity-summary} summarizes the lower bounds proved in this paper as well as the corresponding known upper
bounds.

\begin{table}[t]
\centering
\small
\begin{tabular}{@{}
>{\raggedright\arraybackslash}p{0.25\textwidth}
>{\raggedright\arraybackslash}p{0.32\textwidth}
>{\raggedright\arraybackslash}p{0.33\textwidth}
@{}}
\toprule
Model & Lower bound & Upper bound \\
\midrule
Deterministic first-order, sublevel-$\alpha$-P\L{}, $\alpha\in[1,2)$
&
{\boldmath$\Omega\bigl(L\tau^{2/\alpha}\varepsilon^{-(2-\alpha)/\alpha}\bigr)$}
\par\smallskip
{\small This paper, Thm.~\ref{thm:matching}}
&
$O\bigl(L\tau^{2/\alpha}\varepsilon^{-(2-\alpha)/\alpha}\bigr)$
\par\smallskip
{\small Gradient descent; \citet{AttouchBolteRedontSoubeyran2010}}
\\
\addlinespace[0.5em]
Stochastic-gradient, BV oracle, sublevel-$\alpha$-P\L{}, $\alpha\in[1,2)$
&
{\boldmath$\Omega\bigl(L\tau^{2/\alpha}\varepsilon^{-(2-\alpha)/\alpha}
 + L\sigma^2\tau^{4/\alpha}\varepsilon^{-(4-\alpha)/\alpha}\bigr)$}
\par\smallskip
{\small This paper, Thm.~\ref{thm:stochastic-matching}}
&
$O\bigl(L\tau^{2/\alpha}\varepsilon^{-(2-\alpha)/\alpha}
 + L\sigma^2\tau^{4/\alpha}\varepsilon^{-(4-\alpha)/\alpha}\bigr)$
\par\smallskip
{\small SGD under trajectory containment; \citet{FatkhullinEtesamiHeKiyavash2022}}
\\
\addlinespace[0.5em]
Stochastic-gradient, BV oracle, global P\L{}, $\alpha=2$
&
\boldmath$\Omega\bigl(L\tau^2\sigma^2/\varepsilon\bigr)$
\par\smallskip
\small Noise contribution only; this paper, Thm.~\ref{thm:stochastic-pl-endpoint}
&
$O\bigl(L\tau\log(\Delta_0/\varepsilon)+L\tau^2\sigma^2/\varepsilon\bigr)$
\par\smallskip
\small SGD under global P\L{}; \citet{FatkhullinEtesamiHeKiyavash2022}
\\
\bottomrule
\end{tabular}
\caption{Lower and upper rates on the globally smooth, sublevel-$\alpha$-P\L{} class. The first two
rows concern $\alpha\in[1,2)$ and are stated for
\(\varepsilon\lesssim\min\{\Delta_0,L^{\alpha/(2-\alpha)}\tau^{2/(2-\alpha)}\}\). In the second row, the stochastic
term dominates when \(\varepsilon\lesssim\tau\sigma^\alpha\). The last row concerns the globally
smooth, globally P\L{} subclass and records only the variance-dependent lower bound at $\alpha=2$; the paper does not
prove the full randomized endpoint complexity.}
\label{tab:complexity-summary}
\end{table}


\subsection{Related work}
\label{sec:related-work}

Our results sit at the intersection of three threads: the geometry of P\L{} and Kurdyka--\L ojasiewicz (K\L)
landscapes, sharp first-order complexity theory, and stochastic optimization for structured nonconvex problems. We
organize the comparison around five themes.

\paragraph{P\L{}, K\L, gradient dominance, and applications.}
The P\L{} inequality was introduced by \citet{Polyak1963} as a simple condition under which gradient descent converges
linearly even without convexity. It is one member of a larger family of \emph{gradient-dominance} conditions, which are  inequalities that control the objective gap \(f(x)-f_\star\) by a power of the gradient norm. The case
\(\alpha=2\) in~\eqref{eq:alphaPL-intro} is the classical P\L{} condition, while \(\alpha<2\) corresponds to weaker
H\"older-type gradient dominance. This same structure appears in the \L ojasiewicz and K\L\ gradient inequalities. Indeed, a K\L\ inequality
with exponent \(\theta\in[0,1)\) has the local form
\(\norm{\nabla f(x)}\ge c(f(x)-f_\star)^\theta\); when \(\theta=1/\alpha\), it is equivalent, up to constants, to the
local \(\alpha\)-P\L{} inequality used in this paper. The analytic foundations go back to
\citet{Lojasiewicz1963} and \citet{Kurdyka1998}; nonsmooth and algorithmic consequences were developed by
\citet{BolteDaniilidisLewis2007}, \citet{AttouchBolteRedontSoubeyran2010}, and
\citet{BolteSabachTeboulle2014}; convergence-rate refinements appear in
\citet{FrankelGarrigosPeypouquet2015}; and the calculus of K\L\ exponents under sums, restrictions, and compositions
was studied by \citet{LiPong2018}. The relationships between P\L{}, error bounds, quadratic growth, and related
gradient-dominance conditions are also central in \citet{KarimiNutiniSchmidt2016} and
\citet{DrusvyatskyLewis2018}.  Adapting a local viewpoint akin to this paper, \citet{RebjockBoumal2024} show that, for
\(C^2\) functions near non-isolated local minima, local P\L{}, quadratic growth, and error-bound conditions are
equivalent, and are also equivalent to the Morse--Bott property.
Complementing these local equivalence results, \citet{BoumalCriscitielloRebjock2026} study the global endpoint
\(\alpha=2\): they show that \(C^\infty\) globally P\L{} functions on contractible manifolds (such as \(\mathbb{R}^d\)) have a rigid nonlinear
least-squares structure. Our Theorem~\ref{thm:global-trivial} shows that, below this endpoint, the globally smooth and globally \(\alpha\)-P\L{}
model collapses to constant functions.
Related exponent-threshold obstructions were studied by \citet{AbbaszadehpeivastiDeKlerkZamani2023} and
\citet[Remark~2.21]{RebjockBoumal2024}: in the \L{}ojasiewicz convention, both identify restrictions on nonconstant
smooth functions satisfying inequalities with exponent \(\theta<1/2\), which corresponds to \(\alpha>2\) in our
notation.

Local P\L{}/K\L{}-type assumptions are important because they hold far beyond strongly convex objectives. Analytic, semialgebraic,
subanalytic, and more generally definable functions satisfy K\L{}-type inequalities locally around critical
points~\citep{Lojasiewicz1963,Kurdyka1998,BolteDaniilidisLewis2007}. This includes many optimization models used in
statistics and machine learning, such as regularized empirical-risk objectives, matrix factorization, deep linear
networks, and broad classes of neural-network losses with tame activation and regularization structure; see, for
example, \citet{LiuZhuBelkin2022} and the references above. In overparameterized interpolation regimes, related
gradient-dominance and interpolation conditions also explain fast behavior of stochastic gradient methods
\citep{BassilyBelkinMa2018,VaswaniBachSchmidt2019}. 

\paragraph{Deterministic first-order complexity under convexity and P\L{}.}
The classical first-order lower-bound theory for convex optimization originated
with~\citet{NemirovskyYudin1983}; see~\citet{Nesterov2004,Bubeck2015} for modern textbook accounts. 
Sharp rates on the smooth convex subclasses are known: for smooth convex minimization over a bounded domain of radius \(D\), accelerated gradient methods achieve the optimal complexity
\(\Theta(\sqrt{LD^2/\varepsilon})\), and under strong convexity with parameter \(\mu\) achieve
\(\Theta(\sqrt{L/\mu}\log(\Delta_0/\varepsilon))\). These rates are useful reference points for the endpoints \(\alpha=1\) and \(\alpha=2\), but they are not the minimax rates for the sublevel-\(\alpha\)-P\L{} class.


Under P\L{} and K\L{} inequalities, the deterministic upper bounds in the literature are achieved by descent methods
rather than accelerated ones. For the P\L{} case, \citet{Polyak1963} showed that gradient descent converges
linearly; \citet{KarimiNutiniSchmidt2016} unified P\L{} with closely related error-bound and gradient-dominance
conditions. Convergence rates under K\L{} inequalities with arbitrary desingularization exponent, for descent,
proximal, and splitting methods, were obtained by
\citet{AttouchBolteRedontSoubeyran2010,BolteSabachTeboulle2014,FrankelGarrigosPeypouquet2015}. In local versions, one assumes that the iterates remain in a neighborhood where the K\L{} inequality is valid; this is the ``stay-in-neighborhood'' condition.
The elementary descent recursion yields the gradient-descent upper
bound
$
O\left(L\tau^{2/\alpha}\varepsilon^{-(2-\alpha)/\alpha}\right)
$
for the whole regime \(\alpha\in[1,2)\).

The lower bound is only known at the P\L{} endpoint. For smooth P\L{} functions
(\(\alpha=2\)), gradient descent has complexity \(\Theta(L\tau\log(\Delta_0/\varepsilon))\), and
\citet{YueFangLin2023} proved that this dependence is unimprovable by deterministic first-order methods. For $\alpha < 2$ no matching lower bound was known. Our Theorem~3.1 establishes the matching lower bound $\Omega\bigl(L\tau^{2/\alpha}\varepsilon^{-(2-\alpha)/\alpha}\bigr)$ on the sublevel-$\alpha$-P\L{} class for every $\alpha \in [1,2)$. Gradient descent's dependence on $L$, $\tau$, and $\varepsilon$ is therefore intrinsic to the class, not an artifact of the algorithm.

\paragraph{Stochastic first-order complexity under P\L{} and K\L.}
Classical lower bounds for stochastic convex optimization were obtained by
\citet{NemirovskiJuditskyLanShapiro2009} and \citet{AgarwalBartlettRavikumarWainwright2012}. Under
$\mu$-strong convexity, with $\tau=(2\mu)^{-1}$ in our convention, bounded-variance stochastic optimization has the
sharp noise dependence $\Theta(\tau\sigma^2/\varepsilon)$. Because strongly convex functions form
only a subclass of the globally smooth, globally P\L{} class, this lower bound does not determine whether general
P\L{} objectives incur an additional condition-number factor. At the other endpoint, smooth convex minimization over a bounded sublevel set of radius $D$
has the sharp stochastic dependence $\Theta(D^2\sigma^2/\varepsilon^2)$.

Most relevant to our setting, \citet{FatkhullinEtesamiHeKiyavash2022} obtained SGD-type bounded-variance upper bounds
under global \(\alpha\)-P\L{} inequalities. In the present notation, the rate contains the deterministic descent term $O\bigl(L\tau^{2/\alpha}\varepsilon^{-(2-\alpha)/\alpha}\bigr)$
and the stochastic noise term $O\bigl(L\sigma^2\tau^{4/\alpha}\varepsilon^{-(4-\alpha)/\alpha}\bigr),$
the latter being the dominant contribution when \(\varepsilon\lesssim\tau\sigma^\alpha\).
\citet{KhaledRichtarik2023} gave a general analysis of SGD under P\L{}-type
conditions with bounded variance. \citet{MasihaSalehkaleybarHeKiyavashThiran2026} established matching upper and
lower bounds $\Theta(\varepsilon^{-2/\alpha})$ in the \emph{local} $\alpha$-P\L{} setting via a SARAH-type variance-reduced scheme with a stay-in-neighborhood guarantee, under a \emph{batch-smooth} stochastic first-order
oracle (mean-square smooth across query points, in addition to bounded variance); a stronger oracle model than the gradient-only bounded-variance oracle used in this paper.
At $\alpha=2$, the variance-dependent term in the SGD bound of
\citet{FatkhullinEtesamiHeKiyavash2022} is $L\tau^2\sigma^2/\varepsilon$. In a related distributed stochastic setting,
\citet{FarhadkhaniEtAl2026} note that lower bounds inherited from strongly convex quadratics do not capture this P\L{}
condition-number dependence and conjecture that the additional dependence is intrinsic.
Theorem~\ref{thm:stochastic-pl-endpoint} proves the matching variance-dependent lower bound within the globally smooth,
globally P\L{} class. Thus the condition-number factor is not caused by localizing the P\L{} inequality to an initial
sublevel set.
Our contribution complements the aforementioned results: we establish the sharp \emph{lower} bound
\[
\Omega\left(
L\tau^{2/\alpha}\varepsilon^{-(2-\alpha)/\alpha}
 + L\sigma^2\tau^{4/\alpha}\varepsilon^{-(4-\alpha)/\alpha}
\right)
\]
on the globally smooth, sublevel-$\alpha$-P\L{} class in the gradient-only bounded-variance oracle model. The
dependence on \(L\), \(\tau\), \(\sigma^2\), and \(\varepsilon\) matches the SGD upper rate of
\citet{FatkhullinEtesamiHeKiyavash2022}.

\paragraph{Lower-bound techniques.}
Our deterministic lower bound uses a zero-chain hard instance, a standard device in oracle complexity going back to
\citet{NemirovskyYudin1983} and \citet{Nesterov2004}. 
A zero chain is designed so that coordinate \(i+1\) is invisible
until coordinate \(i\) has already been activated. Starting from the origin, the gradient can therefore expose at most
one new coordinate direction per oracle call: first coordinate \(1\), then coordinate \(2\), and so on.
A method is called \emph{zero-respecting} if it can put nonzero mass only on coordinates
that have already appeared in previous oracle responses. On a zero chain, such a method can activate at most one new
coordinate per query. The P\L{} lower bound of \citet{YueFangLin2023} uses this idea at the endpoint \(\alpha=2\).
To extend the zero-respecting lower bound to arbitrary deterministic first-order methods, we invoke the
resisting-oracle reduction of \citet[Proposition~1]{CarmonDuchiHinderSidford2021}. For any fixed deterministic method
and query budget, the reduction constructs an orthogonal embedding of the zero chain such that every unrevealed chain
direction is orthogonal to the span of the preceding queries. Consequently, the method can reveal at most one new chain
coordinate per oracle call, just as in the zero-respecting model.

The stochastic lower bound requires an additional mechanism. In a bounded-variance stochastic-gradient oracle, the
method sees noisy unbiased gradient samples, so a deterministic zero-chain obstruction is not enough: the oracle could,
in principle, leak information about hidden coordinates through rare noisy samples. We use the probabilistic zero-chain
framework and random-rotation progress lemma of \citet{ArjevaniEtAl2019}, originally developed for bounded-variance
stochastic stationary-point lower bounds. In this construction the next hidden coordinate is revealed only on a
Bernoulli event of probability \(q\). Smaller \(q\) makes the next hidden coordinate harder to discover, but
increases the variance of the unbiased estimator. We choose \(q\) to exactly meet the variance budget.

The new technical ingredients are specific to sublevel-\(\alpha\)-P\L{} complexity. In the deterministic construction,
the zero chain is rescaled by an amplitude \(a_T\). If \(a_T\) is too small, the function satisfies the desired
\(\alpha\)-P\L{} inequality but the unrevealed coordinates left after \(T\) queries contribute too little objective gap to
give a sharp lower bound. If \(a_T\) is too large, those unrevealed coordinates contribute enough objective gap, but the
function no longer satisfies the sublevel-\(\alpha\)-P\L{} inequality. The proof chooses the largest admissible amplitude,
so the objective gap left in the unrevealed coordinates and the \(\alpha\)-P\L{} geometry are balanced. This balance is what produces the deterministic rate
\(T^{-\alpha/(2-\alpha)}\).

The stochastic construction adds a second balance. After the amplitude \(a_m\) is fixed, the oracle reveals the next
hidden coordinate only with probability \(q\). This probability cannot be chosen independently of \(a_m\): thinning a
gradient coordinate of size \(a_m\) by a Bernoulli factor of probability \(q\) creates variance of order \(a_m^2/q\).
Thus the bounded-variance constraint forces \(q\) to be proportional to \(a_m^2\), up to the variance scale. 
Since the deterministic scaling gives
\(a_m^\alpha\asymp\varepsilon\), the reveal probability has the accuracy dependence
\(q\propto\varepsilon^{2/\alpha}\). The method needs about \(m\) successful reveals, each occurring with probability
\(q\), so the sample complexity scales with \(m/q\). Combining this reveal-time calculation with
\(m\asymp\varepsilon^{-(2-\alpha)/\alpha}\) yields the stochastic noise exponent \((4-\alpha)/\alpha\). This is the
additional noise-dominated contribution, relevant when \(\varepsilon\lesssim\tau\sigma^\alpha\); the full stochastic
lower bound also contains the deterministic term.

\paragraph{Notation.}
Let \(\norm{\cdot}\) and \(\ip{\cdot}{\cdot}\) denote the Euclidean norm and inner product, respectively. Write \(\one_m\) for
the all-ones vector in \(\R^m\), and write \(\mathbf 1\{\mathcal E\}\) for the indicator of an event \(\mathcal E\).
For a set \(S\), let \(\operatorname{dist}(x,S):=\inf_{y\in S}\norm{x-y}\). For a differentiable map \(\phi\),
\(J\phi(x)\) denotes its Jacobian at \(x\). For \(x\in\R^m\) and \(\delta\ge0\), define the progress index
\[
\prog_\delta(x):=\max\bigl(\{i\in\{1,\ldots,m\}:\ |x_i|>\delta\}\cup\{0\}\bigr).
\]
For a function \(f:\R^d\to\R\), define
\[
f_\star:=\inf_{x\in\R^d}f(x),
\qquad
X_\star:=\arg\min f
\]
when the minimum is attained. The initial point is fixed to be \(x_0=0\), the initial gap is bounded by
\(\Delta_0\ge f(0)-f_\star\), and
\[
S_0(f):=\{x\in\R^d:\ f(x)\le f(0)\}
\]
denotes the initial sublevel set. For an integer \(k\ge 1\), \(C^k\) denotes the class of functions with \(k\)
continuous derivatives, and \(C^\infty\) denotes the class of smooth functions. The parameters \(L\), \(\tau\),
\(\alpha\), \(\sigma^2\), and \(\varepsilon\) denote the smoothness constant, the sublevel-\(\alpha\)-P\L{} constant, the P\L{} exponent, the stochastic-gradient
variance bound, and the target accuracy, 
respectively.  \(T_\varepsilon\) and \(T_\varepsilon^{\mathrm{BV}}\) denote deterministic first-order and
bounded-variance stochastic-gradient oracle complexities, respectively. Expectations and
probabilities are denoted by \(\E\) and \(\Pbb\). Unless explicitly stated otherwise, \(A=O(B)\) means
\(A\le cB\), \(A=\Omega(B)\) means \(A\ge cB\), and \(A=\Theta(B)\) means both inequalities hold, for positive
constants \(c\) independent of the problem parameters under discussion. We also write \(A\lesssim B\) for
\(A\le cB\) and \(A\asymp B\) for \(A\lesssim B\lesssim A\).

\paragraph{Organization.}
Section~\ref{sec:preliminaries} formalizes the oracle models, proves the structural impossibility result that
globally smooth, globally $\alpha$-P\L{} functions on $\R^d$ are trivial for every $\alpha<2$
(Theorem~\ref{thm:global-trivial}), and defines the globally smooth, sublevel-$\alpha$-P\L{} class.
Section~\ref{sec:det} states the deterministic lower bound and recalls the matching gradient-descent upper bound.
Section~\ref{sec:sto} states the bounded-variance stochastic-gradient lower bound for
$\alpha\in[1,2)$ and a separate variance-dependent lower bound under global P\L{} geometry at the classical endpoint.
Section~\ref{sec:conclusion}
concludes. Detailed proofs appear in Appendix~\ref{app:deterministic} (deterministic) and
Appendix~\ref{app:stochastic} (stochastic).

\section{Preliminaries}
\label{sec:preliminaries}

We first prove a structural impossibility result: the
globally smooth, globally $\alpha$-P\L{} model is trivial for every $\alpha<2$. This motivates the function class
studied throughout the paper, namely globally smooth functions that satisfy the
$\alpha$-P\L{} inequality only on the initial sublevel set $S_0(f)$. Subsequently, we introduce the deterministic and
bounded-variance stochastic-gradient oracle models and their associated minimax complexities. 

\subsection{\texorpdfstring{Triviality of globally smooth and $\alpha$-P\L{} function class for $\alpha<2$}{Global smoothness plus global alpha-PL is trivial for alpha<2}}

The following theorem explains why we cannot simply assume both global smoothness and a global \(\alpha\)-P\L{}
inequality when \(\alpha<2\). 

\begin{theorem}[Global $\alpha$-P\L{} plus global smoothness yields a constant function]
\label{thm:global-trivial}
Let $f:\R^d\to\R$ be continuously differentiable and globally $L$-smooth, i.e.,
\[
\norm{\nabla f(x)-\nabla f(y)}\le L\norm{x-y}
\qquad (\forall x,y\in\R^d).
\]
Assume that $f$ has a finite infimum $f_\star:=\inf_{x\in\R^d}f(x)>-\infty,$ and that there exist $\tau>0$ and $\alpha\in[1,2)$ such that
\[
f(x)-f_\star\le \tau \norm{\nabla f(x)}^\alpha
\qquad (\forall x\in\R^d).
\]
Then $f$ is constant on $\R^d$.
\end{theorem}

\begin{proof}
Define $u(x):=f(x)-f_\star$. Note that $u\in C^1$, $u(x)\ge 0$, and $\nabla u(x)=\nabla f(x)$, hence
 $\nabla u$ is globally $L$-Lipschitz:
\[
\norm{\nabla u(x)-\nabla u(y)}\le L\norm{x-y}
\qquad (\forall x,y\in\R^d)
\]
and satisfies
\[
u(x)\le \tau \norm{\nabla u(x)}^\alpha
\qquad (\forall x\in\R^d).
\]

Fix any $x_0\in\R^d$ and consider the gradient-flow ODE
\[
\dot x(t)=\nabla u(x(t)),
\qquad
x(0)=x_0.
\]
Because $\nabla u$ is globally Lipschitz, this ODE has a unique global solution for all $t\ge 0$ \citep[see, e.g., Chapter~1]{CoddingtonLevinson1955}.

Set
\[
v(t):=u(x(t)).
\]
From the chain rule:
\[
v'(t)
=
\ip{\nabla u(x(t))}{\dot x(t)}
=
\norm{\nabla u(x(t))}^2
\ge 0.
\]
The global $\alpha$-P\L{} inequality along the trajectory implies
\[
\norm{\nabla u(x(t))}\ge \left(\frac{v(t)}{\tau}\right)^{1/\alpha},
\]
and therefore
\begin{equation}
\label{eq:ascent-superlinear-gap}
v'(t)\ge \tau^{-2/\alpha}v(t)^{2/\alpha}.
\end{equation}
Assume for contradiction that $v(0)>0$.
Since $v'(t)\ge 0$, we have $v(t)\ge v(0)>0$ for all $t\ge 0$.
Let $r:={2}/{\alpha}>1$.
Since $1-r<0$, multiplying both sides of $v'(t)\ge \tau^{-2/\alpha}v(t)^r$ by $(1-r)v(t)^{-r}$ reverses the inequality and yields
\[
\frac{d}{dt}\bigl(v(t)^{1-r}\bigr)
=(1-r)v(t)^{-r}v'(t)
\le
(1-r)\tau^{-2/\alpha}.
\]
Integrating from $0$ to $t$ yields
\[
v(t)^{1-r}
\le
v(0)^{1-r}-(r-1)\tau^{-2/\alpha}t.
\]
The right-hand side becomes negative for sufficiently large $t$, which is impossible because $v(t)>0$ and $1-r<0$ imply $v(t)^{1-r}>0$.
Therefore $v(0)=0$, so $u(x_0)=0$.
Since $x_0$ was arbitrary, $u\equiv 0$ on $\R^d$.
Hence $f\equiv f_\star$.
\end{proof}

\begin{remark}[Relation to local \L{}ojasiewicz obstructions]
\label{rem:rebjock-boumal-comparison}
Theorem~\ref{thm:global-trivial} is complementary to the local exponent obstructions discussed by
\citet[Remark~2.21]{RebjockBoumal2024}. Their paper uses the \L{}ojasiewicz convention
\[
|f(x)-f_S|^{2\theta}\le \frac{1}{2\mu}\norm{\nabla f(x)}^2
\]
near a local minimizer set \(S\), where \(f_S\) is the common local minimum value. In our notation this corresponds to
\(\alpha=1/\theta\), so the classical P\L{} case is \(\theta=1/2\), equivalently \(\alpha=2\).

Under locally Lipschitz gradients, \citet[Remark~2.21]{RebjockBoumal2024} rule out nonconstant local examples when
\(\theta<1/2\), equivalently \(\alpha>2\). A related sublevel-set obstruction for smooth functions with bounded
curvature was proved by \citet[Theorem~4]{AbbaszadehpeivastiDeKlerkZamani2023}. The mechanism in
\citet{RebjockBoumal2024} is local: the negative gradient flow and the
\L{}ojasiewicz inequality give a finite-length argument, which yields a growth lower bound of the form
\[
f(x)-f_S \gtrsim \operatorname{dist}(x,S)^{1/(1-\theta)}.
\]
For \(\theta<1/2\), this growth is incompatible with the quadratic upper growth implied by locally Lipschitz
gradients, unless the function is locally constant. Theorem~\ref{thm:global-trivial} concerns the opposite side of the P\L{} threshold: \(\alpha<2\), equivalently
\(\theta>1/2\). For \(\alpha<2\), there is no analogous local contradiction; for instance,
\(f(x)=|x|^p/p\) near \(0\), with \(p=\alpha/(\alpha-1)>2\), is a nonconstant \(C^1\) function with locally Lipschitz gradient satisfying a local
\(\alpha\)-P\L{} inequality. The obstruction in Theorem~\ref{thm:global-trivial} is global. We follow the \emph{positive}
gradient flow, along which the global \(\alpha\)-P\L{} inequality forces the objective gap to satisfy the superlinear
differential inequality~\eqref{eq:ascent-superlinear-gap}. Global smoothness prevents finite-time escape, and this gives the contradiction.
\end{remark}

\subsection{\texorpdfstring{The globally smooth, sublevel-$\alpha$-P\L{} class}{The globally smooth, sublevel-alpha-PL class}}

We keep
global smoothness as the ambient oracle regularity condition: every query is a query to a globally \(L\)-smooth
objective. We localize only the $\alpha$-P\L{} inequality, requiring it on the initial sublevel set from the prescribed
initial point \(0\).

\begin{definition}[Globally smooth, sublevel-$\alpha$-P\L{} class]
\label{def:function-class}
For each dimension $d\ge 1$ and parameters $L,\tau,\Delta_0>0$, define
\[
\mathcal F^{(d)}_{\alpha,L,\tau,\Delta_0}(0)
:=
\left\{
f:\R^d\to\R \ \middle|\
\begin{array}{l}
f \text{ is differentiable},\ f_\star>-\infty,\ f(0)-f_\star\le \Delta_0,\\[0.3em]
\nabla f \text{ is globally } L\text{-Lipschitz on } \R^d,\\[0.3em]
f(x)-f_\star\le \tau \norm{\nabla f(x)}^\alpha \quad (\forall x\in S_0(f))
\end{array}
\right\},
\]
where
\[
S_0(f):=\{x\in\R^d:\ f(x)\le f(0)\}
\]
is the initial sublevel set.
We also define the dimension-free union
\[
\mathcal F_{\alpha,L,\tau,\Delta_0}(0)
:=
\bigcup_{d\ge 1}\mathcal F^{(d)}_{\alpha,L,\tau,\Delta_0}(0).
\]
\end{definition}


Before providing a few examples, let us discuss two structural consequences of Definition~\ref{def:function-class}.

\paragraph{Implicit basin condition.}
The definition uses \(f_\star=\inf f\), the \emph{global} infimum, and requires that the sublevel inequality to hold at every
\(x\in S_0(f)\). In particular it must hold at every stationary point of \(f\) inside \(S_0(f)\). At any such
stationary point \(x_c\), the right-hand side \(\tau\norm{\nabla f(x_c)}^\alpha\) is zero, so the inequality forces
\(f(x_c)=f_\star\). Therefore, for any member of \(\mathcal F_{\alpha,L,\tau,\Delta_0}(0)\) 
\begin{equation}
\label{eq:basin-condition}
\text{every stationary point of \(f\) inside \(S_0(f)\) attains the global infimum \(f_\star\).}
\end{equation}
Thus, within \(S_0(f)\), stationarity certifies global optimality. This is the precise basin-type content of the
assumption: the initialized sublevel set may be nonconvex, but it cannot contain a stationary point with objective value
strictly above \(f_\star\).


\paragraph{Connection to local convergence.}
Definition~\ref{def:function-class} is stated using the global infimum \(f_\star\), because this gives a clean
worst-case oracle-complexity model. Many applications, however, use P\L{}/K\L{} geometry only near a target basin of
local minimizers. In that setting, one can fix a connected component \(\mathcal M\) of local minimizers, writes
\(\ell(\mathcal M)\) for its objective value, and measures error by the local gap \(f(x)-\ell(\mathcal M)\). If the
\(\alpha\)-P\L{} inequality holds in a neighborhood of \(\mathcal M\) and the iterates remain in that neighborhood, then
the same \(\alpha\)-dependent exponents govern convergence to \(\mathcal M\). This is the local-convergence viewpoint
studied by \citet{MasihaSalehkaleybarHeKiyavashThiran2026}.

\begin{example}[The endpoint \(\alpha=1\): smooth convex functions on bounded sublevel sets]
\label{ex:alpha1-convex}
Smooth convex problems with a bounded initial sublevel set satisfy the sublevel \(\alpha\)-P\L{} condition with
\(\alpha=1\). Indeed, assume
\(f\) is convex and \(L\)-smooth, \(f_\star\) is attained, and
\[
D_0:=\sup_{x\in S_0(f)} \operatorname{dist}(x,X_\star)<\infty,
\qquad
X_\star:=\arg\min f .
\]
Then, for any \(x\in S_0(f)\) and any projection \(x_\star\in X_\star\),
\[
f(x)-f_\star
\le
\ip{\nabla f(x)}{x-x_\star}
\le
D_0\norm{\nabla f(x)}.
\]
Thus \(f\in\mathcal F_{1,L,D_0,\Delta_0}(0)\) whenever \(f(0)-f_\star\le \Delta_0\). This shows that bounded-sublevel
smooth convex minimization is contained in our \(\alpha=1\) model. The containment is strict: Appendix~\ref{app:nonconvex-alpha1}
gives a globally smooth nonconvex function satisfying the sublevel \(\alpha=1\) condition.
\end{example}

\begin{example}[Intermediate exponents from local K\L{} geometry]
\label{ex:intermediate-kl}
Local K\L{} theory provides a standard source of intermediate \(\alpha\)-P\L{} exponents. Around critical points,
real-analytic functions and, more generally, semialgebraic, subanalytic, or definable objectives satisfy local K\L{}
gradient inequalities
\citep{Lojasiewicz1963,Kurdyka1998,BolteDaniilidisLewis2007,AttouchBolteRedontSoubeyran2010,LiPong2018}. Assume
\(f:\R^d\to\R\) is subanalytic and globally \(L\)-smooth, and fix a compact connected component
\(\mathcal M^\star\) of global minimizers of \(f\). The pointwise power-form K\L{} inequalities and a finite-cover
argument on the compact set \(\mathcal M^\star\), whose points have the common critical value \(f_\star\), give a
neighborhood \(\mathcal U_0\) of \(\mathcal M^\star\) and common constants \(c_0>0\) and
\(\theta_0\in[0,1)\) such that
\[
\norm{\nabla f(x)}
\ge
c_0\bigl(f(x)-f_\star\bigr)^{\theta_0}
\qquad
\bigl(x\in\mathcal U_0,\ f_\star<f(x)\bigr);
\]
see, for example, the compact uniformization principle in \citet[Lemma~6]{BolteSabachTeboulle2014}. The exponent
\(\theta_0\) need not exceed \(1/2\), whereas the intermediate range \(\alpha\in(1,2)\) corresponds, through
\(\alpha=1/\theta\), to \(\theta\in(1/2,1)\). To obtain an exponent in this range, define the smaller neighborhood
\[
\mathcal U:=\mathcal U_0\cap\{x:\ f(x)-f_\star<1\}.
\]
Because \(f=f_\star\) on \(\mathcal M^\star\), continuity of \(f\) ensures that \(\mathcal U\) is still a neighborhood
of \(\mathcal M^\star\). Choose \(\theta\in(1/2,1)\) such that \(\theta\ge\theta_0\); if
\(\theta_0>1/2\), one may simply take \(\theta=\theta_0\). For \(x\in\mathcal U\) with \(f(x)>f_\star\), let
\(s:=f(x)-f_\star\in(0,1)\). Then \(s^\theta\le s^{\theta_0}\), and, since
\(\mathcal U\subseteq\mathcal U_0\), the preceding K\L{} inequality implies
\[
\norm{\nabla f(x)}\ge c_0\bigl(f(x)-f_\star\bigr)^\theta
\qquad
\bigl(x\in\mathcal U,\ f_\star<f(x)\bigr).
\]
Consequently,
\[
f(x)-f_\star\le c_0^{-1/\theta}\norm{\nabla f(x)}^{1/\theta},
\]
so the local \(\alpha\)-P\L{} inequality holds with $\alpha=1/\theta\in(1,2)$ and
$\tau=c_0^{-1/\theta}$. Therefore, if the initial sublevel set satisfies \(S_0(f)\subset\mathcal U\), then
\(f\in\mathcal F_{\alpha,L,\tau,\Delta_0}(0)\) for every \(\Delta_0\ge f(0)-f_\star\). Thus the sublevel
\(\alpha\)-P\L{} class includes functions whose entire initial sublevel set lies inside a local
K\L{}/\(\alpha\)-P\L{} basin around a connected component of global minimizers.
Appendix~\ref{app:explicit-intermediate} provides concrete examples in
$\mathcal F_{\alpha,L,\tau,\Delta_0}(0)$ for every \(\alpha\in(1,2)\).
\end{example}

\subsection{Deterministic oracle complexity}

\begin{definition}[Deterministic first-order method and minimax error]
\label{def:minimax}
A deterministic first-order method is a family \(A=(A_d)_{d\ge 1}\), one rule for each ambient dimension. On input
\(f:\R^d\to\R\), the method starts from \(x_0=0\), and each query \(x_t\) is a deterministic function of the previous
oracle transcript
\[
\bigl\{x_s,f(x_s),\nabla f(x_s):0\le s<t\bigr\}.
\]
For $f\in \mathcal F^{(d)}_{\alpha,L,\tau,\Delta_0}(0)$, we write
\[
x_0^A(f)=0,\qquad x_1^A(f),x_2^A(f),\ldots\in\R^d
\]
for the iterate sequence generated by $A_d$ on $f$.
The fixed-budget minimax error is
\[
\mathcal E_T\!\left(\mathcal F_{\alpha,L,\tau,\Delta_0}(0)\right)
:=
\inf_A\sup_{d\ge 1}\sup_{f\in\mathcal F^{(d)}_{\alpha,L,\tau,\Delta_0}(0)}
\bigl(f(x_T^A(f))-f_\star\bigr).
\]
The corresponding deterministic oracle complexity is
\[
T_\varepsilon\!\left(\mathcal F_{\alpha,L,\tau,\Delta_0}(0)\right)
:=
\inf\Bigl\{T\ge 0:\
\mathcal E_T\!\left(\mathcal F_{\alpha,L,\tau,\Delta_0}(0)\right)\le \varepsilon
\Bigr\}.
\]
When the class is clear from context, we write this quantity simply as \(T_\varepsilon\).
\end{definition}

\subsection{Bounded-variance stochastic-gradient oracle complexity}

\begin{definition}[Bounded-variance stochastic-gradient oracle]
\label{def:bv-oracle}
For a differentiable $f:\R^d\to\R$ and $\sigma^2\ge 0$, let $\mathfrak O_{\sigma^2}(f)$ be the set of stochastic
gradient oracles \(G\) such that each oracle query at a point \(x\) returns a fresh sample \(G(x,\xi)\) satisfying
\[
\E[G(x,\xi)\mid x]=\nabla f(x),
\qquad
\E\bigl[\norm{G(x,\xi)-\nabla f(x)}^2\mid x\bigr]\le \sigma^2
\qquad (\forall x\in\R^d).
\]
This is a gradient-only stochastic oracle: the algorithm receives samples of \(G(x,\xi)\), not exact function values.
The objective \(f\) is latent; it is used only to specify unbiasedness and variance, and to measure the final error.
\end{definition}

\begin{definition}[Stochastic first-order method and minimax error]
\label{def:stochastic-model}
A randomized stochastic first-order method is a family \(A=(A_d)_{d\ge 1}\). On an objective \(f:\R^d\to\R\) and an
oracle \(G\in\mathfrak O_{\sigma^2}(f)\), the method starts from \(x_0=0\), and each query \(x_t\) is measurable with
respect to the past stochastic-gradient transcript
\[
\{G(x_s,\xi_s):0\le s<t\}
\]
and the method's internal randomness. The algorithm has no oracle access to \(f(x)\). We write \(x_T^A(f,G)\) for the
point returned after \(T\) oracle calls.

The fixed-budget bounded-variance stochastic minimax error over the function class
\(\mathcal F_{\alpha,L,\tau,\Delta_0}(0)\) is
\[
\mathcal E_T^{\mathrm{BV}}\!\left(\mathcal F_{\alpha,L,\tau,\Delta_0}(0),\sigma^2\right)
:=
\inf_A
\sup_{d\ge 1}
\sup_{f\in\mathcal F^{(d)}_{\alpha,L,\tau,\Delta_0}(0)}
\sup_{G\in\mathfrak O_{\sigma^2}(f)}
\E\bigl[f(x_T^A(f,G))-f_\star\bigr],
\]
where the superscript $\mathrm{BV}$ stands for ``bounded variance,'' and the expectation is over both the oracle
randomness and the algorithm's internal randomness. The two suprema have different meanings: first we choose the worst
objective in the function class, and then, for that objective, the worst admissible unbiased bounded-variance gradient
oracle.
The corresponding complexity is
\[
T_\varepsilon^{\mathrm{BV}}\!\left(\mathcal F_{\alpha,L,\tau,\Delta_0}(0),\sigma^2\right)
:=
\inf\Bigl\{T\ge 0:\
\mathcal E_T^{\mathrm{BV}}\!\left(\mathcal F_{\alpha,L,\tau,\Delta_0}(0),\sigma^2\right)\le \varepsilon
\Bigr\}.
\]
We also abbreviate
\[
T_\varepsilon^{\mathrm{BV}}(\alpha,L,\tau,\Delta_0,\sigma^2)
:=
T_\varepsilon^{\mathrm{BV}}\!\left(\mathcal F_{\alpha,L,\tau,\Delta_0}(0),\sigma^2\right).
\]
\end{definition}

\section{\texorpdfstring{Deterministic first-order complexity under $\alpha$-P\L{}}{Deterministic first-order complexity under alpha-PL}}
\label{sec:det}

This section derives the deterministic lower bound, which shows no deterministic first-order
method can improve the standard gradient-descent dependence on \(L,\tau\), and \(\varepsilon\) over the whole globally
smooth, sublevel-\(\alpha\)-P\L{} class. 
A proof sketch follows the theorem statement; the full argument is given in
Appendix~\ref{app:deterministic}.

\begin{theorem}[Sharp deterministic lower bound under sublevel $\alpha$-P\L{} geometry]
\label{thm:matching}
For every $\alpha\in[1,2)$ there exists a constant \(c_\alpha^{\mathrm{det}}>0\) such that, for every
$L,\tau,\Delta_0>0$ and every target accuracy
\[
\varepsilon\in\bigl(0,c_\alpha^{\mathrm{det}}\min\{\Delta_0,L^{\alpha/(2-\alpha)}\tau^{2/(2-\alpha)}\}\bigr],
\]
\[
T_\varepsilon\!\left(\mathcal F_{\alpha,L,\tau,\Delta_0}(0)\right)
\ge
c_\alpha^{\mathrm{det}} L\tau^{2/\alpha}\varepsilon^{-(2-\alpha)/\alpha}.
\]
\end{theorem}

\begin{proof}[Proof sketch]
For each \(\alpha\in[1,2)\) and budget \(T\ge 1\), we construct a function \(\Phi_T\) on \(\R^{2T+1}\) which lies in \(\mathcal F_{\alpha,L,\tau,\Delta_0}(0)\) and for which no zero-respecting
deterministic first-order method can achieve objective gap better than the claimed rate. The construction is built on
the Yue--Fang--Lin (YFL) zero chain~\citep{YueFangLin2023}; the technical content is in choosing the chain's scale so
that two competing constraints, the \emph{rate} forced by the budget and the \emph{geometry} of the function class, are met at the same time.

\paragraph{Base zero chain.}
On \(\R^m\), set
\[
g_m(u):=q_m(u)+\sum_{i=1}^m v(u_i),
\qquad
q_m(u):=\tfrac12 u_1^2+\tfrac12\sum_{i=1}^{m-1}(u_{i+1}-u_i)^2,
\]
where \(v:\R\to\R\) is an explicit piecewise-quadratic bump with \(v(1)>0\) and \(v'(1)=0\). The construction relies
on two properties of \(g_m\):
\begin{enumerate}[label=(P\arabic*),leftmargin=2.4em,itemsep=0.2em]
\item \emph{Zero-chain property.} The centered function \(g_m(\one_m-\,\cdot\,)\) is a first-order zero-chain in
the variable \(x\): the iterate support of a zero-respecting method queried from \(x_0=0\), after \(T\) steps, is contained
in \(\{1,\ldots,T\}\).
\item \emph{Quadratic-P\L{} bound.} \(g_m(u)-g_{m,\star}\le C_{\mathrm{PL}}\,m\,\norm{\nabla g_m(u)}^2\) for a universal
constant \(C_{\mathrm{PL}}\ge 1\) (the \(t=1\) case of \citep[Lemma~2]{YueFangLin2023}).
\end{enumerate}

\paragraph{The rescaled hard instance.}
Set \(m_T:=2T+1\) and, for a scale \(a_T>0\) to be chosen,
\[
\Phi_T(x):=a_T^2\,g_{m_T}\!\Bigl(\one_{m_T}-\frac{x}{a_T}\Bigr),
\qquad x\in\R^{m_T}.
\]
Since scaling preserves coordinate supports, \(\Phi_T\) inherits the zero-chain property from \(g_{m_T}\). By (P1), for 
iterate \(x_T\), \((x_T)_i=0\) for every \(i>T\). Additionally, the corresponding chain variables satisfy \(u_i=1\) on
this hidden tail. By \(v(1)>0\), the tail contributes positive objective value,
\begin{equation}
\label{eq:tail-gap-sketch}
\Phi_T(x_T)-\Phi_{T,\star}\;\gtrsim\;T\,a_T^2,
\end{equation}
while by \(v'(1)=0\) no gradient signal at those coordinates is ever produced. The hidden tail persists in the
objective and cannot be eliminated by further first-order queries.

\paragraph{Two constraints, one scale.}
For the construction to deliver a sharp lower bound at accuracy \(\varepsilon\), the scale \(a_T\) and the budget
\(T\) must jointly satisfy two equations -- one pertaining to the budget, and the other to the function class.

\smallskip
\noindent\emph{(a) Rate constraint.} The hidden-tail gap~\eqref{eq:tail-gap-sketch} must match the target accuracy:
\begin{equation}
\label{eq:rate-constraint}
T\,a_T^2\;\asymp\;\varepsilon.
\end{equation}

\smallskip
\noindent\emph{(b) Geometry constraint.} The rescaled \(\Phi_T\) must lie in
\(\mathcal F_{\alpha,L,\tau,\Delta_0}(0)\), i.e.\ satisfy the sublevel-\(\alpha\)-P\L{} inequality on
\(S_0(\Phi_T)\). The chain rule gives \(\norm{\nabla g_{m_T}}=\norm{\nabla\Phi_T}/a_T\). Substituting this into~(P2)
yields the quadratic-P\L{} bound
\begin{equation}
\label{eq:rescaled-qPL}
\Phi_T(x)-\Phi_{T,\star}\;\le\;C_{\mathrm{PL}}\,m_T\,\norm{\nabla\Phi_T(x)}^2
\qquad(x\in S_0(\Phi_T)),
\end{equation}
whose right-hand side is in terms of \(\norm{\nabla\Phi_T}^2\), not \(\norm{\nabla\Phi_T}^\alpha\). Converting to the
sublevel-\(\alpha\)-P\L{} form in the small-gradient regime \(\norm{\nabla\Phi_T}\le a_T\) requires
\[
m_Ta_T^2\lesssim a_T^\alpha.
\]
Smaller \(a_T\) would still preserve the geometry but would also shrink the hidden-tail gap. The sharp construction
therefore uses the largest admissible scale, i.e.:
\begin{equation}
\label{eq:matching}
\boxed{\;m_T\,a_T^{2}\;=\;a_T^{\alpha}\;}
\end{equation}
(a complementary global gap bound on \(S_0(\Phi_T)\) handles the large-gradient regime
\(\norm{\nabla\Phi_T}>a_T\)).

\smallskip
\noindent\emph{Solving the system.} With \(m_T=2T+1\), equations~\eqref{eq:rate-constraint}
and~\eqref{eq:matching} form a system of equations in the two unknowns \((a_T,T)\) given the target accuracy
\(\varepsilon\). From~\eqref{eq:matching},
\[
a_T\;\asymp\;T^{-1/(2-\alpha)}.
\]
Substituting into~\eqref{eq:rate-constraint},
\[
T\,a_T^2\;\asymp\;T\cdot T^{-2/(2-\alpha)}\;=\;T^{-\alpha/(2-\alpha)}\;\asymp\;\varepsilon,
\qquad\text{equivalently,}\qquad
T\;\asymp\;\varepsilon^{-(2-\alpha)/\alpha}.
\]
Any zero-respecting deterministic method therefore needs
\(T\gtrsim\varepsilon^{-(2-\alpha)/\alpha}\) queries to reach objective gap \(\varepsilon\) on \(\Phi_T\).
The fact that the two constraints admit a common solution is non-trivial: it relies on the YFL chain's quadratic-P\L{}
constant scaling linearly with \(m_T\) in~(P2), which allows for the
identity~\eqref{eq:matching} to be consistent with the rate constraint~\eqref{eq:rate-constraint}.

\paragraph{From zero-respecting to all deterministic methods.}
A final affine rescaling in domain and objective places \(\Phi_T\) in \(\mathcal F_{\alpha,L,\tau,\Delta_0}(0)\) with
the prescribed parameters and yields the fixed-budget lower bound
\[
\mathcal E_T\!\left(\mathcal F_{\alpha,L,\tau,\Delta_0}(0)\right)
\ge
c_\alpha^{\mathrm{det}}\,
\min\Bigl\{\Delta_0,\,\bigl(L\tau^{2/\alpha}/T\bigr)^{\alpha/(2-\alpha)}\Bigr\}.
\]
The accuracy restriction \(\varepsilon\lesssim L^{\alpha/(2-\alpha)}\tau^{2/(2-\alpha)}\) is the natural scale of this rescaled family:
above that scale, the class itself cannot sustain a nonzero sublevel gap up to constants depending only on \(\alpha\).
The class is orthogonally invariant, so the standard resisting-oracle rotation of
\citet[Proposition~1]{CarmonDuchiHinderSidford2021} lifts the bound to all deterministic first-order methods.
Rearranging yields the iteration-complexity statement of Theorem~\ref{thm:matching}. The full argument appears in
Appendix~\ref{app:deterministic}.
\end{proof}

\begin{proposition}[Gradient-descent benchmark]
\label{prop:gd-main}
For every \(\alpha\in[1,2)\), there is a constant \(C_\alpha>0\) such that gradient descent with step size \(1/L\)
achieves
\[
T_\varepsilon\!\left(\mathcal F_{\alpha,L,\tau,\Delta_0}(0)\right)
\le
C_\alpha L\tau^{2/\alpha}\varepsilon^{-(2-\alpha)/\alpha}
\qquad
(\varepsilon\in(0,\Delta_0]).
\]
\end{proposition}

\begin{proof}[Proof idea]
The \(1/L\) gradient-descent step is monotone by global smoothness, so all iterates stay in \(S_0(f)\). On this set,
the sublevel \(\alpha\)-P\L{} inequality converts the descent lemma into the scalar recursion
\[
\Delta_{k+1}
\le
\Delta_k-\frac{1}{2L\tau^{2/\alpha}}\Delta_k^{2/\alpha},
\qquad
\Delta_k:=f(x_k)-f_\star.
\]
Solving the recursion gives the stated complexity. The full calculation is presented in
Corollary~\ref{cor:upper-complexity}.
\end{proof}

\begin{remark}[Minimax optimality]
From Proposition~\ref{prop:gd-main} and Theorem~\ref{thm:matching} together, we have
\[
T_\varepsilon\!\left(\mathcal F_{\alpha,L,\tau,\Delta_0}(0)\right)
=
\Theta\left(L\tau^{2/\alpha}\varepsilon^{-(2-\alpha)/\alpha}\right)
\qquad
(\alpha\in[1,2),\ \varepsilon\le c_\alpha^{\mathrm{det}}\min\{\Delta_0,L^{\alpha/(2-\alpha)}\tau^{2/(2-\alpha)}\}).
\]
Therefore, our zero-chain lower bound proves that the standard gradient-descent rate cannot be improved by any
deterministic first-order method in this class.
\end{remark}

\begin{remark}[The endpoint $\alpha=1$ and the convex bounded-domain subclass]
\label{rem:alpha1-convex}
At $\alpha=1$, Proposition~\ref{prop:gd-main} and Theorem~\ref{thm:matching} give
\[
T_\varepsilon=\Theta\!\left(\frac{L\tau^2}{\varepsilon}\right)
\]
on the full sublevel-\(\alpha=1\) class. As discussed in Example~\ref{ex:alpha1-convex}, smooth convex minimization
with a bounded initial sublevel set is contained in this class with \(\tau\) equal to the sublevel radius \(D_0\), but
it is only a subclass. Therefore the lower bound above does not contradict the accelerated convex rate
\[
T_\varepsilon=\Theta\!\left(\sqrt{\frac{LD^2}{\varepsilon}}\right)
\]
for smooth convex minimization over a bounded domain~\citep{Nesterov2004}. Acceleration uses convexity, while the
sublevel-\(\alpha=1\) model does not impose convexity. For an example of a globally
smooth nonconvex  \(\alpha=1\) class function see Appendix~\ref{app:nonconvex-alpha1}.
\end{remark}

\section{\texorpdfstring{Stochastic lower bound under $\alpha$-P\L{}}{Stochastic lower bound under alpha-PL}}
\label{sec:sto}

This section presents the stochastic lower bounds in two parts. We first consider \(1\le\alpha<2\) and prove the bounded-variance lower bound, including both the contribution that is already present with exact gradients and the
additional contribution caused by stochastic-gradient noise. We then consider the classical P\L{} endpoint
\(\alpha=2\), for which we prove a variance-dependent lower bound even under a global P\L{}
inequality. The two results use the same
probabilistic zero-chain mechanism, but their parameter choices are different.

\subsection{The regime \texorpdfstring{\(1\le\alpha<2\)}{1 <= alpha < 2}}
\label{sec:stochastic-alpha-lt-two}

Section~\ref{sec:det} considered deterministic algorithms receiving exact gradients. We now allow the algorithm to be
randomized, and each oracle call returns an unbiased stochastic gradient whose conditional variance is at most
\(\sigma^2\), as specified in Definition~\ref{def:stochastic-model}.

Two quantities describe the resulting lower bound. Define
\[
D_\varepsilon
:=
L\tau^{2/\alpha}\varepsilon^{-(2-\alpha)/\alpha},
\qquad
S_\varepsilon
:=
L\sigma^2\tau^{4/\alpha}\varepsilon^{-(4-\alpha)/\alpha}.
\]
Here \(D_\varepsilon\) is the exact-gradient difficulty that remains when the oracle has no noise, whereas
\(S_\varepsilon\) is the additional difficulty created by the variance allowance \(\sigma^2\). The next theorem shows
that their sum, up to an \(\alpha\)-dependent constant, is unavoidable in the bounded-variance model.

\begin{theorem}[Sharp stochastic lower bound]
\label{thm:stochastic-matching}
For every $\alpha\in[1,2)$ there exists a constant $c_\alpha^{\mathrm{sto}}>0$ such that for every
$L,\tau,\Delta_0,\sigma^2>0$ and every target accuracy
\[
\varepsilon\in\bigl(0,c_\alpha^{\mathrm{sto}}\min\{\Delta_0,L^{\alpha/(2-\alpha)}\tau^{2/(2-\alpha)}\}\bigr],
\]
\[
T_\varepsilon^{\mathrm{BV}}(\alpha,L,\tau,\Delta_0,\sigma^2)
\ge
c_\alpha^{\mathrm{sto}}\left(
L\tau^{2/\alpha}\varepsilon^{-(2-\alpha)/\alpha}
 + L\sigma^2\tau^{4/\alpha}\varepsilon^{-(4-\alpha)/\alpha}
\right).
\]
\end{theorem}

\begin{proof}[Proof sketch]
\emph{The zero-chain backbone.}
As in Section~\ref{sec:det}, the hard objective stores useful information in coordinates that must be discovered in
order. Such an objective is called a zero chain: before coordinate \(i\) has been activated, the gradient contains no
information in coordinate \(i+1\). We use a chain of length \(2m\). If the method has activated fewer than \(m\)
coordinates, then the final \(m\) coordinates remain hidden and contribute a nonnegligible objective gap.

For concreteness, the normalized chain used in the proof is
\[
F_N(x)
=
\varphi(x_1)
+\sum_{i=1}^{N-1}\Theta(x_i)\varphi(x_{i+1})
+\sum_{i=1}^N B(x_i).
\]
Here \(B\) creates a well in each coordinate, \(\varphi(x_{i+1})\) contributes the error associated with coordinate
\(i+1\), and the gate \(\Theta(x_i)\) prevents the gradient from exposing that contribution before coordinate \(i\)
has been activated. Appendix~\ref{app:stochastic} gives the scalar functions and verifies the zero-chain property.

\emph{Making coordinate revelation stochastic.}
At a query \(x\), let
\[
k:=\prog_{1/4}(x)
:=
\max\bigl(\{0\}\cup\{i:\abs{x_i}>1/4\}\bigr)
\]
denote the last activated coordinate. The reveal probability \(q\in(0,1]\) is the probability that one oracle call
exposes the next available coordinate. To preserve unbiasedness, the oracle returns the true gradient on the activated
coordinates and multiplies the hidden part by \(z/q\), where \(z\sim\mathrm{Bernoulli}(q)\):
\[
[g_{m,q}(x,z)]_i
=
\nabla_iF_{2m}(x)
\left(
1+\mathbf 1\{i>k\}\left(\frac{z}{q}-1\right)
\right).
\]
Because \(\E[z/q]=1\), this is an unbiased gradient estimator. A new coordinate is visible only when \(z=1\), however,
so decreasing \(q\) delays progress. The same factor \(1/q\) also increases the variance, which prevents us from taking
\(q\) arbitrarily small.

\emph{Choosing the amplitude.}
We next introduce an amplitude \(a_m>0\). Scaling the chain by \(a_m\) controls two quantities simultaneously: the
objective gap carried by the hidden coordinates and the constant in the sublevel-\(\alpha\)-P\L{} inequality. The
choice of \(m\), \(a_m\), and \(q\) is determined by three constraints.

\smallskip
\noindent\emph{(a) Error constraint.}
The final \(m\) coordinates contribute gap of order \(m\) before scaling and therefore gap of order \(ma_m^2\) after
scaling. To keep the objective error above the target accuracy, we impose
\begin{equation}
\label{eq:sto-rate-constraint}
ma_m^2\asymp\varepsilon.
\end{equation}
Thus this constraint fixes the amount of objective mass that must remain hidden.

\smallskip
\noindent\emph{(b) Geometry constraint.}
The unscaled chain satisfies a quadratic-P\L{} inequality with a constant proportional to \(m\). After scaling, the
target sublevel-\(\alpha\)-P\L{} inequality requires
\(ma_m^2\lesssim a_m^\alpha\). A smaller amplitude would satisfy this requirement but would weaken the hidden
objective gap. We therefore use the largest admissible amplitude and set
\begin{equation}
\label{eq:sto-geometry-constraint}
ma_m^2=a_m^\alpha.
\end{equation}
Solving this equality gives
\[
a_m=m^{-1/(2-\alpha)},
\qquad
ma_m^2=a_m^\alpha=m^{-\alpha/(2-\alpha)}.
\]
This identity matches the objective scale of the hidden coordinates to the required
sublevel-\(\alpha\)-P\L{} geometry.

\smallskip
\noindent\emph{(c) Variance constraint.}
Let \(s^2\) denote the variance budget in the normalized problem. After amplitude scaling, Bernoulli thinning gives
variance of order \(a_m^2/q\). In the range where the resulting probability is at most one, we saturate the variance
budget:
\begin{equation}
\label{eq:sto-variance-constraint}
\frac{a_m^2}{q}\asymp s^2,
\qquad\text{and hence}\qquad
q\asymp\frac{a_m^2}{s^2}.
\end{equation}
This is the smallest reveal probability allowed by the variance constraint.

\emph{From reveal probability to oracle calls.}
The method must reveal order \(m\) coordinates before it reaches the last half of the chain. Each query produces at
most one such reveal, with probability at most \(q\). Consequently, order \(m/q\) oracle calls are required.

\emph{Hiding the coordinate system.}
The preceding coordinate argument directly constrains methods that respect the visible chain coordinates. To cover
randomized, adaptive methods that may query arbitrary points, we hide the chain basis with a random orthogonal map
\(U\). A radial clipping map \(\rho_m\) keeps the chain input bounded even when the algorithm submits a large query, and
a small quadratic term with a fixed coefficient \(\eta>0\) controls the initial sublevel set.
These operations produce the function-oracle pair
\begin{align}
\widehat F_{m,U}(x)
&:=
F_{2m}(U^\top\rho_m(x))+\frac{\eta}{2}\norm{x}^2,
\label{eq:sto-clipped-objective-main}\\
\widehat g_{m,q,U}(x,z)
&:=
J\rho_m(x)^\top U g_{m,q}(U^\top\rho_m(x),z)+\eta x,
\label{eq:sto-clipped-oracle-main}
\end{align}
where \(J\rho_m(x)\) is the Jacobian of the clipping map. The second formula applies the chain rule to pull the
stochastic zero-chain gradient back to the original query space. Random rotation prevents an adaptive query from
having a large component in an unrevealed chain direction, while clipping preserves global smoothness and the
bounded-variance guarantee. Appendix~\ref{app:stochastic} proves that the \(m/q\) progress barrier therefore holds for
randomized, adaptive, arbitrary-query algorithms.

The amplitude-scaled function-oracle pair is
\[
\Psi_{m,q,U}(x)
:=
a_m^2\widehat F_{m,U}(x/a_m),
\qquad
G_{m,q,U}(x,z)
:=
a_m\,\widehat g_{m,q,U}(x/a_m,z).
\]
This formula implements the amplitude balance described above; the algorithm receives only samples of
\(G_{m,q,U}\).

\emph{Solving the normalized constraints.}
Equations~\eqref{eq:sto-rate-constraint} and~\eqref{eq:sto-geometry-constraint} give
\(a_m^\alpha\asymp\varepsilon\), and hence
\[
m\asymp\varepsilon^{-(2-\alpha)/\alpha},
\qquad
a_m^2\asymp\varepsilon^{2/\alpha}.
\]
Equation~\eqref{eq:sto-variance-constraint} then gives
\[
q\asymp\frac{\varepsilon^{2/\alpha}}{s^2}.
\]
Substitution into the \(m/q\) progress barrier yields the normalized stochastic-noise lower bound
\[
\frac{m}{q}
\asymp
s^2\varepsilon^{-(4-\alpha)/\alpha}.
\]
Thus the exponent \((4-\alpha)/\alpha\) comes from combining the number of coordinates that must be revealed with the
variance-limited probability of each reveal.

\emph{Restoring the problem parameters.}
An objective-space and domain-space scaling transfers the normalized family to the prescribed
\((L,\tau,\sigma^2)\). It turns the preceding display into
\[
T_\varepsilon^{\mathrm{BV}}
\ge
\Omega\left(
L\sigma^2\tau^{4/\alpha}\varepsilon^{-(4-\alpha)/\alpha}
\right)
\]
whenever
\[
\varepsilon
\lesssim
\min\left\{
\Delta_0,\,
L^{\alpha/(2-\alpha)}\tau^{2/(2-\alpha)},\,
\tau\sigma^\alpha
\right\}.
\]
The first restriction places the hard instance within the initial-gap budget, and the second makes the normalized
target accuracy small enough for the construction. The third restriction is used only by the probabilistic
stochastic-noise construction.

\emph{Combining the exact-gradient and stochastic-noise contributions.}
The quantities \(D_\varepsilon\) and \(S_\varepsilon\), defined before the theorem, satisfy
\[
\frac{S_\varepsilon}{D_\varepsilon}
=
\left(\frac{\tau\sigma^\alpha}{\varepsilon}\right)^{2/\alpha}.
\]
Therefore \(S_\varepsilon\) is at least a constant multiple of \(D_\varepsilon\) when
\(\varepsilon\lesssim\tau\sigma^\alpha\). In that range, the probabilistic construction above supplies the
stochastic-noise lower bound. The exact-gradient lower bound is obtained from the same randomly rotated family by
setting the reveal probability to \(q=1\): then \(z=1\) almost surely, so the oracle returns the exact gradient and has
zero variance.

If the two constructions give lower bounds \(c_DD_\varepsilon\) and \(c_SS_\varepsilon\), respectively, then
\[
\max\{c_DD_\varepsilon,c_SS_\varepsilon\}
\ge
\frac{\min\{c_D,c_S\}}{2}
\bigl(D_\varepsilon+S_\varepsilon\bigr).
\]
Hence the two contributions give their sum up to a constant factor. In the complementary range
\(\varepsilon\gtrsim\tau\sigma^\alpha\), the ratio above gives \(S_\varepsilon\lesssim D_\varepsilon\), so the
\(q=1\) exact-gradient construction alone controls the sum. The crossover \(\tau\sigma^\alpha\) thus selects the active
part of the proof but does not restrict the theorem's final accuracy range. Appendix~\ref{app:stochastic} contains the
complete variance, random-rotation, scaling, and accuracy calculations.
\end{proof}

\begin{remark}[Matching SGD upper rates and what remains open]
\label{rem:stochastic-upper-main}
Theorem~\ref{thm:stochastic-matching} is a lower bound for the gradient-only bounded-variance oracle model. Its rate
matches existing SGD upper bounds under \(\alpha\)-P\L{} geometry:
\citet{FatkhullinEtesamiHeKiyavash2022} obtain $O\left(
L\tau^{2/\alpha}\varepsilon^{-(2-\alpha)/\alpha}
 +
L\sigma^2\tau^{4/\alpha}\varepsilon^{-(4-\alpha)/\alpha}
\right)$ for SGD under the assumption that the iterates remain in the region where the \(\alpha\)-P\L{}
inequality is valid; see also the local stochastic first-order analyses of
\citet{MasihaSalehkaleybarHeKiyavash2022,MasihaSalehkaleybarHeKiyavashThiran2026}.
Theorem~\ref{thm:stochastic-matching} shows that this dependence on
\((L,\tau,\sigma^2,\varepsilon)\) is unavoidable; in the regime
\(\varepsilon\lesssim\tau\sigma^\alpha\), the second term is the noise-dominated contribution with exponent
\((4-\alpha)/\alpha\). What remains open is a fully self-contained upper-bound theory for
gradient-only SGD on the present sublevel model: one needs conditions ensuring, with high probability, that the SGD trajectory stays inside the region where the sublevel \(\alpha\)-P\L{} geometry is valid.
\end{remark}

\subsection{The classical P\L{} endpoint \texorpdfstring{\(\alpha=2\)}{alpha = 2}}
\label{sec:sto-endpoint}

The proof for \(1\le\alpha<2\) chose the amplitude
\[
a_m=m^{-1/(2-\alpha)}.
\]
This expression is undefined at \(\alpha=2\), so the endpoint result cannot be obtained by substituting
\(\alpha=2\) into Theorem~\ref{thm:stochastic-matching}. We instead retain the same probabilistic zero-chain
construction and choose new objective and domain scalings for the endpoint.

At \(\alpha=2\), we use the same hard pair \((\widehat F_{m,U},\widehat g_{m,q,U})\) as above. Its oracle reveals
the chain coordinates sequentially, each with probability \(q\); random rotation hides unrevealed directions; and
clipping with quadratic regularization ensures global smoothness and uniformly bounded variance. The new geometric
ingredient is a quantitative estimate proved in Appendix~\ref{app:sto-endpoint-proof}: the objective
\(\widehat F_{m,U}\) satisfies the P\L{} inequality on all of \(\R^d\), with a constant proportional to the chain
length \(m\). This global estimate allows the endpoint lower bound to hold in the global-P\L{} subclass.

To state the stronger result, define the globally P\L{} subclass
\[
\mathcal F^{\mathrm{glob}}_{2,L,\tau,\Delta_0}(0)
:=
\left\{
f\in\mathcal F_{2,L,\tau,\Delta_0}(0):
f(x)-f_\star\le\tau\norm{\nabla f(x)}^2
\ \text{for every }x\text{ in the domain of }f
\right\}.
\]
Thus every member is globally \(L\)-smooth and satisfies the P\L{} inequality globally, not only on its initial
sublevel set. Using the minimax definition from Definition~\ref{def:stochastic-model} with this function
class, write
\[
T_\varepsilon^{\mathrm{BV,glob}}(L,\tau,\Delta_0,\sigma^2)
:=
T_\varepsilon^{\mathrm{BV}}
\left(\mathcal F^{\mathrm{glob}}_{2,L,\tau,\Delta_0}(0),\sigma^2\right).
\]

\begin{theorem}[Variance-dependent lower bound under global P\L{} geometry]
\label{thm:stochastic-pl-endpoint}
There exist universal numerical constants $c_{\mathrm{end}},c_{\mathrm{acc}}^{\mathrm{end}}>0$ such that, for every
$L,\tau,\Delta_0,\sigma^2>0$ satisfying $L\tau\ge1/2$ and every
\[
0<\varepsilon\le c_{\mathrm{acc}}^{\mathrm{end}}\Delta_0,
\]
the gradient-only bounded-variance oracle complexity over the globally smooth, globally P\L{} class satisfies
\[
T_\varepsilon^{\mathrm{BV,glob}}(L,\tau,\Delta_0,\sigma^2)
\ge
c_{\mathrm{end}}\frac{L\tau^2\sigma^2}{\varepsilon}.
\]
One may take $c_{\mathrm{acc}}^{\mathrm{end}}=1/10$.
\end{theorem}

Because
\[
\mathcal F^{\mathrm{glob}}_{2,L,\tau,\Delta_0}(0)
\subseteq
\mathcal F_{2,L,\tau,\Delta_0}(0),
\]
the minimax complexity over the larger sublevel-P\L{} class is at least
\(T_\varepsilon^{\mathrm{BV,glob}}(L,\tau,\Delta_0,\sigma^2)\). Hence the same endpoint lower bound also holds for
the larger sublevel-P\L{} class.

\begin{proof}[Proof sketch]
\emph{What changes at \(\alpha=2\).}
For \(1\le\alpha<2\), the preceding proof chose the amplitude
\[
a_m=m^{-1/(2-\alpha)}
\]
by balancing the objective gap hidden in the chain with the sublevel-\(\alpha\)-P\L{} geometry. This expression is
undefined at \(\alpha=2\). We therefore retain the normalized zero-chain objective \(\widehat F_{m,U}\) defined in \eqref{eq:sto-clipped-objective-main}, where \(m\)
is the chain-length parameter and \(U\) is the hidden rotation, and rescale it directly:
\begin{equation}
\label{eq:endpoint-scaled-objective-main}
f_{m,U}(x):=\lambda\widehat F_{m,U}(\rho x).
\end{equation}
Here \(\lambda>0\) is a vertical scale: it multiplies every objective gap. The scalar \(\rho>0\) rescales the domain
and therefore changes the gradient and smoothness constants. This scalar should not be confused with the radial
clipping map \(\rho_m\) inside \(\widehat F_{m,U}\). The endpoint proof first chooses the chain length \(m\)
proportional to \(L\tau\), ensuring that the scaled objective can be both \(L\)-smooth and globally P\L{} with
constant at most \(\tau\). It then chooses \(\lambda\) to set the residual gap at the target scale \(\varepsilon\) and
\(\rho\) to obtain smoothness \(L\). The reveal probability \(q\) is chosen afterward to satisfy the variance budget
\(\sigma^2\).

\emph{The normalized hard pair.}
Fix a positive integer \(m\) and a reveal probability \(q\in(0,1]\). Recall the objective
\(\widehat F_{m,U}\) and stochastic oracle \(\widehat g_{m,q,U}\) defined in
\eqref{eq:sto-clipped-objective-main} and~\eqref{eq:sto-clipped-oracle-main}, respectively. This pair is built from a
zero chain that requires \(m\) useful coordinate directions to be discovered sequentially. Its oracle reveals each
next available direction with probability \(q\). The random rotation \(U\) hides unrevealed directions from algorithms
whose queries may depend on all previous oracle responses, while clipping and quadratic regularization ensure global
smoothness, uniformly bounded variance, and containment of the initial sublevel set.

There are universal constants \(L_{\rm b},C_{\mathrm{gPL}},C_\Delta,C_{\rm var},c_0>0\) such that the normalized pair has
the five properties needed here:
\begin{equation}
\label{eq:endpoint-normalized-geometry-main}
\text{smoothness}\le L_{\rm b},
\qquad
\text{global P\L{} constant}\le C_{\mathrm{gPL}}m,
\end{equation}
\begin{equation}
\label{eq:endpoint-normalized-oracle-main}
\text{initial gap}\le C_\Delta m,
\qquad
\text{oracle variance}\le \frac{C_{\rm var}}{q},
\end{equation}
and, provided $d\ge d_{m,q}$, where $d_{m,q}$ is defined in
Appendix~\ref{app:sto-clipping}, equation~\eqref{eq:dmq-threshold},
\begin{equation}
\label{eq:endpoint-normalized-barrier-main}
\E\!\left[\widehat F_{m,U}(x_t)-\widehat F_{m,U,\star}\right]\ge m
\qquad
\left(t\le c_0\frac{m}{q}\right).
\end{equation}
The initial-gap, variance, and progress estimates are proved in Appendix~\ref{app:sto-clipping}; see
Lemma~\ref{lem:clipped-rotated}. The sharper endpoint smoothness bound and the global P\L{} estimate are proved in
Appendix~\ref{app:sto-endpoint-proof}; see Lemma~\ref{lem:endpoint-global-pl}.
The time \(m/q\) reflects the need for order \(m\) successful coordinate revelations, each occurring with probability
\(q\).

Let $m_0$ denote the universal minimum chain length required by the normalized construction, and define
\begin{equation}
\label{eq:endpoint-kappa-threshold-main}
\kappa_0:=C_{\mathrm{gPL}}L_{\rm b}m_0.
\end{equation}
We now split the proof according to whether $L\tau$ is above or below this threshold.

\textbf{Case 1: $L\tau>\kappa_0$ (the large-$L\tau$ regime).}
In this case, the value of $L\tau$ permits a zero chain of at least the minimum required length.

\emph{Scaling the normalized pair.}
The scaled objective was defined in~\eqref{eq:endpoint-scaled-objective-main}. We scale its stochastic oracle
consistently by setting
\begin{equation}
\label{eq:endpoint-scaled-oracle-main}
G_{m,q,U}(x,z):=\lambda\rho\,\widehat g_{m,q,U}(\rho x,z).
\end{equation}

Let $L_f$ and $\tau_f$ denote the certified smoothness and global P\L{} constants of $f_{m,U}$, respectively.
These are not new target parameters: the scaling below will ensure $L_f\le L$ and $\tau_f\le\tau$, where $L$ and
$\tau$ are the prescribed parameters in Theorem~\ref{thm:stochastic-pl-endpoint}. Applying the objective
scaling~\eqref{eq:endpoint-scaled-objective-main} to the normalized bounds
in~\eqref{eq:endpoint-normalized-geometry-main} gives
\begin{equation}
\label{eq:endpoint-scaled-geometry-main}
L_f\le\lambda\rho^2L_{\rm b},
\qquad
\tau_f\le\frac{C_{\mathrm{gPL}}m}{\lambda\rho^2}.
\end{equation}
Multiplying the two inequalities in~\eqref{eq:endpoint-scaled-geometry-main} gives
$L_f\tau_f\le C_{\mathrm{gPL}}L_{\rm b}m$, independently of \(\lambda\) and \(\rho\). Consequently, the endpoint
construction takes \(m\) proportional to \(L\tau\). This relation replaces the amplitude balance used for
\(\alpha<2\).

Choose
\begin{equation}
\label{eq:endpoint-chain-length-main}
m:=\left\lfloor\frac{L\tau}{C_{\mathrm{gPL}}L_{\rm b}}\right\rfloor.
\end{equation}
The choice~\eqref{eq:endpoint-chain-length-main} makes \(m\asymp L\tau\) while keeping the global P\L{} constant at
most \(\tau\). Next, set
\begin{equation}
\label{eq:endpoint-objective-scale-main}
\lambda:=\frac{2\varepsilon}{m}.
\end{equation}
By~\eqref{eq:endpoint-normalized-barrier-main} and the scaled-objective definition
in~\eqref{eq:endpoint-scaled-objective-main}, the objective scale in
\eqref{eq:endpoint-objective-scale-main} makes the scaled residual gap at least
\(\lambda m=2\varepsilon\). Finally, set
\begin{equation}
\label{eq:endpoint-domain-scale-main}
\rho^2:=\frac{L}{\lambda L_{\rm b}}.
\end{equation}
Substituting~\eqref{eq:endpoint-domain-scale-main} into the first inequality in
\eqref{eq:endpoint-scaled-geometry-main} gives \(L_f\le L\). The second inequality in
\eqref{eq:endpoint-scaled-geometry-main}, together with~\eqref{eq:endpoint-chain-length-main}, gives
\[
\tau_f\le\frac{C_{\mathrm{gPL}}mL_{\rm b}}{L}\le\tau.
\]
Finally, the initial-gap bound in~\eqref{eq:endpoint-normalized-oracle-main} and the objective scale
in~\eqref{eq:endpoint-objective-scale-main} give
$C_\Delta\lambda m=10\varepsilon\le\Delta_0$ under the theorem's accuracy assumption.

\emph{Choosing the reveal probability.}
Revealing a coordinate only with probability \(q\), while preserving unbiasedness, increases the variance by a factor
proportional to \(1/q\). Applying the oracle scaling~\eqref{eq:endpoint-scaled-oracle-main} to the normalized variance
bound in~\eqref{eq:endpoint-normalized-oracle-main}, and then using~\eqref{eq:endpoint-domain-scale-main}, gives
\begin{equation}
\label{eq:endpoint-scaled-variance-main}
\text{oracle variance}
\le
\frac{C_{\rm var}\lambda L}{L_{\rm b}q}.
\end{equation}
The reveal probability that makes the right-hand side of~\eqref{eq:endpoint-scaled-variance-main} equal to the
variance budget \(\sigma^2\) is
\begin{equation}
\label{eq:endpoint-qstar-main}
q_\star:=\frac{C_{\rm var}\lambda L}{L_{\rm b}\sigma^2}.
\end{equation}

If the value in~\eqref{eq:endpoint-qstar-main} satisfies \(q_\star\le1\), set \(q=q_\star\). By
\eqref{eq:endpoint-scaled-variance-main}, the oracle then has variance at most \(\sigma^2\), and
\[
\frac{m}{q}
\asymp
\frac{m^2\sigma^2}{L\varepsilon}
\asymp
\frac{L\tau^2\sigma^2}{\varepsilon},
\]
where the first relation uses~\eqref{eq:endpoint-objective-scale-main} and the final relation uses
\eqref{eq:endpoint-chain-length-main}. Thus the time needed to reveal enough coordinates has exactly the claimed
parameter dependence.

If instead the value in~\eqref{eq:endpoint-qstar-main} satisfies \(q_\star>1\), set \(q=1\). Since \(q\) is the
reveal probability, \(q=1\) means that every available coordinate is revealed, and the oracle returns the exact
gradient with zero variance. In this case, the inequality \(q_\star>1\), together with
\eqref{eq:endpoint-objective-scale-main} and~\eqref{eq:endpoint-chain-length-main}, implies
\[
\frac{L\tau^2\sigma^2}{\varepsilon}\lesssim m.
\]
Hence the \(m\)-call exact-gradient barrier is already long enough to prove the claimed lower bound.

\emph{Arbitrary randomized algorithms.}
Without random rotation, an algorithm allowed to choose arbitrary query points could directly target hidden chain
directions. The rotation \(U\) prevents this. The residual-gap bound in
\eqref{eq:endpoint-normalized-barrier-main} is obtained after averaging over \(U\), the oracle randomness, and the
algorithm's internal randomness (see Appendix~\ref{app:sto-clipping}, Lemma~\ref{lem:clipped-rotated}(5)). For a fixed
algorithm, let \(R(U)\) denote its expected residual gap on the instance with rotation \(U\). Then
\[
\sup_U R(U)\ge\E_U R(U),
\]
so some fixed rotation is at least as hard as the average rotation. This proves the lower bound for randomized,
adaptive algorithms with arbitrary query points.

\textbf{Case 2: $1/2\le L\tau\le\kappa_0$ (the bounded-$L\tau$ regime).}
The zero-chain construction may not have the minimum required chain length in this range. Instead, set
$\mu:=(2\tau)^{-1}$ and consider the two one-dimensional pairs
\begin{equation}
\label{eq:endpoint-gaussian-pair-main}
f_v(x):=\frac{\mu}{2}(x-v)^2,
\qquad
G_v(x,\xi):=\mu(x-v)+\xi,
\qquad
v\in\left\{-4\sqrt{\frac{\varepsilon}{\mu}},\,
4\sqrt{\frac{\varepsilon}{\mu}}\right\},
\quad \xi\sim\mathcal N(0,\sigma^2).
\end{equation}
Because $L\tau\ge1/2$, we have $\mu\le L$, so each $f_v$ is $L$-smooth. Moreover,
$f_v(x)-(f_v)_\star=\tau\abs{f_v'(x)}^2$, its gap at the initial point $0$ is
$8\varepsilon\le\Delta_0$, and the oracle in
\eqref{eq:endpoint-gaussian-pair-main} is unbiased with variance $\sigma^2$.

For any adaptive algorithm making $T$ calls, let $P_+^T$ and $P_-^T$ denote the probability distributions of the
complete query--response history, including the algorithm's internal randomness, under the positive and negative
choices of $v$, respectively. At each call, conditional on the preceding history, the two Gaussian oracle responses
have means separated by $8\sqrt{\mu\varepsilon}$, independently of the current query. Hence the chain rule for
relative entropy gives
\begin{equation}
\label{eq:endpoint-gaussian-kl-main}
\mathrm{KL}(P_+^T\Vert P_-^T)
=
T\frac{(8\sqrt{\mu\varepsilon})^2}{2\sigma^2}
=
\frac{32T\mu\varepsilon}{\sigma^2}.
\end{equation}
The complete derivation of~\eqref{eq:endpoint-gaussian-kl-main}
is given in Appendix~\ref{app:sto-endpoint-proof}.
If $T\le\sigma^2/(64\mu\varepsilon)$, then the KL divergence in
\eqref{eq:endpoint-gaussian-kl-main} is at most $1/2$. Let $\widehat x_T$ be the algorithm's output, and use it to
guess which of the two instances generated the oracle responses: guess the positive choice of $v$ when
$\widehat x_T\ge0$, and the negative choice otherwise. Pinsker's inequality gives
\[
\norm{P_+^T-P_-^T}_{\rm TV}
\le
\sqrt{\frac{\mathrm{KL}(P_+^T\Vert P_-^T)}{2}}
\le \frac12.
\]
For the uniform choice of the two instances, Le Cam's inequality~\citep{Yu1997} gives an error probability of at least
$(1-\norm{P_+^T-P_-^T}_{\rm TV})/2\ge1/4$. If this guess is wrong, $\widehat x_T$ is at least
$4\sqrt{\varepsilon/\mu}$ from the true minimizer, so its objective gap is at least $8\varepsilon$. Since the
objective gap is always nonnegative, its expectation under the uniform mixture of the two instances is at least
\[
(8\varepsilon)\,\Pbb(\text{wrong guess})
\ge
(8\varepsilon)\frac14
=2\varepsilon.
\]
This mixture expectation is the average of the expected gaps on the two instances. Therefore, at least one of them has
expected gap at least $2\varepsilon$. Thus, for every integer
$T\le\sigma^2/(64\mu\varepsilon)$, no algorithm using $T$ oracle calls can guarantee expected error at most
$\varepsilon$ uniformly over the class. By the definition of
$T_\varepsilon^{\mathrm{BV,glob}}$, this implies
\[
T_\varepsilon^{\mathrm{BV,glob}}(L,\tau,\Delta_0,\sigma^2)
\ge
\frac{\sigma^2}{64\mu\varepsilon}
=
\frac{\tau\sigma^2}{32\varepsilon}.
\]
The final equality uses $\mu=(2\tau)^{-1}$.
Since \(L\tau\le\kappa_0\) in this bounded range, this implies
\[
\Omega\left(\frac{L\tau^2\sigma^2}{\varepsilon}\right)
\]
after changing only the universal constant. Thus the probabilistic zero chain is the principal endpoint construction;
the Gaussian argument only completes the bounded range of \(L\tau\).

Combining the two parameter ranges proves
\[
T_\varepsilon^{\mathrm{BV,glob}}(L,\tau,\Delta_0,\sigma^2)
=
\Omega\left(\frac{L\tau^2\sigma^2}{\varepsilon}\right).
\]
The complete proof is given in Appendix~\ref{app:sto-endpoint-proof}.
\end{proof}

\begin{remark}[Why $L\tau\ge 1/2$ is necessary]
The condition $L\tau\ge 1/2$ in Theorem~\ref{thm:stochastic-pl-endpoint} is necessary for the class to contain a
nonconstant hard instance. Every lower-bounded $L$-smooth function satisfies
\[
\norm{\nabla f(x)}^2\le 2L\bigl(f(x)-f_\star\bigr),
\]
whereas the $2$-P\L{} condition gives
\[
f(x)-f_\star\le \tau\norm{\nabla f(x)}^2.
\]
Given $f(x)>f_\star$ at some point, combining these inequalities and dividing by the positive gap yields
$1\le 2L\tau$. Thus $L\tau\ge 1/2$ is necessary.
\end{remark}

\begin{remark}[Scope of the endpoint result]
Theorem~\ref{thm:stochastic-pl-endpoint} proves the variance-dependent contribution
$\Omega(L\tau^2\sigma^2/\varepsilon)$ within the globally smooth, globally P\L{} class. It does not prove the complete
randomized endpoint rate
$L\tau\log(\Delta_0/\varepsilon)+L\tau^2\sigma^2/\varepsilon$. In particular, the deterministic lower bound of
\citet{YueFangLin2023} does not by itself establish the logarithmic term for randomized methods. The proved
variance-dependent term matches the corresponding SGD upper bound under global P\L{}
geometry~\citep{FatkhullinEtesamiHeKiyavash2022}.
\end{remark}

\begin{remark}[Higher-order regularity and the hard instances]
\label{rem:higher-order-regularity}
Remark~2.21 of \citet{RebjockBoumal2024} contains a local phenomenon that is useful to keep separate from the
present lower bounds. Under stronger local regularity, namely a locally Lipschitz Hessian, a local
\(\alpha\)-P\L{} inequality with \(\alpha\in(3/2,2]\) implies the classical local \(2\)-P\L{} condition; at the
endpoint \(\alpha=3/2\), their statement requires a \(C^3\) assumption.
This does not conflict with either hard-instance construction in this paper. Our oracle class assumes globally
Lipschitz gradients, but it does not assume a Lipschitz Hessian. The deterministic lower bound in
Appendix~\ref{app:deterministic} uses a piecewise quadratic scalar block whose derivative is Lipschitz but whose second
derivative has jumps. The stochastic lower bound in Appendix~\ref{app:stochastic} uses piecewise \(C^1\) scalar blocks;
indeed, the proof only uses the Hessian of the base chain where it exists almost everywhere. Thus the hard instances
belong to the globally smooth \(C^1\) class with Lipschitz gradients studied here, not to a Hessian-Lipschitz subclass. In a smaller
Hessian-Lipschitz class, the local \(2\)-P\L{} upgrade may change the sufficiently small-accuracy regime; our lower
bounds do not claim optimality for that stronger class.
\end{remark}

\section{Conclusion and future work}
\label{sec:conclusion}

This paper shows why the global model is the wrong starting point for \(\alpha<2\): global smoothness and a global
\(\alpha\)-P\L{} inequality force the objective to be constant. The useful replacement studied here is the
sublevel-\(\alpha\)-P\L{} model, where the growth inequality holds only on the initial sublevel set preserved by monotone
descent. This localization is still strong enough to support a sharp oracle-complexity theory. In the deterministic
setting, the lower bound matches the GD upper bound. In the bounded-variance stochastic setting, the lower
bound matches the sum of the GD-type descent term and the SGD noise term under trajectory containment. At the classical P\L{} endpoint, a separate construction proves the variance-dependent lower bound
$\Omega(L\tau^2\sigma^2/\varepsilon)$ even on the globally P\L{} class, exhibiting an additional factor $L\tau$
beyond the lower bound inherited from strongly convex objectives.

The deterministic and stochastic constructions share the same mechanism. The hard instances hide a tail of unexplored
coordinates: the tail must be large enough to keep the objective error above \(\varepsilon\), but shaped carefully enough
that the sublevel \(\alpha\)-P\L{} inequality remains valid on the region that the algorithm can reach.

Three questions remain open. First, for the gradient-only bounded-variance oracle studied in this paper, Theorem~\ref{thm:stochastic-matching}
proves the lower bound $\Omega\left(
L\tau^{2/\alpha}\varepsilon^{-(2-\alpha)/\alpha}
 + L\sigma^2\tau^{4/\alpha}\varepsilon^{-(4-\alpha)/\alpha}
\right),$ where the second term is the noise-dominated contribution, relevant when
\(\varepsilon\lesssim\tau\sigma^\alpha\). This matches the corresponding SGD upper-bound rate. Existing upper bounds
with this rate, however,
require the trajectory to remain inside the region where the \(\alpha\)-P\L{} geometry is valid. A fully
self-contained upper bound on the present sublevel class would need to prove this containment directly, without exact
function-value queries or a monotone accept/reject safeguard.
Second, Theorem~\ref{thm:stochastic-pl-endpoint} proves the
$L\tau^2\sigma^2/\varepsilon$ variance-dependent contribution at $\alpha=2$ even under a global P\L{} inequality.
Establishing the logarithmic exact-gradient contribution for randomized methods remains a separate open problem.
Third, for the stronger average-smooth stochastic oracle used in variance-reduced methods, the local lower-bound
mechanism of \citet{MasihaSalehkaleybarHeKiyavashThiran2026} suggests
\(\Omega(\sigma^2\tau^{2/\alpha}\varepsilon^{-2/\alpha})\) as a candidate noise lower bound. It is natural to ask
whether that mechanism can be adapted rigorously to the present sublevel class with the full dependence on the class
parameters. If such an adaptation is established, combining it with Theorem~\ref{thm:matching} would yield the
candidate sum
\(\Omega(L\tau^{2/\alpha}\varepsilon^{-(2-\alpha)/\alpha}+\sigma^2\tau^{2/\alpha}\varepsilon^{-2/\alpha})\).
The average-smooth upper bound of \citet{FatkhullinEtesamiHeKiyavash2022}, after translating constants to our convention, gives
\(O((\sigma^2\tau^{2/\alpha}+L^2\tau^{4/\alpha})\varepsilon^{-2/\alpha})\). The remaining gap is the
\(L^2\tau^{4/\alpha}\varepsilon^{-2/\alpha}\) term: it is unclear whether this term is minimax intrinsic or an
artifact of current variance-reduced analyses, especially because it does not vanish in the exact-gradient limit
\(\sigma=0\).

\bibliographystyle{plainnat}
\bibliography{alpha_pl_first_order_complexity_unified_revised}

\appendix

\paragraph{Appendix notation.}
For a matrix \(A\), \(\norm{A}_{\mathrm{op}}\) denotes its Euclidean operator norm. For symmetric matrices, \(A\preceq B\)
denotes the Loewner order. We write \(\one_N=(1,\ldots,1)\in\R^N\), \(I_N\) for the \(N\times N\) identity matrix,
\(\supp(x):=\{i:\ x_i\ne0\}\), and \(B(0,R):=\{x:\norm{x}<R\}\). For \(U\in\R^{d\times D}\),
\[
\Ortho(d,D):=\{U\in\R^{d\times D}:\ U^\top U=I_D\}.
\]
For \(x\in\R^D\) and \(\delta\ge0\), the progress index is
\[
\prog_\delta(x):=\max\bigl(\{i\in\{1,\ldots,D\}:\ |x_i|>\delta\}\cup\{0\}\bigr).
\]
Thus \(\prog_0(x)\) is the largest nonzero coordinate index, while \(\prog_{1/4}(x)\) counts only coordinates whose
magnitude has exceeded \(1/4\). We use \(\mathrm{Lip}(h)\) for the global Lipschitz constant of a map \(h\), and
\(J\phi(x)\) for the Jacobian matrix of a differentiable map \(\phi\) at \(x\). The symbol \(\mathcal Z\) denotes an
arbitrary sample space for stochastic oracles. All stochastic expectations and probabilities in the appendices are taken
over the oracle randomness, algorithmic randomness, and random rotations when these are present.

\section{A nonconvex example at \texorpdfstring{\(\alpha=1\)}{alpha=1}}
\label{app:nonconvex-alpha1}

This appendix shows that the endpoint class \(\alpha=1\) strictly contains bounded-sublevel smooth convex
minimization. The example is globally smooth and satisfies the sublevel \(\alpha=1\)-P\L{} inequality, but it is nonconvex
on its initial sublevel set.

\begin{example}[A nonconvex globally smooth sublevel-\(\alpha=1\)-P\L{} function]
\label{ex:nonconvex-alpha1}
Fix \(a>1/\sqrt 3\) and define
\[
f_a(x):=\frac{(x-a)^2}{1+(x-a)^2},
\qquad x\in\R,
\qquad x_0=0.
\]
Then \(f_{a,\star}=0\), attained at \(x_\star=a\), and \(f_a\) is globally smooth. Moreover,
\[
S_0(f_a)=\{x:\ |x-a|\le a\},
\]
and on \(S_0(f_a)\),
\[
f_a(x)-f_{a,\star}
\le
\frac{a(1+a^2)}{2}\,\abs{f_a'(x)} .
\]
Thus \(f_a\in\mathcal F^{(1)}_{1,L_a,a(1+a^2)/2,\Delta_0}(0)\) for some finite smoothness constant \(L_a\) and every
\(\Delta_0\ge f_a(0)-f_{a,\star}\). However, \(f_a\) is not convex on \(S_0(f_a)\).
\end{example}

\begin{proof}
Write \(r:=x-a\). Then
\[
f_a(x)=\frac{r^2}{1+r^2},
\qquad
f_a'(x)=\frac{2r}{(1+r^2)^2},
\qquad
f_a''(x)=\frac{2(1-3r^2)}{(1+r^2)^3}.
\]
The derivatives are bounded on \(\R\), so \(f_a\) has globally Lipschitz gradient with some finite constant \(L_a\).
Also \(f_{a,\star}=0\), attained at \(r=0\), i.e. at \(x=a\).

The function \(r\mapsto r^2/(1+r^2)\) is increasing in \(|r|\). Since the initial point is \(0\), for which
\(|r|=a\), the initial sublevel set is
\[
S_0(f_a)=\{x:\ |x-a|\le a\}.
\]
For \(x\in S_0(f_a)\), we have \(|r|\le a\). If \(r\ne0\), then
\[
\frac{f_a(x)-f_{a,\star}}{\abs{f_a'(x)}}
=
\frac{r^2/(1+r^2)}{2|r|/(1+r^2)^2}
=
\frac{|r|(1+r^2)}{2}
\le
\frac{a(1+a^2)}{2}.
\]
The same sublevel \(\alpha=1\)-P\L{} inequality is trivial at \(r=0\). Therefore
\[
f_a(x)-f_{a,\star}
\le
\frac{a(1+a^2)}{2}\abs{f_a'(x)}
\qquad
(x\in S_0(f_a)).
\]

Finally, because \(a>1/\sqrt3\), the sublevel set contains points with \(|r|>1/\sqrt3\). At those points,
\[
f_a''(x)=\frac{2(1-3r^2)}{(1+r^2)^3}<0,
\]
so \(f_a\) is not convex on \(S_0(f_a)\).
\end{proof}

\section{A concrete intermediate-exponent example}
\label{app:explicit-intermediate}

This appendix records an explicit globally smooth member of
\(\mathcal F_{\alpha,L,\tau,\Delta_0}(0)\) for each \(\alpha\in(1,2)\). The construction is not needed for the main
argument in Section~\ref{sec:preliminaries}; it is included only to show that the intermediate-exponent class is
nonempty even under global smoothness.

\begin{example}[A globally smooth sublevel-\(\alpha\)-P\L{} function for every \(\alpha\in(1,2)\)]
\label{ex:explicit-intermediate}
Fix \(p>2\), set \(\alpha:=p/(p-1)\in(1,2)\), and define \(\phi_p:[0,\infty)\to[0,\infty)\) by
\[
\phi_p(r):=
\begin{cases}
\dfrac{r^p}{p}, & 0\le r\le 1,\\[0.7em]
A_p+B_p\,r+\dfrac{p-1}{2}\,r^2, & r>1,
\end{cases}
\qquad
A_p:=\frac{1}{p}+\frac{p}{2}-\frac{3}{2},
\qquad
B_p:=2-p .
\]
Fix \(d\ge 1\) and \(x_\star\in\R^d\) with \(0<\norm{x_\star}\le 1\), and set
\[
f(x):=\phi_p\bigl(\norm{x-x_\star}\bigr),
\qquad x\in\R^d.
\]
Then \(f_\star=0\), \(f_\star\) is attained at \(x_\star\), \(\nabla f\) is globally \((p-1)\)-Lipschitz, and
\[
S_0(f)=\{x:\norm{x-x_\star}\le \norm{x_\star}\}\subseteq \{x:\norm{x-x_\star}\le 1\}.
\]
Moreover,
\[
f(x)-f_\star
=
\frac1p\norm{\nabla f(x)}^{p/(p-1)}
\qquad
\bigl(x\in S_0(f)\bigr).
\]
Consequently,
\[
f\in \mathcal F^{(d)}_{\alpha,p-1,1/p,\Delta_0}(0)
\qquad
\text{for every }\Delta_0\ge \frac{\norm{x_\star}^p}{p}.
\]
\end{example}

\begin{proof}
The constants \(A_p,B_p\) make the two pieces match to second order at \(r=1\). Indeed,
\[
A_p+B_p+\frac{p-1}{2}=\frac1p,\qquad
B_p+(p-1)=1,\qquad
p-1=p-1.
\]
Thus \(\phi_p\in C^2([0,\infty))\). Also \(\phi_p(0)=0\), \(\phi_p(r)>0\) for \(r>0\), and
\(\phi_p'(r)>0\) for \(r>0\). Hence \(f_\star=0\), attained at \(x_\star\), and
\[
f(0)-f_\star=\phi_p(\norm{x_\star})=\frac{\norm{x_\star}^p}{p}.
\]

For a radial function \(x\mapsto h(\norm{x-x_\star})\), the Hessian at \(x\ne x_\star\) has eigenvalues \(h''(r)\)
in the radial direction and \(h'(r)/r\) in the tangential directions, where \(r=\norm{x-x_\star}\). For \(h=\phi_p\),
\[
\frac{\phi_p'(r)}{r}=r^{p-2},\qquad \phi_p''(r)=(p-1)r^{p-2}
\qquad (0<r\le 1),
\]
and
\[
\frac{\phi_p'(r)}{r}=(p-1)-\frac{p-2}{r},\qquad \phi_p''(r)=p-1
\qquad (r>1).
\]
All these eigenvalues lie in \([0,p-1]\), and the first pair tends to \(0\) as \(r\downarrow0\). Therefore
\(\nabla^2 f\) extends continuously to all of \(\R^d\) and
\(\norm{\nabla^2 f(x)}_{\mathrm{op}}\le p-1\). Hence \(\nabla f\) is globally \((p-1)\)-Lipschitz.

Since \(\phi_p\) is strictly increasing,
\[
S_0(f)
=
\{x:\phi_p(\norm{x-x_\star})\le \phi_p(\norm{x_\star})\}
=
\{x:\norm{x-x_\star}\le \norm{x_\star}\}.
\]
Because \(\norm{x_\star}\le1\), every point in \(S_0(f)\) lies in the region where
\(\phi_p(r)=r^p/p\). Thus, for \(x\in S_0(f)\),
\[
f(x)-f_\star=\frac{\norm{x-x_\star}^p}{p},
\qquad
\norm{\nabla f(x)}=\norm{x-x_\star}^{p-1}.
\]
The displayed sublevel-\(\alpha\)-P\L{} identity follows, including at \(x=x_\star\), where both sides are zero.
\end{proof}

\section{Detailed deterministic upper bound and lower bound}
\label{app:deterministic}

This appendix contains the deterministic arguments supporting Section~\ref{sec:det}.
Subsection~\ref{app:gd-upper} records the gradient-descent upper bound used in Proposition~\ref{prop:gd-main}.
Subsections~\ref{app:det-ingredients}--\ref{app:det-zr-to-all} then prove Theorem~\ref{thm:matching}: the sharp lower
bound from the rescaled Yue--Fang--Lin zero chain, transferred from zero-respecting methods to all deterministic
first-order methods.

\subsection{Gradient-descent upper bound}
\label{app:gd-upper}

We first record the gradient-descent upper bound. For \(\alpha=2\), this is the standard P\L{} descent argument; see, for example, \citet{KarimiNutiniSchmidt2016}. The same proof works for sublevel \(\alpha\)-P\L{} functions: the \(1/L\) gradient-descent step is monotone, so all iterates remain in the initial sublevel set, exactly where the \(\alpha\)-P\L{} inequality is assumed to hold.

\begin{proposition}[Gradient-descent upper bound under sublevel $\alpha$-P\L{}]
\label{prop:gd-recursion}
Let $f:\R^d\to\R$ be differentiable and globally $L$-smooth, i.e.,
\[
\norm{\nabla f(x)-\nabla f(y)}\le L\norm{x-y}
\qquad (\forall x,y\in\R^d).
\]
Assume also that
\[
f(x)-f_\star\le \tau \norm{\nabla f(x)}^\alpha
\qquad (\forall x\in S_0(f))
\]
for some $\tau>0$ and some $\alpha\in[1,2]$.
Let gradient descent with step size $1/L$ be given by
\[
x_{k+1}=x_k-\frac1L\nabla f(x_k),
\qquad
x_0=0.
\]
Define
\[
\gap_k:=f(x_k)-f_\star.
\]
Then every iterate belongs to $S_0(f)$. Moreover, if $\alpha=2$, then
\[
\gap_k\le \left(1-\frac{1}{2L\tau}\right)^k\gap_0
\qquad (\forall k\ge 0),
\]
while if $\alpha\in[1,2)$, then
\[
\gap_k
\le
\left(
\gap_0^{-(2-\alpha)/\alpha}
+
\frac{2-\alpha}{2\alpha L\tau^{2/\alpha}}\,k
\right)^{-\alpha/(2-\alpha)}
\qquad (\forall k\ge 0).
\]
\end{proposition}

\begin{proof}
By global $L$-smoothness and the gradient-descent update,
\[
\begin{aligned}
f(x_{k+1})
&\le
f(x_k)+\ip{\nabla f(x_k)}{x_{k+1}-x_k}+\frac{L}{2}\norm{x_{k+1}-x_k}^2\\
&=
f(x_k)-\frac1L\norm{\nabla f(x_k)}^2+\frac{1}{2L}\norm{\nabla f(x_k)}^2\\
&=
f(x_k)-\frac{1}{2L}\norm{\nabla f(x_k)}^2\\
&\le
f(x_k).
\end{aligned}
\]
Therefore $f(x_{k+1})\le f(x_k)\le f(0)$ by induction, so every iterate lies in $S_0(f)$.

Subtracting $f_\star$ from the descent inequality and using the $\alpha$-P\L{} bound on $S_0(f)$ gives
\[
\gap_{k+1}
\le
\gap_k-\frac{1}{2L}\norm{\nabla f(x_k)}^2
\le
\gap_k-\frac{1}{2L\tau^{2/\alpha}}\gap_k^{2/\alpha}.
\]

If $\alpha=2$, this becomes
\[
\gap_{k+1}\le \left(1-\frac{1}{2L\tau}\right)\gap_k.
\]
If \(\gap_0=0\), the claim is trivial. If \(\gap_0>0\), the factor above is nonnegative; otherwise the inequality
with \(k=0\) would force \(\gap_1<0\), contradicting \(f_\star\le f(x_1)\). Iterating proves the claimed linear
estimate.

Now let $\alpha\in[1,2)$ and set
\[
r:=\frac{2}{\alpha}>1,
\qquad
c:=\frac{1}{2L\tau^{2/\alpha}}.
\]
Then
\[
\gap_{k+1}\le \gap_k-c\gap_k^r.
\]
Since $\gap_{k+1}\le \gap_k$, the sequence is nonincreasing. If $\gap_k=0$ for some $k$, then the recursion yields
$\gap_j=0$ for all $j\ge k$, and the desired bound is immediate. If \(\gap_k>0\) but \(\gap_{k+1}=0\), then the
desired bound at time \(k+1\) and all later times is also immediate. It therefore suffices to consider indices with
\(\gap_k>0\) and \(\gap_{k+1}>0\).

The function $\phi(s)=s^{1-r}$ is convex on $(0,\infty)$, hence
\[
\phi(\gap_{k+1})
\ge
\phi(\gap_k)+\phi'(\gap_k)(\gap_{k+1}-\gap_k)
=
\gap_k^{1-r}+(r-1)\gap_k^{-r}(\gap_k-\gap_{k+1}).
\]
Using $\gap_k-\gap_{k+1}\ge c\gap_k^r$, we obtain
\[
\gap_{k+1}^{1-r}\ge \gap_k^{1-r}+c(r-1).
\]
Summing from $0$ to $k-1$ gives
\[
\gap_k^{1-r}\ge \gap_0^{1-r}+c(r-1)k.
\]
Since $1-r<0$, raising both sides to the power $1/(1-r)$ yields
\[
\gap_k\le \left(\gap_0^{1-r}+c(r-1)k\right)^{1/(1-r)}.
\]
Substituting $r=2/\alpha$ and $c=(2L\tau^{2/\alpha})^{-1}$ proves the stated bound.
\end{proof}

\begin{corollary}[Fixed-budget minimax upper bound]
\label{cor:upper}
For every $\alpha\in[1,2)$, every $L,\tau,\Delta_0>0$, and every $T\ge 0$,
\[
\mathcal E_T\!\left(\mathcal F_{\alpha,L,\tau,\Delta_0}(0)\right)
\le
\left(
\Delta_0^{-(2-\alpha)/\alpha}
+
\frac{2-\alpha}{2\alpha L\tau^{2/\alpha}}\,T
\right)^{-\alpha/(2-\alpha)}.
\]
\end{corollary}

\begin{proof}
Gradient descent with step size $1/L$ is a deterministic first-order method.

Fix $f\in\mathcal F^{(d)}_{\alpha,L,\tau,\Delta_0}(0)$. If $f(0)=f_\star$, monotonicity of gradient descent and the
definition of $f_\star$ give
\[
f_\star\le f(x_k)\le f(0)=f_\star
\qquad (k\ge0),
\]
so every objective gap is zero and the claimed bound holds. Suppose henceforth that $f(0)-f_\star>0$.
Proposition~\ref{prop:gd-recursion} applies on the initial sublevel set $S_0(f)$ and yields
\[
f(x_T)-f_\star
\le
\left(
(f(0)-f_\star)^{-(2-\alpha)/\alpha}
+
\frac{2-\alpha}{2\alpha L\tau^{2/\alpha}}\,T
\right)^{-\alpha/(2-\alpha)}.
\]
Since $f(0)-f_\star\le \Delta_0$, the right-hand side is at most
\[
\left(
\Delta_0^{-(2-\alpha)/\alpha}
+
\frac{2-\alpha}{2\alpha L\tau^{2/\alpha}}\,T
\right)^{-\alpha/(2-\alpha)}.
\]
Taking the supremum over all $d$ and all $f$ in the class, then the infimum over all deterministic methods, proves the claim.
\end{proof}

\begin{corollary}[Iteration complexity upper bound]
\label{cor:upper-complexity}
There exists a constant $C_\alpha>0$ such that, for every $\varepsilon\in(0,\Delta_0]$,
\[
T_\varepsilon\!\left(\mathcal F_{\alpha,L,\tau,\Delta_0}(0)\right)
\le
C_\alpha L\tau^{2/\alpha}\varepsilon^{-(2-\alpha)/\alpha}
\]
suffices to guarantee objective error at most $\varepsilon$ on the class $\mathcal F_{\alpha,L,\tau,\Delta_0}(0)$.
\end{corollary}

\begin{proof}
By Corollary~\ref{cor:upper}, it suffices that
\[
\left(
\Delta_0^{-(2-\alpha)/\alpha}
+
\frac{2-\alpha}{2\alpha L\tau^{2/\alpha}}\,T
\right)^{-\alpha/(2-\alpha)}
\le
\varepsilon.
\]
Equivalently,
\[
\Delta_0^{-(2-\alpha)/\alpha}
+
\frac{2-\alpha}{2\alpha L\tau^{2/\alpha}}\,T
\ge
\varepsilon^{-(2-\alpha)/\alpha}.
\]
Since $\varepsilon\le \Delta_0$, the difference between the right-hand side and the first term is nonnegative, and therefore
\[
T
\ge
\frac{2\alpha}{2-\alpha}
L\tau^{2/\alpha}
\left(
\varepsilon^{-(2-\alpha)/\alpha}-\Delta_0^{-(2-\alpha)/\alpha}
\right)
\]
is sufficient. This is bounded above by a constant multiple of
\[
L\tau^{2/\alpha}\varepsilon^{-(2-\alpha)/\alpha},
\]
which proves the claim.
\end{proof}


\subsection{Ingredients for the deterministic lower bound}
\label{app:det-ingredients}
The lower-bound proof uses the following two known facts.

\paragraph{(A) Quadratic-P\L{} estimate for the Yue block.}
Let
\[
g_m(u):=q_m(u)+\sum_{i=1}^m v(u_i)
\]
be the one-block Yue--Fang--Lin instance defined below.
There exists a universal constant $C_{\mathrm{PL}}\ge 1$ such that
\begin{align}\label{eq_PL2_YFL}
    g_m(u)-g_{m,\star}
\le
C_{\mathrm{PL}}\,m\,\norm{\nabla g_m(u)}^2
\qquad (\forall u\in\R^m,\ \forall m\ge 1).
\end{align}
This is the $t=1$ case of \citet[Lemma~2]{YueFangLin2023}.

\paragraph{(B) Reduction to zero-respecting methods.}
On an orthogonally invariant function class, any lower bound proved for zero-respecting deterministic first-order methods transfers to all deterministic first-order methods.
We use the standard deterministic-to-zero-respecting reduction in the form stated by
\citet[Proposition~1 and the surrounding discussion]{CarmonDuchiHinderSidford2021}.

\paragraph{Roadmap for the deterministic lower bound.}
The lower bound argument has four ingredients. We first construct a normalized one-block hard instance with a hidden tail of
coordinates. We then prove a zero-chain property that implies a zero-respecting method can reveal at most one new coordinate
per iteration, so after $T$ steps a tail of order $T$ remains untouched. Next, we tune the amplitude of the construction so that
this hidden tail already forces the objective error to be of order \(T^{-\alpha/(2-\alpha)}\) while preserving a uniform sublevel
$\alpha$-P\L{} inequality. Finally, orthogonal invariance allows us to extend the lower bound from zero-respecting methods to all
deterministic first-order methods.

\subsection{A hard one-block instance and its basic properties}
\label{app:det-one-block}

Define the scalar function $v:\R\to\R$ by
\[
v(s)=
\begin{cases}
\dfrac12 s^2, & s\le \dfrac{31}{32},\\[0.8em]
\dfrac12 s^2-16\Bigl(s-\dfrac{31}{32}\Bigr)^2, & \dfrac{31}{32}<s\le 1,\\[0.8em]
\dfrac12 s^2-\dfrac{1}{32}+16\Bigl(s-\dfrac{33}{32}\Bigr)^2, & 1<s\le \dfrac{33}{32},\\[0.8em]
\dfrac12 s^2-\dfrac{1}{32}, & s>\dfrac{33}{32}.
\end{cases}
\]
For each $m\ge 1$, define
\[
q_m(u):=\frac12 u_1^2+\frac12\sum_{i=1}^{m-1}(u_{i+1}-u_i)^2,
\qquad
u\in\R^m,
\]
and
\[
g_m(u):=q_m(u)+\sum_{i=1}^m v(u_i).
\]

\begin{lemma}
\label{lem:v-basic}
The function $v$ satisfies:
\begin{enumerate}[label=(\roman*),leftmargin=1.5em]
    \item $v(s)\ge 0$ for every $s\in\R$;
    \item $v(1)=31/64$;
    \item $v'(1)=0$;
    \item $v'$ is globally $33$-Lipschitz.
\end{enumerate}
\end{lemma}

\begin{proof}
Define
\[
b(s):=
\begin{cases}
1-32|s-1|, & \dfrac{31}{32}\le s\le \dfrac{33}{32},\\[0.4em]
0, & \text{otherwise}.
\end{cases}
\]
The pieces match in value and first derivative at \(31/32\), \(1\), and \(33/32\), so \(v\) is differentiable.
Differentiating the piecewise formula for \(v\) gives
\[
v'(s)=s-b(s)
\qquad (\forall s\in\R).
\]
Since $b(1)=1$, we obtain $v'(1)=0$.

Evaluating either middle branch at $s=1$ gives
\[
v(1)=\frac12-16\left(1-\frac{31}{32}\right)^2=\frac12-\frac{1}{64}=\frac{31}{64}.
\]

We now prove nonnegativity. For $s\le 31/32$, we have $v(s)=s^2/2\ge 0$. For $31/32<s\le 1$,
\[
v(s)=\frac12 s^2-16\Bigl(s-\frac{31}{32}\Bigr)^2
\]
is a concave quadratic, so its minimum on $[31/32,1]$ is attained at an endpoint. Both endpoint values are positive. For $1<s\le 33/32$,
\[
v(s)=\frac12 s^2-\frac{1}{32}+16\Bigl(s-\frac{33}{32}\Bigr)^2\ge \frac12-\frac{1}{32}=\frac{15}{32}>0.
\]
For $s>33/32$,
\[
v(s)=\frac12 s^2-\frac{1}{32}>\frac{15}{32}>0.
\]
Thus $v(s)\ge 0$ for all $s$.

Finally, the slopes of the piecewise-affine derivative \(v'\) are \(1\), \(-31\), \(33\), and \(1\). Hence \(v'\) is globally \(33\)-Lipschitz.
\end{proof}

\begin{lemma}
\label{lem:q-basic}
For every $m\ge 1$, the function $q_m$ satisfies:
\begin{enumerate}[label=(\roman*),leftmargin=1.5em]
    \item $q_m(u)\ge 0$ for every $u\in\R^m$;
    \item $q_m(\one)=1/2$, where $\one=(1,\ldots,1)\in\R^m$;
    \item $\nabla q_m$ is globally $4$-Lipschitz;
    \item Define $B_m:=\nabla^2 q_m$, then
    $0\preceq B_m\preceq 4I.$
\end{enumerate}
\end{lemma}

	\begin{proof}
	Items (i) and (ii) are immediate from the sum-of-squares definition.
The case \(m=1\) is also immediate for the Hessian bound: \(q_1(u)=u_1^2/2\), so \(B_1=[1]\). Hence assume
\(m\ge2\).

	For the Hessian,
\[
(B_mu)_1=2u_1-u_2,
\]
\[
(B_mu)_i=2u_i-u_{i-1}-u_{i+1}
\qquad (2\le i\le m-1),
\]
\[
(B_mu)_m=u_m-u_{m-1}.
\]
Hence $B_m$ is the tridiagonal matrix with diagonal entries $(2,2,\ldots,2,1)$ and off-diagonal entries $-1$. Since $q_m$ is a sum of squares, $B_m\succeq 0$.
The absolute row sums are bounded by $3$ in the first row, $4$ in the interior rows, and $2$ in the last row, so Gershgorin's theorem yields
\[
\norm{B_m}_{\mathrm{op}}\le 4.
\]
Because $B_m$ is symmetric, this is equivalent to $B_m\preceq 4I$.
Therefore $\nabla q_m$ is globally $4$-Lipschitz.
\end{proof}

\begin{definition}
\label{def:zero-chain}
A differentiable function $h:\R^m\to\R$ is a \emph{first-order zero-chain} if
\[
\supp(x)\subseteq \{1,\ldots,k\}
\quad\Longrightarrow\quad
\supp(\nabla h(x))\subseteq \{1,\ldots,k+1\}
\qquad (\forall k\in\{0,\ldots,m-1\}).
\]
A deterministic first-order method $A$ is \emph{zero-respecting} if for every dimension $d$, every differentiable $f:\R^d\to\R$, and every $t\ge 1$,
\[
\supp\bigl(x_t^A(f)\bigr)
\subseteq
\bigcup_{s=0}^{t-1}\supp\bigl(\nabla f(x_s^A(f))\bigr).
\]
\end{definition}

\begin{lemma}
\label{lem:zero-chain}
For every $m\ge 1$, the function
\[
h_m(x):=g_m(\one-x)
\qquad (x\in\R^m)
\]
is a first-order zero-chain.
\end{lemma}

	\begin{proof}
The case \(m=1\) is immediate, so assume \(m\ge2\).
	Fix $x\in\R^m$ and set $u:=\one-x$.
From the definition of \(q_m\),
	\[
	(\nabla q_m(u))_1=2u_1-u_2,
\]
\[
(\nabla q_m(u))_i=2u_i-u_{i-1}-u_{i+1}
\qquad (2\le i\le m-1),
\]
\[
(\nabla q_m(u))_m=u_m-u_{m-1}.
\]
Also, by Lemma~\ref{lem:v-basic},
\[
\nabla h_m(x)=-\nabla q_m(u)-\bigl(v'(u_1),\ldots,v'(u_m)\bigr),
\qquad
v'(1)=0.
\]

Assume $\supp(x)\subseteq \{1,\ldots,k\}$. Then $x_j=0$ for all $j\ge k+1$, so $u_j=1$ for all $j\ge k+1$.
For every $j\ge k+2$, every neighboring coordinate appearing in $(\nabla q_m(u))_j$ is also equal to $1$, hence $(\nabla q_m(u))_j=0$.
Moreover, $v'(u_j)=v'(1)=0$ for every $j\ge k+1$.
Therefore,
	\[
	(\nabla h_m(x))_j=0
	\qquad (\forall j\ge k+2).
	\]
This is exactly \(\supp(\nabla h_m(x))\subseteq\{1,\ldots,k+1\}\), so \(h_m\) is a first-order zero-chain.
	\end{proof}

\subsection{The base hard family}
\label{app:det-base-family}

For each integer $T\ge 1$, define
\[
p:=\frac{\alpha}{2-\alpha},
\qquad
m_T:=2T+1,
\qquad
a_T:=m_T^{-1/(2-\alpha)}=m_T^{-(p+1)/2},
\]
and
\[
\Phi_T(x):=a_T^2\,g_{m_T}\!\left(\one_{m_T}-\frac{x}{a_T}\right),
\qquad
x\in\R^{m_T}.
\]
We also set
\[
\Delta_T:=\Phi_T(0)-\Phi_{T,\star}.
\]

\begin{proposition}
\label{prop:base-instance}
For every $T\ge 1$, the function $\Phi_T$ satisfies:
\begin{enumerate}[label=(\roman*),leftmargin=1.5em]
    \item $\Phi_{T,\star}=0$, attained at $x_T^\star=a_T\one_{m_T}$;
    \item $\nabla \Phi_T$ is globally $37$-Lipschitz;
    \item on the initial sublevel set $S_0(\Phi_T)$,
    \[
    \Phi_T(x)-\Phi_{T,\star}\le C_{\mathrm{PL}}\norm{\nabla \Phi_T(x)}^\alpha;
    \]
    \item
    \[
    \Delta_T
    =
    a_T^2\left(\frac12+\frac{31}{64}m_T\right)
    \le
    \frac{63}{64}(2T+1)^{-p}.
    \]
\end{enumerate}
\end{proposition}

\begin{proof}
Let
\[
u:=\one_{m_T}-x/a_T.
\]
Then
\[
\Phi_T(x)=a_T^2 g_{m_T}(u),
\qquad
\nabla \Phi_T(x)=-a_T\nabla g_{m_T}(u).
\]

Since $q_{m_T}\ge 0$, $v\ge 0$, and $g_{m_T}(0)=0$, we have $g_{m_T,\star}=0$. Hence
\[
	\Phi_T(a_T\one_{m_T})=a_T^2g_{m_T}(0)=0,
	\]
so $\Phi_{T,\star}=0$, proving (i). We next compute
	\[
	\Delta_T
=
\Phi_T(0)
=
a_T^2 g_{m_T}(\one_{m_T})
=
a_T^2\left(q_{m_T}(\one_{m_T})+m_Tv(1)\right)
=
a_T^2\left(\frac12+\frac{31}{64}m_T\right).
\]
Since $a_T\le 1$ and
\[
m_Ta_T^2=a_T^\alpha,
\]
we obtain
\[
\Delta_T
\le
\left(\frac12+\frac{31}{64}\right)a_T^\alpha
=
\frac{63}{64}a_T^\alpha
=
\frac{63}{64}(2T+1)^{-p}.
\]
This proves (iv).

Item (ii) follows directly from Lemmas~\ref{lem:q-basic} and~\ref{lem:v-basic}, which show that \(\nabla q_{m_T}\)
and \(v'\) are \(4\)- and \(33\)-Lipschitz, respectively.
Hence $\nabla g_{m_T}$ is globally $37$-Lipschitz and
\[
\begin{aligned}
\norm{\nabla \Phi_T(x)-\nabla \Phi_T(y)}
&=
a_T\norm{\nabla g_{m_T}(u_x)-\nabla g_{m_T}(u_y)}\\
&\le
37a_T\norm{u_x-u_y}
=
37\norm{x-y},
\end{aligned}
\]
where $u_x=\one_{m_T}-x/a_T$ and $u_y=\one_{m_T}-y/a_T$.

It remains to prove item (iii). Set
\[
s:=\norm{\nabla \Phi_T(x)}.
\]
By the quadratic-P\L{} inequality \eqref{eq_PL2_YFL} for $g_{m_{T}}$ stated in Appendix~\ref{app:det-ingredients}(A),
\[
\Phi_T(x)-\Phi_{T,\star}
=
a_T^2\bigl(g_{m_T}(u)-g_{m_T,\star}\bigr)
\le
C_{\mathrm{PL}}m_Ta_T^2\norm{\nabla g_{m_T}(u)}^2
=
C_{\mathrm{PL}}m_Ts^2.
\]
If $s\le a_T$, then since $m_T=a_T^{-(2-\alpha)}$,
\[
\Phi_T(x)-\Phi_{T,\star}
\le
C_{\mathrm{PL}}a_T^{-(2-\alpha)}s^2
\le
C_{\mathrm{PL}}s^\alpha.
\]

\noindent If $s\ge a_T$ and $x\in S_0(\Phi_T)$, then
\[
\Phi_T(x)-\Phi_{T,\star}
\le
\Phi_T(0)-\Phi_{T,\star}
=
\Delta_T
\le
\frac{63}{64}a_T^\alpha
\le
C_{\mathrm{PL}}s^\alpha,
\]
because \(C_{\mathrm{PL}}\ge 1\), \(s\ge a_T\), and item (iv).

\end{proof}

\begin{proposition}[Zero-respecting hardness of the base chain]
\label{prop:base-hardness-zr}
Let \(T\ge1\), and let \(x^{(T)}\) be the \(T\)-th iterate generated from \(0\) by any zero-respecting deterministic
first-order method applied to \(\Phi_T\). Then
\[
x^{(T)}_i=0
\qquad (i>T).
\]
Consequently, the unreached tail satisfies
\[
\Phi_T(x^{(T)})-\Phi_{T,\star}
\ge
\frac{31}{64}(T+1)(2T+1)^{-(p+1)}.
\]
In particular,
\[
\Phi_T(x^{(T)})-\Phi_{T,\star}
\ge
\frac{31}{64\cdot 3^{p+1}}\,T^{-p}.
\]
Moreover,
this residual gap is a constant fraction of the initial gap:
\[
\frac{\Phi_T(x^{(T)})-\Phi_{T,\star}}{\Delta_T}\ge \frac{62}{125}.
\]
\end{proposition}

\begin{proof}
By Lemma~\ref{lem:zero-chain}, the function
\[
h_{m_T}(z):=g_{m_T}(\one_{m_T}-z)
\]
is a zero-chain.
Since
	\[
	\Phi_T(x)=a_T^2 h_{m_T}(x/a_T),
	\]
	the function $\Phi_T$ is also a zero-chain, because scaling coordinates by the positive scalar $a_T$ does not change supports.
Because the method is zero-respecting and starts from \(0\), a standard induction gives
	\[
	\supp\bigl(x^{(t)}\bigr)\subseteq \{1,\ldots,t\}
	\qquad (0\le t\le T).
	\]
Thus \(x^{(T)}_i=0\) for every \(i>T\): after \(T\) iterations, the coordinates \(T+1,\ldots,m_T\) have not been reached.

Now set
\[
\nu:=\one_{m_T}-x^{(T)}/a_T.
\]
Then \(\nu_i=1\) for every unreached tail coordinate \(i>T\).
Since $q_{m_T}\ge 0$ and $v\ge 0$,
\[
\Phi_T(x^{(T)})
=
a_T^2\left(q_{m_T}(\nu)+\sum_{i=1}^{m_T}v(\nu_i)\right)
\ge
a_T^2\sum_{i=T+1}^{m_T}v(\nu_i)
=
a_T^2(m_T-T)v(1).
\]
Because $m_T=2T+1$ and $v(1)=31/64$, this gives
\[
\Phi_T(x^{(T)})-\Phi_{T,\star}
\ge
\frac{31}{64}(T+1)a_T^2
=
\frac{31}{64}(T+1)(2T+1)^{-(p+1)}.
\]
This proves the first lower bound.

Since $T+1\ge T$ and $2T+1\le 3T$ for $T\ge 1$,
\[
\Phi_T(x^{(T)})-\Phi_{T,\star}
\ge
\frac{31}{64}\,T\,(3T)^{-(p+1)}
=
\frac{31}{64\cdot 3^{p+1}}\,T^{-p}.
\]

\noindent Finally, using Proposition~\ref{prop:base-instance},
\[
\frac{\Phi_T(x^{(T)})-\Phi_{T,\star}}{\Delta_T}
\ge
\frac{\frac{31}{64}(T+1)}{\frac12+\frac{31}{64}(2T+1)}
=
\frac{31(T+1)}{62T+63}.
\]
For $T\ge 1$,
\[
\frac{31(T+1)}{62T+63}\ge \frac{62}{125}
\iff
125\cdot 31(T+1)-62(62T+63)=31(T-1)\ge 0.
\]
This proves the last claim.
\end{proof}

\subsection{A scaling lemma}
\label{app:det-scaling}

\begin{lemma}[Objective-space and domain-space scaling]
\label{lem:scaling}
Let $f:\R^d\to\R$ be differentiable, and let $\lambda,\rho>0$.
Define
\[
(\mathcal S_{\lambda,\rho}f)(x):=\lambda f(\rho x).
\]
Then:
\begin{enumerate}[label=(\roman*),leftmargin=1.5em]
    \item $(\mathcal S_{\lambda,\rho}f)_\star=\lambda f_\star$ and
    \[
    S_0(\mathcal S_{\lambda,\rho}f)=\rho^{-1}S_0(f);
    \]
    \item if $\nabla f$ is globally $L_f$-Lipschitz on $\R^d$, then $\nabla(\mathcal S_{\lambda,\rho}f)$ is globally $(\lambda\rho^2L_f)$-Lipschitz on $\R^d$;
    \item if
    \[
    f(y)-f_\star\le \tau_f \norm{\nabla f(y)}^\alpha
    \qquad (\forall y\in S_0(f)),
    \]
    then
    \[
    (\mathcal S_{\lambda,\rho}f)(x)-(\mathcal S_{\lambda,\rho}f)_\star
    \le
    \tau_f\lambda^{1-\alpha}\rho^{-\alpha}
    \norm{\nabla(\mathcal S_{\lambda,\rho}f)(x)}^\alpha
    \qquad (\forall x\in S_0(\mathcal S_{\lambda,\rho}f));
    \]
    \item if every zero-respecting deterministic first-order method started from \(0\) has objective gap at least \(R\)
    after \(T\) steps on \(f\), then every zero-respecting deterministic first-order method started from \(0\) has
    objective gap at least \(\lambda R\) after \(T\) steps on \(\mathcal S_{\lambda,\rho}f\).
\end{enumerate}
\end{lemma}

\begin{proof}
 (i) is immediate from the definition. To show (ii),
\[
\nabla(\mathcal S_{\lambda,\rho}f)(x)=\lambda\rho\,\nabla f(\rho x),
\]
so for arbitrary $x,y\in \R^d$,
\[
\norm{\nabla(\mathcal S_{\lambda,\rho}f)(x)-\nabla(\mathcal S_{\lambda,\rho}f)(y)}
=
\lambda\rho\,\norm{\nabla f(\rho x)-\nabla f(\rho y)}
\le
\lambda\rho^2L_f\norm{x-y}.
\]

To show (iii), note that if $x\in S_0(\mathcal S_{\lambda,\rho}f)$, then $\rho x\in S_0(f)$ from (i). Hence
\[
\begin{aligned}
(\mathcal S_{\lambda,\rho}f)(x)-(\mathcal S_{\lambda,\rho}f)_\star
&=
\lambda\bigl(f(\rho x)-f_\star\bigr)\\
&\le
\lambda\tau_f\norm{\nabla f(\rho x)}^\alpha\\
&=
\tau_f\lambda^{1-\alpha}\rho^{-\alpha}
\norm{\nabla(\mathcal S_{\lambda,\rho}f)(x)}^\alpha.
\end{aligned}
\]

We prove (iv) by contradiction.
Suppose there exists a zero-respecting deterministic first-order method $A$ such that, on the scaled function $g:=\mathcal S_{\lambda,\rho}f$, the $T$-th iterate satisfies
\[
g\bigl(x_T^A(g)\bigr)-g_\star<\lambda R.
\]
We build from \(A\) a zero-respecting deterministic first-order method \(B\) for \(f\) as follows.
At each simulated query \(x\) made by \(A\) on \(g\), method \(B\) queries \(f\) at \(y:=\rho x\) and feeds the synthetic oracle values
\[
g(x)=\lambda f(y),
\qquad
\nabla g(x)=\lambda\rho\,\nabla f(y)
\]
to the internal simulation of $A$.
If the simulated method $A$ outputs its next iterate $x^+$ on $g$, then $B$ outputs the iterate $y^+:=\rho x^+$ on $f$.
Since multiplication by nonzero scalars preserves supports, $B$ is zero-respecting whenever $A$ is zero-respecting.
By construction,
\[
y_t^B(f)=\rho x_t^A(g)
\qquad (\forall t),
\]
and therefore
\[
f\bigl(y_T^B(f)\bigr)-f_\star
=
\lambda^{-1}\bigl(g\bigl(x_T^A(g)\bigr)-g_\star\bigr)
<
R,
\]
which is a contradiction for $f$.
Thus every zero-respecting method on $\mathcal S_{\lambda,\rho}f$ has objective gap at least $\lambda R$ at time $T$.
\end{proof}

\subsection{\texorpdfstring{Zero-respecting lower bound at arbitrary $(L,\tau,\Delta_0)$}{Zero-respecting lower bound at arbitrary (L,tau,Delta0)}}
\label{app:det-zr}

\begin{theorem}
\label{thm:zr}
Let \(p:=\alpha/(2-\alpha)\). There exists a constant \(c_\alpha^{\mathrm{zr}}>0\) such that for every
\(L,\tau,\Delta_0>0\) and every \(T\ge 1\), there exists a function
\[
f_T\in \mathcal F^{(2T+1)}_{\alpha,L,\tau,\Delta_0}(0)
\]
with the following property: if \(x^{(T)}\) denotes the \(T\)-th iterate generated from \(0\) by any zero-respecting
deterministic first-order method applied to \(f_T\), then
\[
f_T(x^{(T)})-(f_T)_\star
\ge
c_\alpha^{\mathrm{zr}}
\min\!\left\{
\Delta_0,\,
\left(\frac{L\tau^{2/\alpha}}{T}\right)^p
\right\}.
\]
One may take
\[
c_\alpha^{\mathrm{zr}}
:=
\min\!\left\{
\frac{62}{125},\,
\frac{31}{64\cdot 3^{p+1}\cdot 37^p\cdot C_{\mathrm{PL}}^{2/(2-\alpha)}}
\right\}.
\]
\end{theorem}

\begin{proof}
Fix $L,\tau,\Delta_0>0$ and $T\ge 1$.
Start from the base instance $\Phi_T$ and define
\[
\lambda_T
:=
\min\!\left\{
\frac{\Delta_0}{\Delta_T},\,
\left(\frac{\tau}{C_{\mathrm{PL}}}\right)^{2/(2-\alpha)}
\left(\frac{L}{37}\right)^p
\right\},
\qquad
\rho_T:=\sqrt{\frac{L}{37\lambda_T}},
\]
and
\[
f_T:=\mathcal S_{\lambda_T,\rho_T}\Phi_T,
\qquad
\text{i.e.}\qquad
f_T(x)=\lambda_T\Phi_T(\rho_T x).
\]

We first verify that $f_T\in \mathcal F^{(2T+1)}_{\alpha,L,\tau,\Delta_0}(0)$.
By Lemma~\ref{lem:scaling}(ii) and Proposition~\ref{prop:base-instance}(ii),
\[
\nabla f_T
\text{ is }
\lambda_T\rho_T^2\cdot 37
=
L
\text{-Lipschitz on }\R^{2T+1}.
\]
Also,
\[
f_T(0)-(f_T)_\star=\lambda_T\Delta_T\le \Delta_0.
\]

By Lemma~\ref{lem:scaling}(iii) and Proposition~\ref{prop:base-instance}(iii), the function $f_T$ satisfies the $\alpha$-P\L{} inequality on $S_0(f_T)$ with constant
\[
C_{\mathrm{PL}}\lambda_T^{1-\alpha}\rho_T^{-\alpha}
=
C_{\mathrm{PL}}\lambda_T^{1-\alpha/2}\left(\frac{37}{L}\right)^{\alpha/2}.
\]
Since $1-\alpha/2>0$ and
\[
\lambda_T
\le
\left(\frac{\tau}{C_{\mathrm{PL}}}\right)^{2/(2-\alpha)}
\left(\frac{L}{37}\right)^p,
\]
we get
\[
\begin{aligned}
C_{\mathrm{PL}}\lambda_T^{1-\alpha/2}\left(\frac{37}{L}\right)^{\alpha/2}
&\le
C_{\mathrm{PL}}
\left[
\left(\frac{\tau}{C_{\mathrm{PL}}}\right)^{2/(2-\alpha)}
\left(\frac{L}{37}\right)^p
\right]^{1-\alpha/2}
\left(\frac{37}{L}\right)^{\alpha/2}\\
&=
C_{\mathrm{PL}}
\left(\frac{\tau}{C_{\mathrm{PL}}}\right)
\left(\frac{L}{37}\right)^{p(1-\alpha/2)-\alpha/2}.
\end{aligned}
\]
Now
\[
p\left(1-\frac{\alpha}{2}\right)-\frac{\alpha}{2}
=
\frac{\alpha}{2-\alpha}\cdot \frac{2-\alpha}{2}-\frac{\alpha}{2}
=
0,
\]
so the right-hand side is exactly $\tau$.
Thus $f_T\in \mathcal F^{(2T+1)}_{\alpha,L,\tau,\Delta_0}(0)$.

Now apply Lemma~\ref{lem:scaling}(iv) together with Proposition~\ref{prop:base-hardness-zr}. Let
\[
R_T^{\mathrm{base}}
:=
\frac{31}{64\cdot 3^{p+1}}\,T^{-p}.
\]
Proposition~\ref{prop:base-hardness-zr} states that every zero-respecting method on \(\Phi_T\) has objective gap at
least \(R_T^{\mathrm{base}}\) after \(T\) steps. Hence Lemma~\ref{lem:scaling}(iv) implies that every
zero-respecting method on \(f_T=\mathcal S_{\lambda_T,\rho_T}\Phi_T\) has objective gap at least
\(\lambda_T R_T^{\mathrm{base}}\) after \(T\) steps:
\[
f_T(x^{(T)})-(f_T)_\star
\ge
\lambda_T R_T^{\mathrm{base}}.
\]

If the second branch in the definition of $\lambda_T$ is active, then
\[
\begin{aligned}
f_T(x^{(T)})-(f_T)_\star
&\ge
\frac{31}{64\cdot 3^{p+1}}T^{-p}
\left(\frac{\tau}{C_{\mathrm{PL}}}\right)^{2/(2-\alpha)}
\left(\frac{L}{37}\right)^p\\
&=
\frac{31}{64\cdot 3^{p+1}\cdot 37^p\cdot C_{\mathrm{PL}}^{2/(2-\alpha)}}
\left(\frac{L\tau^{2/\alpha}}{T}\right)^p,
\end{aligned}
\]
because
\[
\left(\tau^{2/\alpha}\right)^p=\tau^{2/(2-\alpha)}.
\]

If instead the first branch is active, then \(\lambda_T=\Delta_0/\Delta_T\). The ratio bound in
Proposition~\ref{prop:base-hardness-zr} and Lemma~\ref{lem:scaling}(iv) give
	\[
	f_T(x^{(T)})-(f_T)_\star
	\ge
	\lambda_T\cdot \frac{62}{125}\Delta_T
	=
	\frac{62}{125}\Delta_0.
	\]

Combining the two cases proves the theorem.
\end{proof}

\subsection{From zero-respecting to all deterministic first-order methods}
\label{app:det-zr-to-all}

\begin{lemma}
\label{lem:orth-inv}
For every $\alpha\in[1,2)$ and every $L,\tau,\Delta_0>0$, the class
\[
\mathcal F_{\alpha,L,\tau,\Delta_0}(0)
=
\bigcup_{d\ge 1}\mathcal F^{(d)}_{\alpha,L,\tau,\Delta_0}(0)
\]
is orthogonally invariant in the following sense:
if $f\in\mathcal F^{(d)}_{\alpha,L,\tau,\Delta_0}(0)$ and $U\in\R^{d'\times d}$ satisfies $U^\top U=I_d$, then the lifted function
\[
f_U(z):=f(U^\top z),
\qquad z\in\R^{d'},
\]
belongs to $\mathcal F^{(d')}_{\alpha,L,\tau,\Delta_0}(0)$.
\end{lemma}

\begin{proof}
Since \(U^\top U=I_d\), the map \(U^\top:\R^{d'}\to\R^d\) is onto: for every \(y\in\R^d\), choosing \(z=Uy\)
gives \(U^\top z=y\). Hence
\[
(f_U)_\star
=
\inf_{z\in\R^{d'}} f(U^\top z)
=
\inf_{y\in\R^d} f(y)
=
f_\star.
\]
Also,
\[
f_U(0)=f(0),
\]
so the initial gap is unchanged.
Moreover,
\[
S_0(f_U)=\{z:\ f(U^\top z)\le f(0)\}=\{z:\ U^\top z\in S_0(f)\}.
\]
The gradient is
\[
\nabla f_U(z)=U\nabla f(U^\top z).
\]
Because \(U\) is an isometry from \(\R^d\) into \(\R^{d'}\), \(\norm{Ua}=\norm a\) for every \(a\in\R^d\). Hence, for
arbitrary $z,w\in\R^{d'}$,
\[
\begin{aligned}
\norm{\nabla f_U(z)-\nabla f_U(w)}
&=
\norm{U\bigl(\nabla f(U^\top z)-\nabla f(U^\top w)\bigr)}\\
&=
\norm{\nabla f(U^\top z)-\nabla f(U^\top w)}\\
&\le
L\norm{U^\top z-U^\top w}\\
&\le
L\norm{z-w}.
\end{aligned}
\]
Thus $\nabla f_U$ is globally $L$-Lipschitz on $\R^{d'}$.

Finally, if $z\in S_0(f_U)$, then $U^\top z\in S_0(f)$ and therefore
\[
f_U(z)-(f_U)_\star
=
f(U^\top z)-f_\star
\le
\tau\norm{\nabla f(U^\top z)}^\alpha
=
\tau\norm{U\nabla f(U^\top z)}^\alpha
=
\tau\norm{\nabla f_U(z)}^\alpha.
\]
Hence $f_U\in \mathcal F^{(d')}_{\alpha,L,\tau,\Delta_0}(0)$.
\end{proof}

\begin{corollary}[Fixed-budget minimax lower bound]
\label{cor:fixed-budget-lower}
There exists a constant $c_\alpha>0$ such that for every $L,\tau,\Delta_0>0$ and every $T\ge 1$,
\[
\mathcal E_T\!\left(\mathcal F_{\alpha,L,\tau,\Delta_0}(0)\right)
\ge
c_\alpha
\min\!\left\{
\Delta_0,\,
\left(\frac{L\tau^{2/\alpha}}{T}\right)^p
\right\}.
\]
One may take $c_\alpha=c_\alpha^{\mathrm{zr}}$ from Theorem~\ref{thm:zr}.
\end{corollary}

\begin{proof}
Theorem~\ref{thm:zr} gives the fixed-budget lower bound for zero-respecting deterministic first-order methods.
By Lemma~\ref{lem:orth-inv}, the class \(\mathcal F_{\alpha,L,\tau,\Delta_0}(0)\) is orthogonally invariant.
Therefore the resisting-oracle rotation of
\citet[Proposition~1 and the surrounding discussion]{CarmonDuchiHinderSidford2021} transfers the zero-respecting lower
bound to arbitrary deterministic first-order methods on the same class. This is an adversarial construction for each
fixed deterministic method, not a common distribution over rotations. Applying the reduction to
Theorem~\ref{thm:zr} yields
\[
\mathcal E_T\!\left(\mathcal F_{\alpha,L,\tau,\Delta_0}(0)\right)
\ge
c_\alpha^{\mathrm{zr}}
\min\!\left\{
\Delta_0,\,
\left(\frac{L\tau^{2/\alpha}}{T}\right)^p
\right\},
\]
which proves the claim.
\end{proof}

\subsection{Completion of the deterministic argument}
\label{app:det-completion}

\begin{proof}[Proof of Theorem~\ref{thm:matching}]
Let \(c_\alpha^0\) be the constant in Corollary~\ref{cor:fixed-budget-lower}, and choose a constant
\(c_\alpha>0\), depending only on \(\alpha\), such that \(c_\alpha\le c_\alpha^0/2\). Fix
\[
\varepsilon\in\bigl(0,c_\alpha\min\{\Delta_0,L^{\alpha/(2-\alpha)}\tau^{2/(2-\alpha)}\}\bigr],
\]
and suppose a deterministic first-order method uses \(T\) oracle calls and guarantees error at most \(\varepsilon\) on
\(\mathcal F_{\alpha,L,\tau,\Delta_0}(0)\). Then Corollary~\ref{cor:fixed-budget-lower} implies
\[
\varepsilon
\ge
c_\alpha^0
\min\!\left\{
\Delta_0,\,
\left(\frac{L\tau^{2/\alpha}}{T}\right)^p
\right\}.
\]
Since \(\varepsilon\le c_\alpha\Delta_0\), the minimum cannot be attained by
\(\Delta_0\); otherwise the right-hand side would be at least \(c_\alpha^0\Delta_0> \varepsilon\). Hence
\[
\varepsilon
\ge
c_\alpha^0
\left(\frac{L\tau^{2/\alpha}}{T}\right)^p.
\]
Rearranging gives
\[
T
\ge
(c_\alpha^0)^{1/p}L\tau^{2/\alpha}\varepsilon^{-1/p}
=
\Omega\left(L\tau^{2/\alpha}\varepsilon^{-(2-\alpha)/\alpha}\right),
\]
because $1/p=(2-\alpha)/\alpha$.
Thus the iteration-complexity lower bound holds with a constant depending only on \(\alpha\).
\end{proof}

\begin{remark}[Ambient dimension]
This lower bound is stated in the standard dimension-free oracle model, where the ambient dimension of the hard
instance may depend on the query budget. The resisting-oracle argument above does not, by itself, establish the same
bound when the dimension is fixed in advance. Such a result would require tracking how many orthogonal directions the
construction needs as a function of the query budget, and we do not analyze that dependence here.
\end{remark}

\section{Detailed proof of the stochastic lower bound}
\label{app:stochastic}

This appendix contains the full proof of Theorem~\ref{thm:stochastic-matching}.
Subsections~\ref{app:sto-lb-overview}--\ref{app:sto-normalized-lb} build a probabilistic zero-chain hard instance and
prove the lower bound for the normalized class.
Subsections~\ref{app:sto-scaling}--\ref{app:sto-completion} transfer the normalized lower bound to arbitrary
parameters and complete the proof.

For compactness in this appendix only, we use the following shorthand for hard subfamilies of objective-oracle pairs.
Given parameters \((\alpha,L,\tau,\Delta_0,\sigma^2)\), define
\begin{equation}
\label{eq:stochastic-pair-family}
\mathfrak P_{\alpha,L,\tau,\Delta_0,\sigma^2}(0)
:=
\bigcup_{d\ge 1}
\Bigl\{(f,G):\
f\in\mathcal F^{(d)}_{\alpha,L,\tau,\Delta_0}(0),\
G\in\mathfrak O_{\sigma^2}(f)
\Bigr\}.
\end{equation}
For any subfamily \(\mathfrak H\subseteq
\mathfrak P_{\alpha,L,\tau,\Delta_0,\sigma^2}(0)\) of objective-oracle pairs, define
\begin{equation}
\label{eq:pair-family-minimax}
\mathcal E_T^{\mathrm{BV}}(\mathfrak H)
:=
\inf_A\sup_{(f,G)\in\mathfrak H}\E\bigl[f(x_T^A(f,G))-f_\star\bigr],
\qquad
T_\varepsilon^{\mathrm{BV}}(\mathfrak H)
:=
\inf\bigl\{T\ge 0:\mathcal E_T^{\mathrm{BV}}(\mathfrak H)\le\varepsilon\bigr\}.
\end{equation}
This pair-family notation is only a proof device; the full stochastic minimax model is the double supremum over
objectives and admissible BV oracles in Definition~\ref{def:stochastic-model}.

\subsection{Probabilistic zero-chain overview}
\label{app:sto-lb-overview}

We next establish the matching lower bound in the bounded-variance stochastic oracle model. The proof follows the
probabilistic zero-chain construction of \citet{ArjevaniEtAl2019}, but it is stated here in the notation of the
present paper.
To extend the zero-chain lower bound from zero-respecting methods to arbitrary adaptive methods, we use the
random-rotation technique of \citet[Lemma~5 and Appendix~B.1]{ArjevaniEtAl2019}. Their result controls an algorithm's
progress through a randomly rotated zero chain. We need a slightly different formulation: for a chain with \(D\)
coordinates, it is enough to prevent the algorithm from reaching a prescribed progress level \(M\le D\), rather than
to track its progress through the entire chain. We therefore state and prove this target-progress adaptation in
Lemma~\ref{lem:random-rotation-progress}. In that lemma, \(D\) is the full chain dimension and \(M\) is the target
progress level; our application uses \(D=2m\) and \(M=m\).
At a high level, the hard instance keeps the hidden-coordinate idea from the deterministic
lower bound, but the oracle is now designed so that the next hidden coordinate is revealed only with small probability.
The resulting complexity is therefore controlled by the time required to discover new coordinates through noisy
observations, which is exactly what creates the
additional factor
\[
\sigma^2\tau^{2/\alpha}\varepsilon^{-2/\alpha}
\]
and hence the sharp exponent \((4-\alpha)/\alpha\) in the noise-dominated regime.

\begin{definition}[Probability-$q$ zero-chain]
\label{def:probability-zero-chain}
Recall the progress index \(\prog_\delta\) from the appendix notation: \(\prog_0(x)\) is the largest nonzero coordinate
index of \(x\), while \(\prog_{1/4}(x)\) counts only coordinates whose magnitude has exceeded \(1/4\).

Let $F:\R^T\to\R$ be differentiable and let $g:\R^T\times \mathcal Z\to\R^T$ be a stochastic gradient function.
We say that $g$ is a probability-$q$ zero-chain if
\[
\Pbb\!\left(\exists x:\ \prog_0\bigl(g(x,z)\bigr)=\prog_{1/4}(x)+1\right)\le q,
\]
and
\[
\Pbb\!\left(\exists x:\ \prog_0\bigl(g(x,z)\bigr)>\prog_{1/4}(x)+1\right)=0,
\]
where the probability is over a single draw of $z$.
\end{definition}

\subsection{The base chain}
\label{app:sto-base-chain}

Define
\[
\Theta(t):=
\begin{cases}
0, & t\le \frac12,\\[0.3em]
3(2t-1)^2-2(2t-1)^3, & \frac12<t<1,\\[0.3em]
1, & t\ge 1,
\end{cases}
\qquad
\varphi(t):=1-e^{-(t-1)^2/2},
\]
and
\[
B(t):=
\begin{cases}
2+\frac12\left(t+\frac12\right)^2, & t\le -\frac12,\\[0.6em]
2, & -\frac12<t\le \frac12,\\[0.3em]
2(1-\Theta(t)), & \frac12<t<1,\\[0.3em]
\frac12(t-1)^2, & t\ge 1.
\end{cases}
\]
For $N\ge 1$, define
\[
F_N(x):=\varphi(x_1)+\sum_{i=1}^{N-1}\Theta(x_i)\varphi(x_{i+1})+\sum_{i=1}^N B(x_i),
\qquad
x\in\R^N.
\]

The next lemma collects the scalar estimates needed to control the zero-chain building blocks \(F_N\).  The bounds on
\(\Theta\), \(\varphi\), and \(B\) are used later for smoothness, radius control, the zero-chain support property, and
the quadratic-P\L{} estimate on the initial sublevel set.  The lemma is mainly bookkeeping for constants.

\begin{lemma}[Scalar estimates]
\label{lem:scalar-estimates}
The functions $\Theta$, $\varphi$, and $B$ are $C^1$, and
\[
0\le \Theta\le 1,
\qquad
\sup_t |\Theta'(t)|=3,
\qquad
\mathrm{Lip}(\Theta')\le 24,
\]
\[
0\le \varphi\le 1,
\qquad
\sup_t |\varphi'(t)|=e^{-1/2}<1,
\qquad
\mathrm{Lip}(\varphi')\le 1,
\]
and
\[
\mathrm{Lip}(B')\le 48.
\]
Moreover,
\[
\Theta(7/8)=27/32,
\qquad
\sup_{t\in[7/8,1]}|\Theta'(t)|=9/4,
\]
\[
\inf_{t\in[7/8,9/8]}\varphi''(t)\ge e^{-1/128}\left(1-\frac{1}{64}\right)>\frac12,
\]
\[
\sup_{t\in[7/8,9/8]}\varphi(t)\le 1-e^{-1/128}<\frac{1}{128},
\qquad
\sup_{t\in[7/8,9/8]}|\varphi'(t)|\le \frac18,
\]
and
\[
t^2\le 4B(t)+2
\qquad
(\forall t\in\R).
\]
\end{lemma}

\begin{proof}
All claims follow from elementary calculations on the defining pieces, but we record the relevant formulas explicitly.

On the interval $1/2<t<1$, set $u:=2t-1\in(0,1)$.
Then
\[
\Theta(t)=3u^2-2u^3,
\qquad
\Theta'(t)=12u(1-u),
\qquad
\Theta''(t)=24(1-2u).
\]
Hence $0\le \Theta\le 1$, $\sup_t |\Theta'(t)|=3$, and $\mathrm{Lip}(\Theta')\le 24$.
At $t=7/8$ we have $u=3/4$, so
\[
\Theta(7/8)=3\left(\frac34\right)^2-2\left(\frac34\right)^3=\frac{27}{32},
\qquad
\Theta'(7/8)=12\cdot \frac34\cdot \frac14=\frac94.
\]
Since $\Theta'$ decreases on $[7/8,1]$, this also gives
\[
\sup_{t\in[7/8,1]}|\Theta'(t)|=\frac94.
\]

Next,
\[
\varphi'(t)=(t-1)e^{-(t-1)^2/2},
\qquad
\varphi''(t)=e^{-(t-1)^2/2}\bigl(1-(t-1)^2\bigr).
\]
Therefore $0\le \varphi\le 1$.
The function $s\mapsto |s|e^{-s^2/2}$ is maximized at $|s|=1$, so
\[
\sup_t |\varphi'(t)|=e^{-1/2}<1.
\]
Also $|\varphi''(t)|\le 1$ for every $t$, hence $\mathrm{Lip}(\varphi')\le 1$.
If $t\in[7/8,9/8]$, then $|t-1|\le 1/8$, and thus
\[
\varphi''(t)\ge e^{-1/128}\left(1-\frac{1}{64}\right)>\frac12.
\]
On the same interval,
\[
\varphi(t)\le 1-e^{-1/128}<\frac{1}{128},
\qquad
|\varphi'(t)|\le |t-1|\le \frac18.
\]

Finally, $B'$ is piecewise affine:
\[
B'(t)=
\begin{cases}
t+\frac12, & t\le -\frac12,\\[0.3em]
0, & -\frac12<t\le \frac12,\\[0.3em]
-2\Theta'(t), & \frac12<t<1,\\[0.3em]
t-1, & t\ge 1.
\end{cases}
\]
The slopes on these four pieces are respectively $1$, $0$, $-2\Theta''(t)$, and $1$, so
\[
\mathrm{Lip}(B')\le \max\{1,0,2\,\mathrm{Lip}(\Theta'),1\}\le 48.
\]

It remains to prove the crude coercivity estimate
\[
t^2\le 4B(t)+2
\qquad
(\forall t\in\R).
\]
If $t\le -1/2$, then
\[
4B(t)+2
=
10+2\left(t+\frac12\right)^2
=
t^2+2t+\frac{21}{2}+t^2
\ge t^2.
\]
If $-1/2<t\le 1/2$, then $B(t)=2$, so $4B(t)+2=10\ge t^2$.
If $1/2<t<1$, then $B(t)\ge 0$, hence $4B(t)+2\ge 2>t^2$.
If $t\ge 1$, then
\[
4B(t)+2
=
2(t-1)^2+2
=
t^2+(t-2)^2
\ge t^2.
\]
This completes the proof.
\end{proof}

When \(N=1\), the only gradient coordinate is
\[
\partial_1 F_1(x)=\varphi'(x_1)+B'(x_1).
\]
When \(N\ge2\), the gradient coordinates are
\[
\partial_1 F_N(x)=\varphi'(x_1)+\Theta'(x_1)\varphi(x_2)+B'(x_1),
\]
\[
\partial_i F_N(x)=\Theta(x_{i-1})\varphi'(x_i)+\Theta'(x_i)\varphi(x_{i+1})+B'(x_i),
\qquad
2\le i\le N-1,
\]
and
\[
\partial_N F_N(x)=\Theta(x_{N-1})\varphi'(x_N)+B'(x_N).
\]

We now lift the scalar estimates to the chain \(F_N\).  The next lemma records the basic properties needed later: the
function is nonnegative and smooth, its relevant sublevel set is contained in a controlled ball, and its gradient has
the zero-chain support structure that makes coordinates reveal sequentially.

\begin{lemma}[Smoothness, radius, and zero-chain structure]
\label{lem:FN-basic}
For every $N\ge 1$:
\begin{enumerate}[label=\arabic*.,leftmargin=1.5em]
    \item $F_N\ge 0$, and $F_N(x)=0$ if and only if $x=\one_N$.
    \item $\nabla F_N$ is globally $79$-Lipschitz.
    \item $F_N(0)=\varphi(0)+2N\le 3N$.
    \item If $F_N(x)\le F_N(0)$, then $\norm{x}^2\le 14N$.
    \item $F_N$ is a first-order zero-chain in the sense that
    \[
    \prog_0\bigl(\nabla F_N(x)\bigr)\le \prog_{1/4}(x)+1
    \qquad
    (\forall x\in\R^N).
    \]
    In fact, if $k:=\prog_{1/4}(x)$, then $\nabla_j F_N(x)=0$ for all $j\ge k+2$.
    \item If $k:=\prog_{1/4}(x)<N$, then among the hidden coordinates $j>k$ only $j=k+1$ can be nonzero, and
    \[
    |\nabla_{k+1}F_N(x)|\le e^{-1/2}<1.
    \]
    If \(k=N\), there are no hidden coordinates.
\end{enumerate}
\end{lemma}

\begin{proof}
\begin{enumerate}[label=(\arabic*),leftmargin=1.5em]
\item Since $\varphi\ge 0$, $\Theta\ge 0$, and $B\ge 0$, we have $F_N\ge 0$.
Also $B(t)=0$ if and only if $t=1$, hence $F_N(x)=0$ implies $x_i=1$ for every $i$.
Conversely $F_N(\one_N)=0$.

\item Let $\Delta_j:=x_j-y_j$.
If \(N=1\), then
\[
|\partial_1F_1(x)-\partial_1F_1(y)|
\le
\bigl(\mathrm{Lip}(\varphi')+\mathrm{Lip}(B')\bigr)|\Delta_1|
\le 49|\Delta_1|
\le 79|\Delta_1|.
\]
Hence assume \(N\ge2\). Set \(\Delta_0=\Delta_{N+1}=0\) for notational convenience.
For interior indices,
\[
\begin{aligned}
|\partial_i F_N(x)-\partial_i F_N(y)|
&\le
|\Theta(x_{i-1})-\Theta(y_{i-1})|\,|\varphi'(x_i)|
+|\Theta(y_{i-1})|\,|\varphi'(x_i)-\varphi'(y_i)|\\
&\quad
+|\Theta'(x_i)-\Theta'(y_i)|\,|\varphi(x_{i+1})|
+|\Theta'(y_i)|\,|\varphi(x_{i+1})-\varphi(y_{i+1})|\\
&\quad
+|B'(x_i)-B'(y_i)|\\
&\le
3|\Delta_{i-1}|+(1+24+48)|\Delta_i|+3|\Delta_{i+1}|\\
&=
3|\Delta_{i-1}|+73|\Delta_i|+3|\Delta_{i+1}|.
\end{aligned}
\]
The endpoint cases satisfy the same bound with the missing neighbor removed.
Hence, by weighted Cauchy--Schwarz,
\[
|\partial_i F_N(x)-\partial_i F_N(y)|^2
\le
79\bigl(3|\Delta_{i-1}|^2+73|\Delta_i|^2+3|\Delta_{i+1}|^2\bigr).
\]
Summing over $i$ gives
\[
\norm{\nabla F_N(x)-\nabla F_N(y)}^2\le 79^2\norm{x-y}^2.
\]

\item Immediate from the definition.

\item By Lemma~\ref{lem:scalar-estimates},
\[
\norm{x}^2=\sum_{i=1}^N x_i^2\le 4\sum_{i=1}^N B(x_i)+2N\le 4F_N(x)+2N\le 4F_N(0)+2N\le 14N.
\]

\item Let $k:=\prog_{1/4}(x)$.
Then $|x_j|\le 1/4<1/2$ for all $j\ge k+1$.
Hence, for all such $j$,
\[
\Theta(x_j)=\Theta'(x_j)=B'(x_j)=0.
\]
If $j\ge k+2$, then also $|x_{j-1}|\le 1/4$, so $\Theta(x_{j-1})=0$.
Therefore $\nabla_j F_N(x)=0$.

\item If \(k=N\), there are no hidden coordinates and the claim is vacuous. Assume \(k<N\).
This follows from the same observation.
The only hidden coordinate that can remain is $j=k+1$.
If $k\ge 1$, then
\[
|\nabla_{k+1}F_N(x)|
=
|\Theta(x_k)\varphi'(x_{k+1})|
\le
\sup_t |\varphi'(t)|
=
e^{-1/2}
<1.
\]
If $k=0$, then $|x_1|\le 1/4$, hence $\Theta'(x_1)=B'(x_1)=0$ and
\[
|\nabla_1F_N(x)|=|\varphi'(x_1)|\le e^{-1/2}<1.
\]
\end{enumerate}
\end{proof}

\begin{lemma}[Objective barrier at low progress]
\label{lem:objective-barrier}
For every $m\ge 1$ and every $x\in\R^{2m}$,
\[
\prog_{1/4}(x)<m
\quad\Longrightarrow\quad
F_{2m}(x)\ge 2m.
\]
\end{lemma}

\begin{proof}
If $\prog_{1/4}(x)<m$, then $|x_j|\le 1/4$ for all $j\ge m+1$.
Hence $B(x_j)=2$ for each such $j$.
Therefore
\[
F_{2m}(x)\ge \sum_{j=m+1}^{2m}B(x_j)=2m.
\]
\end{proof}

\begin{remark}
The lemma says that if the algorithm has not made progress into the first half of the chain, then the last \(m\)
coordinates still contribute a fixed objective value \(2m\). This is the objective barrier that turns slow coordinate
discovery into an objective-error lower bound.
\end{remark}

\subsection{Quadratic-P\L{} on the initial sublevel of the base chain}
\label{app:sto-quadratic-pl}

Let
\[
\mathcal B_N:=[7/8,9/8]^N.
\]

To prove a quadratic-P\L{} inequality on the initial sublevel set, we split the analysis into two regions. Inside the
box \(\mathcal B_N\), the scalar curvature estimates from Lemma~\ref{lem:scalar-estimates} give local control near
\(\one_N\). Outside the box, we only need a crude guarantee that the gradient cannot be too small. The next lemma gives
this guarantee; in fact, the bound holds everywhere outside \(\mathcal B_N\), not only on the initial sublevel set.

\begin{lemma}[Gradient bounded away from zero outside the box]
\label{lem:outside-box}
For every $N\ge 1$ and every $x\in \R^N\setminus \mathcal B_N$,
\[
\norm{\nabla F_N(x)}\ge \frac{1}{10}.
\]
\end{lemma}

\begin{proof}
Choose the first index $j$ such that $x_j\notin [7/8,9/8]$.
If $j>1$, then $x_{j-1}\in [7/8,9/8]$, hence
\[
q:=\Theta(x_{j-1})\in [27/32,1].
\]
If $j=1$, set $q:=1$.
If $j<N$, set $s:=\varphi(x_{j+1})\in [0,1]$; if $j=N$, set $s:=0$.
Then
\[
\partial_j F_N(x)=q\varphi'(x_j)+s\Theta'(x_j)+B'(x_j).
\]

We treat four regions.

(i) If $x_j\ge 9/8$, then $\Theta'(x_j)=0$ and $B'(x_j)=x_j-1\ge 1/8$, while $q\varphi'(x_j)\ge 0$.
Thus $\partial_j F_N(x)\ge 1/8$.

(ii) If $1/2<x_j\le 7/8$, then $B'(x_j)=-2\Theta'(x_j)$, so
\[
\partial_j F_N(x)=q\varphi'(x_j)+(s-2)\Theta'(x_j)\le q\varphi'(x_j)<0.
\]
Since
\[
-\varphi'(t)=(1-t)e^{-(1-t)^2/2}
\]
is decreasing on $[1/2,7/8]$,
\[
|\partial_j F_N(x)|
\ge
q\bigl(-\varphi'(7/8)\bigr)
\ge
\frac{27}{32}\cdot \frac18 e^{-1/128}
>
\frac{1}{10}.
\]

(iii) If $-1/2\le x_j\le 1/2$, then $\Theta'(x_j)=B'(x_j)=0$, so
\[
\partial_j F_N(x)=q\varphi'(x_j)<0.
\]
Now
\[
-\varphi'(t)=(1-t)e^{-(1-t)^2/2}\ge \frac12 e^{-1/8}
\qquad
\text{on }[-1/2,1/2],
\]
hence
\[
|\partial_j F_N(x)|\ge \frac{27}{32}\cdot \frac12 e^{-1/8}>\frac{1}{10}.
\]

(iv) If $x_j< -1/2$, then again $\Theta'(x_j)=0$ and $B'(x_j)=x_j+1/2$, so
\[
\partial_j F_N(x)=h(x_j),
\qquad
h(t):=q\varphi'(t)+t+\frac12.
\]
For $t\le -1/2$, writing \(r:=1-t\ge 3/2\), the function
\((r^2-1)e^{-r^2/2}\) is maximized on \([3/2,\infty)\) at \(r=\sqrt 3\). Hence
\[
\varphi''(t)=e^{-(t-1)^2/2}\bigl(1-(t-1)^2\bigr)\ge -2e^{-3/2}>-\frac12,
\]
therefore
\[
h'(t)=q\varphi''(t)+1\ge 1-2e^{-3/2}>\frac12.
\]
Thus $h$ is increasing on $(-\infty,-1/2]$, and hence
\[
\partial_j F_N(x)=h(x_j)\le h(-1/2)=q\varphi'(-1/2)\le -\frac{27}{32}\cdot \frac32 e^{-9/8}<-\frac{1}{10}.
\]

In all cases $|\partial_j F_N(x)|\ge 1/10$.
\end{proof}

\begin{lemma}[Local strong convexity on the box]
\label{lem:box-strong}
For almost every $x\in \mathcal B_N$, the Hessian of $F_N$ exists and satisfies
\[
\nabla^2 F_N(x)\succeq \frac12 I_N.
\]
Consequently, $F_N$ is $1/2$-strongly convex on $\mathcal B_N$.
\end{lemma}
\begin{proof}
\emph{Step 1: scalar bounds on the box.}
The only scalar breakpoint contained in \([7/8,9/8]\) is \(1\). For \(t\in[7/8,1)\), set
\(u:=2t-1\in[3/4,1)\). The formulas in the proof of Lemma~\ref{lem:scalar-estimates} give
\[
\Theta''(t)=24(1-2u),
\qquad
B''(t)=-2\Theta''(t)=48(2u-1)\ge24.
\]
For \(t\in(1,9/8]\), we instead have
\(\Theta(t)=1\), \(\Theta'(t)=\Theta''(t)=0\), and \(B''(t)=1\). Consequently, at every point of twice
differentiability in \([7/8,9/8]\),
\[
\Theta(t)\ge\frac{27}{32},
\qquad
|\Theta'(t)|\le\frac94,
\qquad
|\Theta''(t)|\le24,
\qquad
B''(t)\ge1.
\]
Lemma~\ref{lem:scalar-estimates} also gives, throughout this interval,
\[
\varphi''(t)\ge\frac12,
\qquad
0\le\varphi(t)\le\frac{1}{128},
\qquad
|\varphi'(t)|\le\frac18.
\]

\emph{Step 2: lower bound the Hessian.}
The function \(F_N\) is twice differentiable on \(\mathcal B_N\) except on the measure-zero union of the hyperplanes
\(\{x_i=1\}\). If \(N=1\), then, wherever the Hessian exists,
\[
F_1''(x_1)=\varphi''(x_1)+B''(x_1)\ge\frac32>\frac12.
\]
Suppose now that \(N\ge2\).

For \(2\le i\le N-1\), the diagonal Hessian entries are
\[
\partial_{ii}^2F_N(x)
=
\Theta(x_{i-1})\varphi''(x_i)
+\Theta''(x_i)\varphi(x_{i+1})
+B''(x_i).
\]
At the two endpoints,
\[
\partial_{11}^2F_N(x)
=
\varphi''(x_1)+\Theta''(x_1)\varphi(x_2)+B''(x_1),
\qquad
\partial_{NN}^2F_N(x)
=
\Theta(x_{N-1})\varphi''(x_N)+B''(x_N).
\]
The only nonzero off-diagonal entries are between adjacent coordinates:
\[
\partial_{i,i+1}^2F_N(x)
=
\Theta'(x_i)\varphi'(x_{i+1}),
\qquad 1\le i\le N-1.
\]

The bounds from Step~1 therefore imply, for every \(i\),
\[
\partial_{ii}^2F_N(x)
\ge
\frac{27}{32}\cdot\frac12
-24\cdot\frac{1}{128}
+1
=\frac{79}{64},
\]
where the endpoint formulas only improve this lower bound, and
\[
\bigl|\partial_{i,i+1}^2F_N(x)\bigr|
\le
\frac94\cdot\frac18
=\frac{9}{32}.
\]
Thus, for every \(z\in\R^N\),
\begin{align*}
z^\top\nabla^2F_N(x)z
&\ge
\frac{79}{64}\sum_{i=1}^N z_i^2
-\frac{9}{16}\sum_{i=1}^{N-1}|z_i z_{i+1}|\\
&\ge
\left(\frac{79}{64}-\frac{9}{16}\right)\norm{z}^2
=
\frac{43}{64}\norm{z}^2,
\end{align*}
where the second inequality uses
\(\sum_{i=1}^{N-1}|z_i z_{i+1}|\le\norm{z}^2\).
Hence
\[
\nabla^2F_N(x)\succeq\frac{43}{64}I_N\succeq\frac12 I_N
\]
at almost every \(x\in\mathcal B_N\).

\emph{Step 3: pass from the almost-everywhere bound to strong convexity.}
Fix \(x,y\in\mathcal B_N\), set \(d:=x-y\), and define
\(\gamma(t):=y+td\) for \(t\in[0,1]\). Let
\[
\mathcal T
:=
\bigl\{t\in(0,1):\gamma_i(t)=1\text{ for some }i\text{ with }d_i\ne0\bigr\}.
\]
Because every nonconstant affine coordinate crosses \(1\) at most once, \(\mathcal T\) is finite. On each component
of \([0,1]\setminus\mathcal T\), every coordinate that varies along the segment remains on one side of the breakpoint.
Coordinates that are identically equal to \(1\) have \(d_i=0\); after those coordinates are held fixed, the restricted
function \(t\mapsto F_N(\gamma(t))\) is twice differentiable on the interior of each component. Its second derivative
contains only Hessian entries associated with the varying coordinates, so the quadratic-form estimate from Step~2
gives
\[
\frac{d}{dt}\ip{\nabla F_N(\gamma(t))}{d}
\ge
\frac12\norm{d}^2
\]
there.

By Lemma~\ref{lem:FN-basic}(2), \(\nabla F_N\) is globally Lipschitz. Hence
\(t\mapsto\ip{\nabla F_N(\gamma(t))}{d}\) is absolutely continuous on \([0,1]\). We may therefore integrate over the
finitely many subintervals and sum the resulting inequalities to obtain
\[
\ip{\nabla F_N(x)-\nabla F_N(y)}{x-y}
\ge
\frac12\norm{x-y}^2.
\]
Finally, apply this strong-monotonicity inequality to the pair \(y+t(x-y),y\), divide by \(t>0\), and integrate over
\(t\in[0,1]\). This yields
\[
F_N(x)
\ge
F_N(y)+\ip{\nabla F_N(y)}{x-y}
+\frac14\norm{x-y}^2.
\]
This is precisely \(1/2\)-strong convexity of \(F_N\) on the convex set \(\mathcal B_N\).
\end{proof}

\subsection{A regularized version of the base chain}
\label{app:sto-regularized}

Fix
\[
\eta:=\frac{1}{20},
\qquad
\Gamma_N(x):=F_N(x)+\frac{\eta}{2}\norm{x}^2.
\]

\begin{lemma}[Basic properties of $\Gamma_N$]
\label{lem:Gamma-basic}
For every $N\ge 1$:
\begin{enumerate}[label=\arabic*.,leftmargin=1.5em]
    \item $\Gamma_N$ is coercive, hence attains its minimum.
    \item $\Gamma_N(0)=F_N(0)\le 3N$.
    \item $S_0(\Gamma_N)\subseteq S_0(F_N)$.
\end{enumerate}
\end{lemma}

\begin{proof}
For item (1), Lemma~\ref{lem:FN-basic}(1) gives \(F_N\ge 0\). Therefore
\[
\Gamma_N(x)=F_N(x)+\frac{\eta}{2}\norm{x}^2\ge \frac{\eta}{2}\norm{x}^2
\qquad
(\forall x\in\R^N).
\]
Hence $\Gamma_N(x)\to +\infty$ as $\norm{x}\to\infty$, so $\Gamma_N$ is coercive. Since \(\Gamma_N\) is continuous,
it attains its minimum.

For item (2), evaluating the definition at the origin gives
\[
\Gamma_N(0)=F_N(0)+\frac{\eta}{2}\norm{0}^2=F_N(0)\le 3N,
\]
where the last inequality is Lemma~\ref{lem:FN-basic}(3).

For item (3), if $x\in S_0(\Gamma_N)$, then
\[
F_N(x)\le \Gamma_N(x)\le \Gamma_N(0)=F_N(0).
\]
Thus $x\in S_0(F_N)$.
\end{proof}

\begin{lemma}[Gradient bounded away from zero outside the box for $\Gamma_N$]
\label{lem:Gamma-outside-box}
For every $N\ge 1$ and every $x\in \R^N\setminus \mathcal B_N$,
\[
\norm{\nabla \Gamma_N(x)}\ge \frac{1}{20}.
\]
\end{lemma}

\begin{proof}
We repeat the first-bad-coordinate choice in the proof of Lemma~\ref{lem:outside-box}.
Let $j$ be the first coordinate outside $[7/8,9/8]$.
Then
\[
\partial_j \Gamma_N(x)=\partial_j F_N(x)+\eta x_j.
\]
The proof of Lemma~\ref{lem:outside-box} gives the following sign information:
\begin{itemize}[leftmargin=1.5em]
    \item if $x_j\ge 9/8$, then $\partial_j F_N(x)\ge 1/8$;
    \item if $1/2<x_j\le 7/8$, then $\partial_j F_N(x)\le -1/10$;
    \item if $-1/2\le x_j\le 1/2$, then $\partial_j F_N(x)\le -1/10$;
    \item if $x_j< -1/2$, then $\partial_j F_N(x)\le -1/10$.
\end{itemize}

If $x_j\ge 9/8$, then $\eta x_j\ge 0$, so $\partial_j \Gamma_N(x)\ge 1/8>1/20$.

If $1/2<x_j\le 7/8$, then $\eta x_j\le 7\eta/8=7/160$, so
\[
\partial_j \Gamma_N(x)\le -\frac{1}{10}+\frac{7}{160}=-\frac{9}{160}<-\frac{1}{20}.
\]

If $-1/2\le x_j\le 1/2$, then $\eta x_j\le \eta/2=1/40$, so
\[
\partial_j \Gamma_N(x)\le -\frac{1}{10}+\frac{1}{40}=-\frac{3}{40}<-\frac{1}{20}.
\]

If $x_j< -1/2$, then $\eta x_j< -\eta/2=-1/40$, so
\[
\partial_j \Gamma_N(x)\le -\frac{1}{10}-\frac{1}{40}=-\frac18<-\frac{1}{20}.
\]

In all cases $|\partial_j \Gamma_N(x)|\ge 1/20$, hence $\norm{\nabla \Gamma_N(x)}\ge 1/20$.
\end{proof}

\begin{lemma}[Local strong convexity on the box for $\Gamma_N$]
\label{lem:Gamma-box-strong}
$\Gamma_N$ is $(1/2+\eta)$-strongly convex on $\mathcal B_N$.
\end{lemma}

\begin{proof}
By Lemma~\ref{lem:box-strong}, $F_N$ is $1/2$-strongly convex on $\mathcal B_N$.
Adding the quadratic term $(\eta/2)\norm{x}^2$ increases the strong-convexity parameter by $\eta$.
\end{proof}

\begin{proposition}[Quadratic-P\L{} on the initial sublevel of $\Gamma_N$]
\label{prop:Gamma-quadratic-PL}
For every $N\ge 1$ and every $x\in S_0(\Gamma_N)$,
\[
\Gamma_N(x)-\Gamma_{N,\star}\le 1200N\norm{\nabla \Gamma_N(x)}^2.
\]
\end{proposition}

\begin{proof}
Let $x^\star$ be a global minimizer of $\Gamma_N$, which exists by Lemma~\ref{lem:Gamma-basic}(1).
Since $\Gamma_N(x^\star)\le \Gamma_N(0)$, we have $x^\star\in S_0(\Gamma_N)$.
Also $\nabla \Gamma_N(x^\star)=0$, so Lemma~\ref{lem:Gamma-outside-box} implies $x^\star\in \mathcal B_N$.

We distinguish whether $x$ lies inside or outside the box $\mathcal B_N$.

\emph{Case 1: $x\in \mathcal B_N$.}
Then $x$ and $x^\star$ both belong to the convex set $\mathcal B_N$, and Lemma~\ref{lem:Gamma-box-strong} shows that $\Gamma_N$ is $(1/2+\eta)$-strongly convex on that set.
The standard gradient lower bound for strongly convex functions, applied on the convex set \(\mathcal B_N\), therefore gives
\[
\Gamma_N(x)-\Gamma_{N,\star}\le \frac{1}{2(1/2+\eta)}\norm{\nabla \Gamma_N(x)}^2\le \norm{\nabla \Gamma_N(x)}^2.
\]

\emph{Case 2: $x\notin \mathcal B_N$.}
Then Lemma~\ref{lem:Gamma-outside-box} yields
\[
\norm{\nabla \Gamma_N(x)}\ge \frac{1}{20},
\]
and Lemma~\ref{lem:Gamma-basic}(2) gives $\Gamma_N(0)\le 3N$.
Since \(x\in S_0(\Gamma_N)\), we have \(\Gamma_N(x)\le \Gamma_N(0)\). Also \(\Gamma_N\ge0\), so
\(\Gamma_{N,\star}\ge 0\). Therefore
\[
\Gamma_N(x)-\Gamma_{N,\star}\le \Gamma_N(0)\le 3N\le 1200N\norm{\nabla \Gamma_N(x)}^2.
\]

Combining the two cases proves the claim.
\end{proof}

Thus \(\Gamma_N\) is a globally smooth zero-chain whose initial sublevel set satisfies a classical quadratic-P\L{}
inequality. This is the base geometric estimate used below: after choosing the chain length and amplitude, this
quadratic-P\L{} control is converted into the required sublevel \(\alpha\)-P\L{} inequality for the stochastic hard
instance.

\subsection{A probabilistic zero-chain oracle}
\label{app:sto-oracle}

For $m\ge 1$ and $q\in(0,1]$, let $z\sim \mathrm{Bernoulli}(q)$ and define, on $\R^{2m}$,
\[
[g_{m,q}(x,z)]_i
:=
\nabla_i F_{2m}(x)\left(1+\mathbf 1\{i>\prog_{1/4}(x)\}\left(\frac{z}{q}-1\right)\right).
\]

\begin{lemma}[Oracle properties]
\label{lem:oracle-properties}
For every $m\ge 1$ and $q\in(0,1]$:
\begin{enumerate}[label=\arabic*.,leftmargin=1.5em]
    \item $g_{m,q}$ is an unbiased gradient estimator for $F_{2m}$.
    \item $g_{m,q}$ is a probability-$q$ zero-chain.
    \item
    \[
    \E\bigl[\norm{g_{m,q}(x,z)-\nabla F_{2m}(x)}^2\mid x\bigr]\le q^{-1}
    \qquad
    \text{for all }x.
    \]
\end{enumerate}
\end{lemma}

\begin{proof}
We prove the three claims in order.

For item (1), fix $x$ and set $k:=\prog_{1/4}(x)$.
For coordinates \(i\le k\), the multiplier in the definition of \(g_{m,q}\) is equal to one, so
\[
[g_{m,q}(x,z)]_i=\nabla_i F_{2m}(x)
\qquad
\text{for }i\le k,
\]
whereas for coordinates \(i>k\),
\[
[g_{m,q}(x,z)]_i
=
\nabla_i F_{2m}(x)\left(1+\frac{z}{q}-1\right)
=
\nabla_i F_{2m}(x)\frac{z}{q}.
\]
Since $\E[z]=q$, this yields $\E[g_{m,q}(x,z)\mid x]=\nabla F_{2m}(x)$.

For item (2), Lemma~\ref{lem:FN-basic}(5) implies that if $k=\prog_{1/4}(x)$, then
\[
\nabla_j F_{2m}(x)=0
\qquad
\text{for all }j\ge k+2.
\]
Therefore, for every realization of $z$,
\[
\prog_0\bigl(g_{m,q}(x,z)\bigr)\le \prog_{1/4}(x)+1
\qquad
(\forall x\in\R^{2m}),
\]
which is the second requirement in Definition~\ref{def:probability-zero-chain}.
Moreover, if $z=0$, then every coordinate \(i>k\) is suppressed by the factor \(z/q\), so the support cannot extend
beyond $\prog_{1/4}(x)$. If $z=1$, then the support may extend by coordinate \(k+1\), but the preceding display shows
that it can never extend farther.
Hence
\[
\Pbb\!\left(\exists x:\ \prog_0\bigl(g_{m,q}(x,z)\bigr)=\prog_{1/4}(x)+1\right)\le \Pbb(z=1)=q,
\]
which is the first requirement.

For item (3), the same zero-chain support property shows that only coordinate $k+1$ can fluctuate.
If \(k=2m\), no hidden coordinate remains, so \(g_{m,q}(x,z)=\nabla F_{2m}(x)\) and the variance is zero.
Assume \(k<2m\).
Writing $e_{k+1}$ for the $(k+1)$st coordinate vector, we have
\[
g_{m,q}(x,z)-\nabla F_{2m}(x)
=
\nabla_{k+1}F_{2m}(x)\left(\frac{z}{q}-1\right)e_{k+1}.
\]
By Lemma~\ref{lem:FN-basic}(6), $|\nabla_{k+1}F_{2m}(x)|\le 1$.
Therefore
\[
\E\bigl[\norm{g_{m,q}(x,z)-\nabla F_{2m}(x)}^2\mid x\bigr]
\le
\E\left[\left(\frac{z}{q}-1\right)^2\right]
=
q\left(\frac{1}{q}-1\right)^2+(1-q)
=
\frac{1-q}{q}
\le
\frac{1}{q}.
\]
\end{proof}

\subsection{Rotation, clipping, and quadratic regularization}
\label{app:sto-clipping}

We now turn the base probabilistic zero-chain into a hard instance for arbitrary stochastic first-order methods. This
requires three technical modifications. First, a random rotation hides the chain coordinates from the algorithm; the
progress lemma below then says that bounded-norm queries cannot discover the hidden coordinates much faster than the
probability-\(q\) reveal process allows. Second, a radial clipping map forces the point passed to the hidden chain to
remain in a bounded ball, so the bounded-query hypothesis of the rotation lemma is available without restricting the
algorithm. Third, a small quadratic regularization makes the final objective coercive and preserves the sublevel
quadratic-P\L{} estimate proved above.

The next lemma extends the probabilistic zero-chain barrier from zero-respecting methods to arbitrary adaptive
methods. A random rotation hides the coordinate system of the chain, so an algorithm can make progress in only two
ways: the oracle can reveal the next coordinate through its Bernoulli event of probability at most \(q\), or a query
can have an unusually large projection onto a rotated direction that has not yet been revealed. The proof controls
both events simultaneously.

The use of random rotation for this reduction is not new. In particular, Lemma~5 of
\citet{ArjevaniEtAl2019} establishes such a progress barrier for a probability-\(p\) zero chain of dimension \(T\)
and allows \(K\) queries per round. To express their result in our notation, set
\[
K=1,\qquad T=D,\qquad p=q.
\]
Under this identification, their lemma prevents the algorithm from reaching the final coordinate \(D\) for
approximately \(D/q\) rounds. Our objective barrier, however, is triggered before the end of the chain: the hard
instance has \(D=2m\) coordinates, but it is enough to prevent the algorithm from reaching coordinate \(m\). We
therefore introduce a separate target level \(M\le D\) and prove
\(\prog_{1/4}(U^\top x_t)<M\) for approximately \(M/q\) rounds; our application takes \(M=m\). Although the target
is \(M\), the ambient-dimension requirement still depends on the full chain dimension \(D\), because an adaptive
transcript may contain information about any of the \(D\) rotated directions. This target-progress formulation is not
stated in the cited lemma, so we prove it below by adapting the filtration and random-projection argument from their
Appendix~B.1.

\begin{lemma}[Random-rotation progress barrier]
\label{lem:random-rotation-progress}
Let \(1\le M\le D\) and \(d\ge1\) be integers, \(q\in(0,1]\), \(R\ge1\), and \(\delta\in(0,1)\). Let
\(F:\R^D\to\R\), and let \(g:\R^D\times\mathcal Z\to\R^D\) be a probability-\(q\) zero chain for \(F\).
Assume
\[
d\ge
\max\left\{
D,\,
18R^2Dq^{-1}\log\left(\frac{2D^2}{q\delta}\right)
\right\}.
\]
Draw \(U\) uniformly from
\[
\Ortho(d,D):=\{U\in\R^{d\times D}:U^\top U=I_D\}.
\]
Consider any adaptive, possibly randomized algorithm that makes one query \(x_t\in\R^d\) per round and receives the
value-and-gradient response
\[
O_U(x_t,z_t):=\bigl(F(U^\top x_t),\,Ug(U^\top x_t,z_t)\bigr).
\]
Assume that \(U\), the oracle seeds, and the algorithm's internal randomness are mutually independent, that the
\(z_t\) are independent samples from the seed distribution defining \(g\), and that \(\norm{x_t}\le R\) almost surely
for every \(t\). Define
\[
t_\star
:=
\left\lfloor\frac{M-\log(2/\delta)}{2q}\right\rfloor.
\]
Then
\[
\Pbb\left(
\prog_{1/4}(U^\top x_t)<M
\ \text{for every integer }1\le t\le t_\star
\right)
\ge1-\delta,
\]
where the probability is over \(U\), the oracle seeds, and the algorithm's internal randomness.
\end{lemma}

\begin{proof}
The normalization $R\ge1$ is harmless: if queries are bounded by a smaller radius, one may use the valid upper bound
$R=1$. It also makes the numerical dimension estimate below uniform over all parameter values.

\emph{Step 1: control progress caused by oracle reveals.}
Represent all of the algorithm's internal randomness by a seed \(r\) sampled before the first query, and condition on
\(r\). For almost every such seed, the assumed query bound continues to hold almost surely. At the end of the proof
we average over \(r\). Write \(U=(u_1,\ldots,u_D)\), and for the base-coordinate gradient response
\(g_t:=g(U^\top x_t,z_t)\), define
\[
\pi_t:=\max_{i\le t}\prog_{1/4}(U^\top x_i),
\qquad
\gamma_t:=\max_{i\le t}\prog_0(g_i),
\qquad
\gamma_0:=0,
\]
and introduce the no-leakage event
\[
V_i:=\{\prog_{1/4}(U^\top x_i)\le\gamma_{i-1}\}.
\]
Let \(\mathcal F_{i-1}:=\sigma(U,r,z_1,\ldots,z_{i-1})\). The query \(x_i\), the previous progress
\(\gamma_{i-1}\), and the event \(V_i\) are all \(\mathcal F_{i-1}\)-measurable. On \(V_i\), the zero-chain property
gives
\[
\prog_0(g_i)\le\prog_{1/4}(U^\top x_i)+1\le\gamma_{i-1}+1
\qquad\text{almost surely}.
\]
Hence \(\gamma_i-\gamma_{i-1}\in\{0,1\}\) on \(V_i\). Moreover, if this increment equals one, then necessarily
\[
\prog_{1/4}(U^\top x_i)=\gamma_{i-1},
\qquad
\prog_0(g_i)=\prog_{1/4}(U^\top x_i)+1.
\]
The defining probability-\(q\) event is uniform over the query point, and \(z_i\) is independent of
\(\mathcal F_{i-1}\). Therefore
\[
\Pbb(\gamma_i-\gamma_{i-1}=1,\,V_i\mid\mathcal F_{i-1})\le q.
\]
To stop the reveal count if geometric leakage occurs, set
\[
E_i:=\bigcap_{s=1}^iV_s,\qquad E_0:=\Omega,
\qquad
\zeta_i:=(\gamma_i-\gamma_{i-1})\mathbf 1_{E_i}.
\]
If \(E_{i-1}\) fails, then \(\zeta_i=0\). On
\(E_{i-1}\), the preceding conditional reveal bound implies
\(\zeta_i\in\{0,1\}\) and
\(\Pbb(\zeta_i=1\mid\mathcal F_{i-1})\le q\). Consequently,
\[
\E[e^{\zeta_i}\mid\mathcal F_{i-1}]
\le 1+(e-1)q\le e^{2q}.
\]
Iterating conditional expectations gives
\(\E[\exp(\sum_{i=1}^t\zeta_i)]\le e^{2qt}\). On \(E_t\), the stopped sum telescopes to
\(\sum_{i=1}^t\zeta_i=\gamma_t\), so Markov's inequality gives
\begin{equation}
\label{eq:target-reveal-bound}
\Pbb\left(\gamma_t\ge M,\ \bigcap_{i=1}^tV_i\right)
\le e^{2qt-M}.
\end{equation}

\emph{Step 2: convert small subspace increments into small hidden-coordinate projections.}
Fix an integer horizon \(t\ge1\). For \(i\le t\) and \(j\le D\), let
$P_{j-1,i}$ be the orthogonal projector onto
\[
\operatorname{span}\{u_1,\ldots,u_{j-1},x_1,\ldots,x_i\},
\]
and set $P_{j-1,i}^{\perp}:=I-P_{j-1,i}$. Define
\[
G_{i,j}:=
\left\{
\norm{P_{j-1,i}P_{j-1,i-1}^{\perp}u_j}<\frac{1}{4R\sqrt t}
\right\}.
\]
To verify the consequence of these events, define
\[
Q_{s,j}
:=
P_{j-1,s}-P_{j-1,s-1}
=
P_{j-1,s}P_{j-1,s-1}^{\perp},
\qquad 1\le s\le i.
\]
Because the underlying subspaces are nested, the \(Q_{s,j}\) are mutually orthogonal projections. Moreover,
\(P_{j-1,0}u_j=0\), since \(u_j\) is orthogonal to \(u_1,\ldots,u_{j-1}\), and
\[
P_{j-1,i}=P_{j-1,0}+\sum_{s=1}^iQ_{s,j}.
\]
Since \(x_i\) belongs to the range of \(P_{j-1,i}\), it follows that
\begin{align*}
|\ip{u_j}{x_i}|
&=
\left|\sum_{s=1}^i\ip{Q_{s,j}u_j}{Q_{s,j}x_i}\right|\\
&\le
\left(\sum_{s=1}^i\norm{Q_{s,j}u_j}^2\right)^{1/2}
\left(\sum_{s=1}^i\norm{Q_{s,j}x_i}^2\right)^{1/2}.
\end{align*}
On \(\cap_{s=1}^iG_{s,j}\), the first factor is smaller than
\(\sqrt{i/(16R^2t)}\). Orthogonality of the \(Q_{s,j}\) bounds the second factor by \(\norm{x_i}\).
Therefore
\begin{equation}
\label{eq:target-projector-bound}
\bigcap_{s=1}^iG_{s,j}
\quad\Longrightarrow\quad
|\ip{u_j}{x_i}|
\le \frac{\norm{x_i}}{4R}\sqrt{\frac{i}{t}}
\le\frac14.
\end{equation}

\emph{Step 3: bound accidental discovery of an unrevealed rotated direction.}
We now bound the failures of \(G_{i,j}\) while \(u_j\) is unrevealed. Introduce the geometric filtration
\[
\mathcal H_{i-1,j}
:=
\sigma\bigl(r,u_1,\ldots,u_{j-1},
U^\top x_1,z_1,\ldots,U^\top x_{i-1},z_{i-1}\bigr).
\]
Set \(T_{i,j}:=\{j>\gamma_{i-1}\}\). This event is \(\mathcal H_{i-1,j}\)-measurable: from each pair
\((U^\top x_s,z_s)\), one can compute \(g_s=g(U^\top x_s,z_s)\) and hence \(\gamma_{i-1}\). On \(T_{i,j}\), every
previous base response \(g_s\), \(s<i\), is supported on coordinates \(1,\ldots,j-1\). The information in
\(\mathcal H_{i-1,j}\) therefore determines the full previous oracle transcript, because
\[
F(U^\top x_s)
\quad\text{is known, and}\quad
Ug_s=\sum_{\ell=1}^{j-1}(g_s)_\ell u_\ell.
\]
Consequently, on \(T_{i,j}\), the next adaptive query \(x_i\), and hence the projectors
\(P_{j-1,i-1}\) and \(P_{j-1,i}\), are determined by \(\mathcal H_{i-1,j}\).

We claim that, conditional on \(\mathcal H_{i-1,j}\) and \(T_{i,j}\), the normalized residual
\[
\widehat u_j:=
\frac{P_{j-1,i-1}^{\perp}u_j}
{\norm{P_{j-1,i-1}^{\perp}u_j}}
\]
is uniform on the unit sphere in
\[
S^\perp,
\qquad
S:=\operatorname{span}\{u_1,\ldots,u_{j-1},x_1,\ldots,x_{i-1}\}.
\]
The denominator is nonzero almost surely; on the null event where it vanishes, define \(\widehat u_j\) arbitrarily.
To prove the claim, let \(W\) be any orthogonal transformation that fixes \(S\) pointwise, replace \(U\) by \(WU\),
and keep \(r,z_1,\ldots,z_{i-1}\) unchanged. We show inductively that the first \(i-1\) queries and oracle responses
are unchanged. The first query depends only on \(r\). If the queries agree through round \(s<i\), then \(W x_s=x_s\),
so
\[
(WU)^\top x_s=U^\top x_s.
\]
The base response and the function value are therefore unchanged. On \(T_{i,j}\), the base response is supported on
coordinates below \(j\), and \(W\) fixes \(u_1,\ldots,u_{j-1}\); hence \(WUg_s=Ug_s\). Thus the full response in
round \(s\) is unchanged, and the next query is unchanged as well. This completes the induction.

The transformation \(U\mapsto WU\) therefore preserves the conditioned information and the event \(T_{i,j}\).
Because the uniform distribution on \(\Ortho(d,D)\) is invariant under left multiplication by \(W\), the conditional
law of \(\widehat u_j\) is invariant under every rotation of \(S^\perp\). Hence it is uniform on the unit sphere in
\(S^\perp\), as claimed. This is the filtration argument of
\citet[Appendix~B.1]{ArjevaniEtAl2019}, written here for one query per round and with the function values included in
the transcript.

Let \(n:=\dim(S^\perp)\). Since \(S\) is spanned by at most \(j-1\) rotation columns and \(i-1\) previous queries,
\[
n\ge d-i-j+2.
\]
On \(T_{i,j}\), the operator
\[
Q_{i,j}:=P_{j-1,i}P_{j-1,i-1}^\perp
\]
is determined by \(\mathcal H_{i-1,j}\) and is an orthogonal projector of rank at most one. Moreover,
\[
Q_{i,j}u_j
=
\norm{P_{j-1,i-1}^{\perp}u_j}\,Q_{i,j}\widehat u_j,
\qquad
\norm{P_{j-1,i-1}^{\perp}u_j}\le1.
\]
Thus \(G_{i,j}^c\) implies
\[
\norm{Q_{i,j}\widehat u_j}\ge \frac{1}{4R\sqrt t}.
\]
For a vector \(v\) uniformly distributed on the unit sphere in an \(n\)-dimensional Euclidean space and any
rank-one orthogonal projector \(Q\), the standard spherical-coordinate bound gives
\[
\Pbb\left(\norm{Qv}\ge a\right)
\le
2\exp\left(-\frac{(n-1)a^2}{2}\right)
\qquad (0\le a\le1).
\]
Apply this bound conditionally on \(\mathcal H_{i-1,j}\) and \(T_{i,j}\), with
\(a=(4R\sqrt t)^{-1}\), and then average over the conditioning. Since \(R,t\ge1\), this value of \(a\) lies in
\([0,1]\), and the lower bound on \(n\) yields
\[
\Pbb\left(G_{i,j}^c\cap T_{i,j}\right)
\le
2\exp\left(-\frac{d-i-j}{32R^2t}\right)
\le
2\exp\left(-\frac{d-t-D}{32R^2t}\right).
\]

We next translate these bounds into the no-leakage events \(V_i\). Suppose that
\(G_{s,j}\) holds whenever \(s\le t\) and \(j>\gamma_{s-1}\). Fix \(i\le t\) and
\(j>\gamma_{i-1}\). Since \(\gamma_s\) is nondecreasing, \(j>\gamma_{s-1}\) for every \(s\le i\), so
\(G_{s,j}\) holds for all \(s\le i\). Equation~\eqref{eq:target-projector-bound} then gives
\(|\langle u_j,x_i\rangle|<1/4\). This holds for every \(j>\gamma_{i-1}\), and hence \(V_i\) holds. Therefore
\[
\left(\bigcap_{i=1}^tV_i\right)^c
\subseteq
\bigcup_{i=1}^t\bigcup_{j=1}^D
\left(G_{i,j}^c\cap T_{i,j}\right).
\]
Combining this containment with the preceding tail estimate and taking a union bound gives
\begin{equation}
\label{eq:target-leakage-bound}
\Pbb\left(\left(\bigcap_{i=1}^tV_i\right)^c\right)
\le
2tD\exp\left(-\frac{d-t-D}{32R^2t}\right).
\end{equation}
The union is over all $D$ columns, not only the first $M$, because the conditional Stiefel geometry and the information
available in an adaptive transcript involve the full rotated chain. This is why the dimension requirement depends on
$D$.

\emph{Step 4: choose the horizon and combine the two failure probabilities.}
Let $H:=\log(2D^2/(q\delta))$.
If $t_\star<1$, there are no positive integers in the claimed time range, so the conclusion is immediate. Assume
that $t_\star\ge1$. Since $M\le D$,
\[
t_\star
\le \frac{M-\log(2/\delta)}{2q}
\le \frac{D}{2q},
\qquad
\log\left(\frac{4t_\star D}{\delta}\right)\le H.
\]
The leakage bound~\eqref{eq:target-leakage-bound} is at most \(\delta/2\) whenever
\begin{equation}
\label{eq:target-dimension-sufficient}
d
\ge
t_\star+D
+32R^2t_\star\log\left(\frac{4t_\star D}{\delta}\right).
\end{equation}
The logarithmic term on the right satisfies
\[
32R^2t_\star\log\left(\frac{4t_\star D}{\delta}\right)
\le16R^2Dq^{-1}H.
\]
It remains to bound the first two terms. We claim that
\[
t_\star+D\le 2R^2Dq^{-1}H.
\]
If \(q\le1/2\), then
\(t_\star+D\le D/q\le2R^2Dq^{-1}H\) because \(H\ge\log2>1/2\). If \(q>1/2\), the assumption
\(t_\star\ge1\) forces \(D\ge2\), so \(H\ge\log8\) and
\(t_\star+D\le(1/2+q)D/q\le(3/2)D/q\le2R^2Dq^{-1}H\). Therefore
\[
 t_\star+D
 +32R^2t_\star\log\left(\frac{4t_\star D}{\delta}\right)
 \le18R^2Dq^{-1}H
 \le d.
\]
Thus~\eqref{eq:target-dimension-sufficient} holds, and~\eqref{eq:target-leakage-bound} gives
\[
\Pbb(E_{t_\star}^c)\le\frac\delta2.
\]
This is the numerical specialization to \(K=1\) of the dimension calculation in the proof of
\citet[Lemma~5]{ArjevaniEtAl2019}.

By the definition of \(t_\star\), Equation~\eqref{eq:target-reveal-bound} also gives
\[
\Pbb(\gamma_{t_\star}\ge M,\ E_{t_\star})
\le e^{2qt_\star-M}
\le\frac\delta2.
\]
Finally,
on \(E_{t_\star}=\cap_{i=1}^{t_\star}V_i\), we have
\(\pi_{t_\star}\le\gamma_{t_\star-1}\le\gamma_{t_\star}\), and therefore
\[
\Pbb(\pi_{t_\star}\ge M)
\le
\Pbb\left(\gamma_{t_\star}\ge M,\ \bigcap_{i=1}^{t_\star}V_i\right)
+
\Pbb\left(\left(\bigcap_{i=1}^{t_\star}V_i\right)^c\right)
\le\delta.
\]
Since $\pi_t$ is nondecreasing, this proves the simultaneous claim for every $t\le t_\star$. Finally, averaging the
conditional bound over $r$ incorporates the algorithm's internal randomness. The proof uses one query in each round,
which is the specialization $K=1$ of the batched model in the cited result.
\end{proof}

\begin{remark}
Lemma~\ref{lem:random-rotation-progress} is stated for a stronger oracle than the bounded-variance stochastic-gradient
oracle of Definition~\ref{def:stochastic-model}. A lower bound in this stronger model immediately implies the same
lower bound for the gradient-only model.
\end{remark}

Fix $m\ge 1$ and set
\[
R_m:=16\sqrt m.
\]
Define a $C^1$ profile $h_m:[0,\infty)\to [0,2R_m)$ by
\[
h_m(r):=
\begin{cases}
r, & 0\le r\le R_m,\\[0.3em]
R_m+R_m\bigl(1-e^{-(r-R_m)/R_m}\bigr), & r>R_m.
\end{cases}
\]
Then $h_m(r)=r$ on $[0,R_m]$, $0<h_m'(r)\le 1$ for $r>R_m$, and
\[
h_m(r)\le r,
\qquad
h_m(r)<2R_m
\qquad
(\forall r\ge 0),
\]
because $1-e^{-s}\le s$ for $s\ge 0$.
Define the radial clipping map
\[
\rho_m(x):=
\begin{cases}
\dfrac{h_m(\norm{x})}{\norm{x}}\,x, & x\neq 0,\\[0.8em]
0, & x=0.
\end{cases}
\]
Then $\rho_m(x)=x$ for $\norm{x}\le R_m$, $\norm{\rho_m(x)}\le 2R_m$, $\norm{\rho_m(x)}\le \norm{x}$ for all $x$, and
\[
\sup_x \norm{J\rho_m(x)}_{\mathrm{op}}\le 1.
\]
Indeed, for $x\neq 0$ with $r=\norm{x}$, the eigenvalues of $J\rho_m(x)$ are $h_m'(r)$ in the radial direction and $h_m(r)/r$ on tangential directions; both lie in $[0,1]$.

\begin{lemma}[Global regularity of the radial clipping map]
\label{lem:radial-clipping}
For every $m\ge 1$ and every Euclidean dimension \(d\), the Jacobian $J\rho_m$ is globally
$(11/R_m)$-Lipschitz on \(\R^d\):
\[
\norm{J\rho_m(x)-J\rho_m(y)}_{\mathrm{op}}\le \frac{11}{R_m}\norm{x-y}
\qquad
(\forall x,y\in\R^d).
\]
\end{lemma}

\begin{proof}
For $x\neq 0$, write
\[
r:=\norm{x},
\qquad
e:=\frac{x}{\norm{x}},
\qquad
a(r):=\frac{h_m(r)}{r},
\qquad
b(r):=h_m'(r)-\frac{h_m(r)}{r}.
\]
Then
\[
J\rho_m(x)=a(r)I+b(r)ee^\top.
\]
If $r\le R_m$, then $h_m(r)=r$ and $h_m'(r)=1$, so $a(r)=1$ and $b(r)=0$.

Assume now that $r\ge R_m$.
From the definition of $h_m$,
\[
h_m(r)=R_m\bigl(2-e^{-(r-R_m)/R_m}\bigr),
\qquad
h_m'(r)=e^{-(r-R_m)/R_m},
\qquad
h_m''(r)=-\frac{1}{R_m}e^{-(r-R_m)/R_m}.
\]
Since $r\ge R_m$, this implies
\[
0\le a(r)\le 1,
\qquad
|a'(r)|
=
\left|\frac{h_m'(r)r-h_m(r)}{r^2}\right|
\le
\frac{r+2R_m}{r^2}
\le
\frac{3}{R_m},
\]
and
\[
|b(r)|\le 1,
\qquad
|b'(r)|\le |h_m''(r)|+|a'(r)|\le \frac{1}{R_m}+\frac{3}{R_m}=\frac{4}{R_m}.
\]

Suppose first that $\norm{x}\ge R_m$ and $\norm{y}\ge R_m$, and define
\[
s:=\norm{y},
\qquad
f:=\frac{y}{\norm{y}}.
\]
Then
\[
J\rho_m(x)-J\rho_m(y)
=
[a(r)-a(s)]I+[b(r)-b(s)]ee^\top+b(s)(ee^\top-ff^\top).
\]
Hence
\[
\norm{J\rho_m(x)-J\rho_m(y)}_{\mathrm{op}}
\le
|a(r)-a(s)|+|b(r)-b(s)|+|b(s)|\,\norm{ee^\top-ff^\top}_{\mathrm{op}}.
\]
By the bounds above,
\[
|a(r)-a(s)|\le \frac{3}{R_m}|r-s|,
\qquad
|b(r)-b(s)|\le \frac{4}{R_m}|r-s|.
\]
Also
\[
\norm{ee^\top-ff^\top}_{\mathrm{op}}
\le
2\norm{e-f}
\le
\frac{4}{R_m}\norm{x-y},
\]
because $\min\{r,s\}\ge R_m$ and the map $x\mapsto x/\norm{x}$ is $(2/R_m)$-Lipschitz on $\{x:\ \norm{x}\ge R_m\}$.
Therefore
\[
\norm{J\rho_m(x)-J\rho_m(y)}_{\mathrm{op}}\le \frac{11}{R_m}\norm{x-y}.
\]

If $\norm{x}\le R_m$ and $\norm{y}\le R_m$, then \(\rho_m\) is the identity map in a neighborhood of both points and
$J\rho_m(x)=J\rho_m(y)=I$, so the claim is trivial.
In the remaining case, the line segment joining $x$ and $y$ crosses the sphere $\{\norm{z}=R_m\}$ at some point $z$.
Since $J\rho_m(z)=I$, the inside-ball estimate applies to the segment from the inside point to \(z\), and the
outside-ball estimate applies to the segment from \(z\) to the outside point. Therefore
\[
\begin{aligned}
\norm{J\rho_m(x)-J\rho_m(y)}_{\mathrm{op}}
&\le
\norm{J\rho_m(x)-J\rho_m(z)}_{\mathrm{op}}+\norm{J\rho_m(z)-J\rho_m(y)}_{\mathrm{op}}\\
&\le
\frac{11}{R_m}\bigl(\norm{x-z}+\norm{y-z}\bigr)
=
\frac{11}{R_m}\norm{x-y}.
\end{aligned}
\]
This proves the lemma.
\end{proof}

Let $m_0:=5$ and $\delta:=1/4$.

For each pair $(m,q)$ with $m\ge m_0$ and $q\in(0,1]$, let $d_{m,q}$ be any integer satisfying
\begin{equation}
\label{eq:dmq-threshold}
d_{m,q}\ge
\max\left\{
2m,\,
18(2R_m)^2(2m)q^{-1}\log\left(\frac{2(2m)^2}{q\delta}\right)
\right\}.
\end{equation}

Fix any $d\ge d_{m,q}$ and any $U\in \Ortho(d,2m)$, and define
\begin{equation}
\label{eq:hatF-def}
\widehat F_{m,U}(x):=F_{2m}(U^\top \rho_m(x))+\frac{\eta}{2}\norm{x}^2,
\end{equation}
and
\begin{equation}
\label{eq:hatg-def}
\widehat g_{m,q,U}(x,z):=J\rho_m(x)^\top U g_{m,q}(U^\top \rho_m(x),z)+\eta x.
\end{equation}

The dimension threshold in~\eqref{eq:dmq-threshold} is exactly the condition in
Lemma~\ref{lem:random-rotation-progress} with
$D=2m$, $M=m$, $R=2R_m$, and $\delta=1/4$. In particular, it accounts for all $2m$ rotated chain directions rather
than only the first $m$ coordinates needed by the objective barrier. The next lemma packages the properties needed for
the lower bound: the oracle is unbiased with bounded variance, the objective is globally smooth and has
quadratic-P\L{} geometry on its initial sublevel set, and random rotation plus clipping turns slow coordinate discovery
into an expected objective-error lower bound.

\begin{lemma}[Properties of the clipped, rotated, regularized instance]
\label{lem:clipped-rotated}
There exists a numerical constant $c_0>0$ such that the following hold for every $m\ge m_0$, every $q\in(0,1]$, every $d\ge d_{m,q}$, and every $U\in \Ortho(d,2m)$:
\begin{enumerate}[label=\arabic*.,leftmargin=1.5em]
    \item $\widehat g_{m,q,U}$ is an unbiased gradient estimator for $\widehat F_{m,U}$.
    \item
    \[
    \E\bigl[\norm{\widehat g_{m,q,U}(x,z)-\nabla \widehat F_{m,U}(x)}^2\mid x\bigr]\le \frac{4}{q}.
    \]
    \item $\nabla \widehat F_{m,U}$ is globally $L_{\mathrm{clip}}$-Lipschitz for a universal numerical constant $L_{\mathrm{clip}}$.
    \item $S_0(\widehat F_{m,U})\subseteq B(0,R_m)$.
    Hence, by~\eqref{eq:hatF-def}, on $S_0(\widehat F_{m,U})$,
    \[
    \widehat F_{m,U}(x)=F_{2m}(U^\top x)+\frac{\eta}{2}\norm{x}^2,
    \]
    and
    \[
    \widehat F_{m,U}(x)-\widehat F_{m,U,\star}\le 2400m\norm{\nabla \widehat F_{m,U}(x)}^2
    \qquad
    (x\in S_0(\widehat F_{m,U})).
    \]
    \item If $U$ is drawn uniformly from $\Ortho(d,2m)$, $A$ is any randomized algorithm in the gradient-only stochastic
    first-order model of this paper, and
    \[
    t\le c_0\frac{m}{q},
    \]
    then
    \[
    \E_{U,\xi}\bigl[\widehat F_{m,U}(x_t^A)-\widehat F_{m,U,\star}\bigr]\ge m,
    \]
    where \(x_t^A\) is the point generated when \(A\) is run with stochastic-gradient oracle
    \(\widehat g_{m,q,U}\), and the expectation is over $U$, the oracle randomness, and the internal randomness of $A$.
\end{enumerate}
\end{lemma}

\begin{proof}
We prove the five claims in order.

\emph{Step 1: unbiasedness.}
Since $U$ is linear and $\rho_m$ is $C^1$, the chain rule applied to~\eqref{eq:hatF-def} yields
\[
\nabla \widehat F_{m,U}(x)=J\rho_m(x)^\top U\nabla F_{2m}(U^\top \rho_m(x))+\eta x.
\]
Taking conditional expectation in~\eqref{eq:hatg-def} and using Lemma~\ref{lem:oracle-properties}(1) therefore gives
\[
\E[\widehat g_{m,q,U}(x,z)\mid x]=\nabla \widehat F_{m,U}(x).
\]

\emph{Step 2: variance.}
By Lemma~\ref{lem:oracle-properties}(3) and the bound $\norm{J\rho_m(x)}_{\mathrm{op}}\le 1$,
\[
\begin{aligned}
\E\bigl[\norm{\widehat g_{m,q,U}(x,z)-\nabla \widehat F_{m,U}(x)}^2\mid x\bigr]
&=
\E\Bigl[\norm{J\rho_m(x)^\top U\bigl(g_{m,q}(U^\top \rho_m(x),z)-\nabla F_{2m}(U^\top \rho_m(x))\bigr)}^2\mid x\Bigr]\\
&\le
\E\bigl[\norm{g_{m,q}(U^\top \rho_m(x),z)-\nabla F_{2m}(U^\top \rho_m(x))}^2\mid x\bigr]\\
&\le
q^{-1}.
\end{aligned}
\]
In particular, the variance is at most $4/q$.

\emph{Step 3: global smoothness.}
Lemma~\ref{lem:FN-basic}(2) shows that $\nabla F_{2m}$ is globally $79$-Lipschitz and, by the explicit gradient formula at the origin, $\norm{\nabla F_{2m}(0)}=e^{-1/2}<1$.
Since $\norm{\rho_m(x)}\le 2R_m$ for all $x$, we obtain
\[
\sup_{\norm{u}\le 2R_m}\norm{\nabla F_{2m}(u)}\le 1+79\cdot 2R_m\le 159R_m.
\]
For arbitrary $x,x'\in\R^d$, set
\[
y:=U^\top \rho_m(x),
\qquad
y':=U^\top \rho_m(x').
\]
Then
\[
\nabla \widehat F_{m,U}(x)=J\rho_m(x)^\top U\nabla F_{2m}(y)+\eta x,
\]
and therefore
\[
\begin{aligned}
\norm{\nabla \widehat F_{m,U}(x)-\nabla \widehat F_{m,U}(x')}
&\le
\norm{J\rho_m(x)^\top U\bigl(\nabla F_{2m}(y)-\nabla F_{2m}(y')\bigr)}\\
&\quad
+\norm{\bigl(J\rho_m(x)-J\rho_m(x')\bigr)^\top U\nabla F_{2m}(y')}
+\eta \norm{x-x'}.
\end{aligned}
\]
Using $\norm{J\rho_m}_{\mathrm{op}}\le 1$, the global $79$-smoothness of $F_{2m}$, and the fact that $\rho_m$ is $1$-Lipschitz,
\[
\norm{J\rho_m(x)^\top U\bigl(\nabla F_{2m}(y)-\nabla F_{2m}(y')\bigr)}
\le
79\norm{x-x'}.
\]
Using Lemma~\ref{lem:radial-clipping} and the bound on $\sup_{\norm{u}\le 2R_m}\norm{\nabla F_{2m}(u)}$,
\[
\norm{\bigl(J\rho_m(x)-J\rho_m(x')\bigr)^\top U\nabla F_{2m}(y')}
\le
\frac{11}{R_m}\cdot 159R_m\norm{x-x'}
=
1749\norm{x-x'}.
\]
Hence
\[
\norm{\nabla \widehat F_{m,U}(x)-\nabla \widehat F_{m,U}(x')}
\le
(79+1749+\eta)\norm{x-x'}.
\]
We may therefore take
\[
L_{\mathrm{clip}}:=1829.
\]

\emph{Step 4: geometry on the initial sublevel set.}
We first show that every point in the initial sublevel set lies inside the unclipped region.
From~\eqref{eq:hatF-def},
\[
\widehat F_{m,U}(0)=F_{2m}(0)=\varphi(0)+4m<5m,
\]
every $x\in S_0(\widehat F_{m,U})$ satisfies
\[
\frac{\eta}{2}\norm{x}^2\le \widehat F_{m,U}(x)\le \widehat F_{m,U}(0)<5m,
\]
hence
\[
\norm{x}^2<\frac{10m}{\eta}=200m<256m=R_m^2.
\]
Therefore
\[
S_0(\widehat F_{m,U})\subseteq B(0,R_m),
\]
and so $\rho_m(x)=x$ throughout $S_0(\widehat F_{m,U})$.

Fix now $x\in S_0(\widehat F_{m,U})$ and decompose it orthogonally as
\[
y:=U^\top x\in\R^{2m},
\qquad
w:=x-Uy\in \mathrm{range}(U)^\perp.
\]
Using~\eqref{eq:hatF-def}, $\rho_m(x)=x$, and $U^\top U=I_{2m}$, we have
$\norm{x}^2=\norm{y}^2+\norm{w}^2$, and therefore
\[
\widehat F_{m,U}(x)=F_{2m}(y)+\frac{\eta}{2}\norm{y}^2+\frac{\eta}{2}\norm{w}^2=\Gamma_{2m}(y)+\frac{\eta}{2}\norm{w}^2.
\]
In particular, $y\in S_0(\Gamma_{2m})$.

Let $y^\star$ be a minimizer of $\Gamma_{2m}$.
Since $\Gamma_{2m}(y^\star)\le \Gamma_{2m}(0)<5m$, the same argument as above gives
\[
\norm{y^\star}^2<\frac{10m}{\eta}=200m<R_m^2,
\]
so $\rho_m(Uy^\star)=Uy^\star$.
For arbitrary $x$,
\[
\norm{U^\top \rho_m(x)}\le \norm{\rho_m(x)}\le \norm{x},
\]
and hence, by~\eqref{eq:hatF-def},
\[
\widehat F_{m,U}(x)=F_{2m}(U^\top \rho_m(x))+\frac{\eta}{2}\norm{x}^2\ge \Gamma_{2m}(U^\top \rho_m(x))\ge \Gamma_{2m}(y^\star).
\]
On the other hand, using~\eqref{eq:hatF-def} again and $\rho_m(Uy^\star)=Uy^\star$,
\[
\widehat F_{m,U}(Uy^\star)=\Gamma_{2m}(y^\star),
\]
so
\[
\widehat F_{m,U,\star}=\Gamma_{2m,\star}.
\]

We next compare objective gaps and gradients.
Differentiating~\eqref{eq:hatF-def} on $S_0(\widehat F_{m,U})$, where $\rho_m(x)=x$, gives
\[
\nabla \widehat F_{m,U}(x)=U\nabla F_{2m}(y)+\eta x=U\nabla \Gamma_{2m}(y)+\eta w.
\]
The first term lies in $\mathrm{range}(U)$, while the second lies in its orthogonal complement; therefore
\[
\norm{\nabla \widehat F_{m,U}(x)}^2=\norm{\nabla \Gamma_{2m}(y)}^2+\eta^2\norm{w}^2.
\]
Using Proposition~\ref{prop:Gamma-quadratic-PL},
\[
\Gamma_{2m}(y)-\Gamma_{2m,\star}\le 2400m\norm{\nabla \Gamma_{2m}(y)}^2,
\]
and hence
\[
\begin{aligned}
\widehat F_{m,U}(x)-\widehat F_{m,U,\star}
&=
\bigl(\Gamma_{2m}(y)-\Gamma_{2m,\star}\bigr)+\frac{\eta}{2}\norm{w}^2\\
&\le
2400m\norm{\nabla \Gamma_{2m}(y)}^2+\frac{1}{2\eta}\eta^2\norm{w}^2\\
&\le
2400m\norm{\nabla \Gamma_{2m}(y)}^2+2400m\,\eta^2\norm{w}^2\\
&=
2400m\norm{\nabla \widehat F_{m,U}(x)}^2,
\end{aligned}
\]
where we used $(2\eta)^{-1}=10\le 2400m$ because $m\ge 1$.

\emph{Step 5: progress barrier after random rotation.}
Fix a randomized algorithm $A$ in the gradient-only stochastic first-order model of this paper.
We simulate it by an algorithm $\widetilde A$ that operates in the stronger oracle model of Lemma~\ref{lem:random-rotation-progress} for the rotated base oracle
\[
\widetilde O_{m,q,U}(y,z):=\bigl(F_{2m}(U^\top y),\, Ug_{m,q}(U^\top y,z)\bigr).
\]
Whenever the simulated algorithm $A$ queries a point $x$, the algorithm $\widetilde A$ queries
\[
y:=\rho_m(x).
\]
Because $\norm{\rho_m(x)}\le 2R_m$, every query of $\widetilde A$ is bounded by $2R_m$.
From the returned pair
\[
\bigl(F_{2m}(U^\top y),\, Ug_{m,q}(U^\top y,z)\bigr)
\]
and the known point $x$, the simulator can reconstruct, using~\eqref{eq:hatF-def}, the corresponding objective value
\[
\widehat F_{m,U}(x)=F_{2m}(U^\top y)+\frac{\eta}{2}\norm{x}^2
\]
and, more importantly for the gradient-only algorithm \(A\), the stochastic-gradient response from~\eqref{eq:hatg-def},
\[
\widehat g_{m,q,U}(x,z)=J\rho_m(x)^\top Ug_{m,q}(U^\top y,z)+\eta x.
\]
The value component is available only because Lemma~\ref{lem:random-rotation-progress} is stated for a stronger oracle;
the simulated algorithm \(A\) uses only the stochastic-gradient response. Hence the trajectory produced by the simulation
is exactly the trajectory of $A$ on the clipped oracle, and the sequence
\[
y_t:=\rho_m(x_t^A)
\]
is a bounded-query trajectory in the stronger model.

The point \(x_t^A\) returned after \(t\) oracle calls is measurable with respect to the first \(t\) responses. We may
therefore append \(y_t:=\rho_m(x_t^A)\) as the next query of the simulating algorithm and ignore the response; this does
not give \(A\) any additional information. Apply Lemma~\ref{lem:random-rotation-progress} to that appended query with
$D=2m$, $M=m$, $R=2R_m$, and $\delta=1/4$.
The definition~\eqref{eq:dmq-threshold} verifies both required inequalities, namely $d\ge2m$ and
\[
d\ge18(2R_m)^2(2m)q^{-1}
\log\left(\frac{2(2m)^2}{q(1/4)}\right).
\]
Consequently, with probability at least $3/4$ over the random rotation $U$, the oracle randomness, and the internal
randomness of $A$,
\[
\prog_{1/4}(U^\top y_t)<m
\qquad
\text{whenever }t+1\le \frac{m-\log 8}{2q}.
\]
Since $m\ge m_0=5$ and $q\le1$, the bound
\[
\frac{m}{16q}+1\le \frac{m-\log 8}{2q}
\]
holds, so the same event holds for every
\[
t\le c_0\frac{m}{q},
\qquad
c_0:=\frac1{16}.
\]
On this event, Lemma~\ref{lem:objective-barrier} implies
\[
F_{2m}(U^\top y_t)\ge 2m.
\]
Therefore, by~\eqref{eq:hatF-def} and nonnegativity of the quadratic regularizer,
\[
\widehat F_{m,U}(x_t^A)-\widehat F_{m,U,\star}\ge 2m-\widehat F_{m,U,\star}.
\]
Since $\widehat F_{m,U,\star}=\Gamma_{2m,\star}$ and
\[
\Gamma_{2m,\star}\le \Gamma_{2m}(\one_{2m})=\eta m,
\]
we conclude that on the good event,
\[
\widehat F_{m,U}(x_t^A)-\widehat F_{m,U,\star}\ge (2-\eta)m.
\]
Taking expectations gives
\[
\E_{U,\xi}\bigl[\widehat F_{m,U}(x_t^A)-\widehat F_{m,U,\star}\bigr]
\ge
\frac34(2-\eta)m>m.
\]
This completes the proof.
\end{proof}

\subsection{The normalized hard family}
\label{app:sto-normalized-family}

For $m\ge 1$, define
\[
a_m:=m^{-1/(2-\alpha)}.
\]
We use the scaled version of the clipped instance defined in~\eqref{eq:hatF-def}--\eqref{eq:hatg-def} as follows:
\begin{equation}
\label{eq:scaled-family-def}
\Psi_{m,q,U}(x):=a_m^2\widehat F_{m,U}(x/a_m),
\qquad
G_{m,q,U}(x,z):=a_m\,\widehat g_{m,q,U}(x/a_m,z).
\end{equation}

\begin{lemma}[Scaled properties]
\label{lem:scaled-properties}
Use the notation fixed before Lemma~\ref{lem:clipped-rotated}, and let \(c_0\) be the numerical constant from that
lemma. There exist numerical constants \(L_0,\tau_0,\Delta_0^\ast>0\) such that, for every \(m\ge m_0\),
\(q\in(0,1]\), \(d\ge d_{m,q}\), and fixed \(U\in\Ortho(d,2m)\), the scaled pair
in~\eqref{eq:scaled-family-def} satisfies:
\begin{enumerate}[label=\arabic*.,leftmargin=1.5em]
    \item $\nabla \Psi_{m,q,U}$ is globally $L_0$-Lipschitz on $\R^d$.
    \item
    \[
    \Psi_{m,q,U}(0)-\Psi_{m,q,U,\star}\le \Delta_0^\ast a_m^\alpha.
    \]
    \item $\Psi_{m,q,U}$ satisfies the sublevel $\alpha$-P\L{} inequality
    \[
    \Psi_{m,q,U}(x)-\Psi_{m,q,U,\star}\le \tau_0\norm{\nabla \Psi_{m,q,U}(x)}^\alpha
    \qquad
    (x\in S_0(\Psi_{m,q,U})).
    \]
    \item \(G_{m,q,U}\) is an unbiased gradient estimator for \(\Psi_{m,q,U}\), and its conditional variance satisfies
    \[
    \E\bigl[\norm{G_{m,q,U}(x,z)-\nabla \Psi_{m,q,U}(x)}^2\mid x\bigr]\le
    \frac{4a_m^2}{q}.
    \]
    \item Moreover, if \(U\) is drawn uniformly from \(\Ortho(d,2m)\), \(A\) is any randomized first-order algorithm run with the
    oracle \(G_{m,q,U}\), and
    \[
    t\le c_0\frac{m}{q},
    \]
    then
    \[
    \E_{U,\xi}\bigl[\Psi_{m,q,U}(x_t^A)-\Psi_{m,q,U,\star}\bigr]\ge a_m^\alpha.
    \]
    Here \(x_t^A\) is the point generated by \(A\), and the expectation is over \(U\), the oracle randomness, and the
    internal randomness of \(A\).
\end{enumerate}
\end{lemma}

\begin{proof}
Write \(a=a_m\). From~\eqref{eq:scaled-family-def},
\[
\Psi_{m,q,U,\star}=a^2\widehat F_{m,U,\star},
\qquad
S_0(\Psi_{m,q,U})=a\,S_0(\widehat F_{m,U}),
\qquad
\nabla\Psi_{m,q,U}(x)=a\,\nabla\widehat F_{m,U}(x/a).
\]
Thus the Lipschitz constant of the gradient is unchanged by this scaling, while objective gaps are multiplied by
\(a^2\). Lemma~\ref{lem:clipped-rotated} gives
\[
\mathrm{Lip}(\nabla \widehat F_{m,U})\le 1829,
\qquad
\widehat F_{m,U}(0)-\widehat F_{m,U,\star}\le 5m.
\]
Since \(ma^2=a^\alpha\), items (1) and (2) follow.
In particular, we may take
\[
L_0:=1829,
\qquad
\Delta_0^\ast:=5.
\]

For item (4), the same scaling gives
\[
\E[G_{m,q,U}(x,z)\mid x]
=
a\,\E[\widehat g_{m,q,U}(x/a,z)\mid x]
=
a\,\nabla\widehat F_{m,U}(x/a)
=
\nabla\Psi_{m,q,U}(x),
\]
and Lemma~\ref{lem:clipped-rotated}(2) gives
\[
\E\bigl[\norm{G_{m,q,U}(x,z)-\nabla\Psi_{m,q,U}(x)}^2\mid x\bigr]
=
a^2\,\E\bigl[\norm{\widehat g_{m,q,U}(x/a,z)-\nabla\widehat F_{m,U}(x/a)}^2\mid x\bigr]
\le
\frac{4a^2}{q}.
\]

It remains to prove items (3) and (5).

\emph{Proof of item (3).}
The scaling identities above and Lemma~\ref{lem:clipped-rotated}(4) yield the quadratic-P\L{} bound
\[
\Psi_{m,q,U}(x)-\Psi_{m,q,U,\star}\le 2400m\norm{\nabla \Psi_{m,q,U}(x)}^2
\qquad
(x\in S_0(\Psi_{m,q,U})).
\]
Fix $x\in S_0(\Psi_{m,q,U})$ and write
\[
s:=\norm{\nabla \Psi_{m,q,U}(x)}.
\]
If $s\le a_m$, then, since $ma_m^{2-\alpha}=1$,
\[
2400ms^2
=
2400(ma_m^{2-\alpha})s^\alpha\left(\frac{s}{a_m}\right)^{2-\alpha}
\le
2400s^\alpha.
\]
If \(s\ge a_m\), then \(x\in S_0(\Psi_{m,q,U})\) and item (2) give
\[
\Psi_{m,q,U}(x)-\Psi_{m,q,U,\star}\le 5a_m^\alpha\le 5s^\alpha.
\]
Thus item (3) holds with
\[
\tau_0:=2400.
\]

\emph{Proof of item (5).}
Fix a randomized first-order algorithm $A$ for $(\Psi_{m,q,U},G_{m,q,U})$.
We simulate it by an algorithm $\widetilde A$ for $(\widehat F_{m,U},\widehat g_{m,q,U})$ as follows.
Whenever $A$ queries a point $x_t$, the simulator queries
\[
y_t:=x_t/a_m.
\]
When the base oracle returns $\widehat g_{m,q,U}(y_t,z_t)$, the simulator feeds back the scaled response
\[
a_m\widehat g_{m,q,U}(y_t,z_t)=G_{m,q,U}(x_t,z_t)
\]
to the internal copy of $A$.
Consequently the two trajectories are coupled by $x_t=a_my_t$, and
\[
\Psi_{m,q,U}(x_t)-\Psi_{m,q,U,\star}
=
a_m^2\bigl(\widehat F_{m,U}(y_t)-\widehat F_{m,U,\star}\bigr).
\]
Applying Lemma~\ref{lem:clipped-rotated}(5) to $\widetilde A$ yields
\[
\E_{U,\xi}\bigl[\Psi_{m,q,U}(x_t)-\Psi_{m,q,U,\star}\bigr]
\ge
a_m^2 m
=
a_m^\alpha.
\]
\end{proof}

\subsection{The normalized stochastic lower bound}
\label{app:sto-normalized-lb}

For notational simplicity, specialize the pair-family definition~\eqref{eq:stochastic-pair-family} to
\((L,\tau,\Delta_0,\sigma^2)=(L_0,\tau_0,1,\bar\sigma^2)\) and write the resulting normalized pair class as
\begin{equation}
\label{eq:normalized-pair-class}
\mathfrak P_{\mathrm{norm}}(\bar \sigma^2):=\mathfrak P_{\alpha,L_0,\tau_0,1,\bar \sigma^2}(0).
\end{equation}
Set
\[
C_\alpha^{\mathrm{amp}}:=2^{1+p}=2^{2/(2-\alpha)},
\qquad
c_\alpha^{\mathrm{noise}}:=2^{-\alpha}(C_\alpha^{\mathrm{amp}})^{-1}=2^{-\alpha-2/(2-\alpha)}.
\]

Define the explicit normalized accuracy threshold
\begin{equation}
\label{eq:normalized-accuracy-threshold}
\varepsilon_\alpha^0
:=
\min\left\{
1,\,
2^{-(p+1)},\,
\frac{1}{2(m_0+1)^p},\,
\frac{1}{5C_\alpha^{\mathrm{amp}}}
\right\}.
\end{equation}

\begin{proposition}[Normalized stochastic lower bound]
\label{prop:normalized-stochastic-lb}
There exists $c_\alpha^{\mathrm{st}}>0$ such that, for every $\bar \sigma^2>0$ and every
$\varepsilon\in(0,\varepsilon_\alpha^0]$ satisfying
\[
\varepsilon\le c_\alpha^{\mathrm{noise}}\bar \sigma^\alpha,
\]
there exists a subfamily of the normalized pair class \eqref{eq:normalized-pair-class},
\[
\mathfrak H_{\varepsilon,\bar \sigma^2}\subseteq \mathfrak P_{\mathrm{norm}}(\bar \sigma^2)
\]
such that
\[
T_\varepsilon^{\mathrm{BV}}(\mathfrak H_{\varepsilon,\bar \sigma^2})
\ge
c_\alpha^{\mathrm{st}}\bar \sigma^2\varepsilon^{-(4-\alpha)/\alpha},
\]
where \(T_\varepsilon^{\mathrm{BV}}(\cdot)\) for pair subfamilies is defined in
\eqref{eq:pair-family-minimax},
and every $(f,G)\in \mathfrak H_{\varepsilon,\bar \sigma^2}$ satisfies
\[
f(0)-f_\star\le 5C_\alpha^{\mathrm{amp}} \varepsilon.
\]
Consequently,
\[
T_\varepsilon^{\mathrm{BV}}\bigl(\mathfrak P_{\mathrm{norm}}(\bar \sigma^2)\bigr)
\ge
c_\alpha^{\mathrm{st}}\bar \sigma^2\varepsilon^{-(4-\alpha)/\alpha}.
\]
\end{proposition}

\begin{proof}
We choose the chain length so that the normalized objective gap is of order $\varepsilon$.
Set
\[
m:=\left\lfloor (2\varepsilon)^{-1/p}\right\rfloor.
\]
Since
\[
m\le (2\varepsilon)^{-1/p}
\qquad\text{and}\qquad
m\ge \frac12(2\varepsilon)^{-1/p},
\]
it follows that
\[
a_m^\alpha=m^{-p}\ge 2\varepsilon,
\qquad
a_m^\alpha\le C_\alpha^{\mathrm{amp}} \varepsilon.
\]
Thus the natural gap scale of the hard instance is comparable to $\varepsilon$.

Next we choose the reveal probability so that the variance budget equals $\bar \sigma^2$.
From the upper bound on $a_m^\alpha$ we obtain
\[
a_m^2
=
(a_m^\alpha)^{2/\alpha}
\le
(C_\alpha^{\mathrm{amp}})^{2/\alpha}\varepsilon^{2/\alpha}
\le
\frac{\bar \sigma^2}{4},
\]
where the last step follows from
\[
(C_\alpha^{\mathrm{amp}})^{2/\alpha}\bigl(c_\alpha^{\mathrm{noise}}\bar\sigma^\alpha\bigr)^{2/\alpha}
=
(C_\alpha^{\mathrm{amp}})^{2/\alpha}
\bigl(2^{-\alpha}(C_\alpha^{\mathrm{amp}})^{-1}\bar\sigma^\alpha\bigr)^{2/\alpha}
=
\frac{\bar\sigma^2}{4}.
\]
Set
\[
q:=\frac{4a_m^2}{\bar \sigma^2}\in (0,1].
\]

Now fix any $d\ge d_{m,q}$ and define the hard subfamily
\begin{equation}
\label{eq:normalized-hard-subfamily}
\mathfrak H_{\varepsilon,\bar \sigma^2}
:=
\{(\Psi_{m,q,U},G_{m,q,U}):\ U\in \Ortho(d,2m)\}.
\end{equation}
By Lemma~\ref{lem:scaled-properties}(4), every oracle in this family has variance at most $\bar \sigma^2$.
By the choice of \(\varepsilon_\alpha^0\), we also have
\[
5a_m^\alpha\le 5C_\alpha^{\mathrm{amp}} \varepsilon\le 1,
\]
and therefore Lemma~\ref{lem:scaled-properties}(1)--(3) show that
\[
\mathfrak H_{\varepsilon,\bar \sigma^2}\subseteq \mathfrak P_{\mathrm{norm}}(\bar \sigma^2).
\]

Now let $A$ be any randomized first-order method.
Lemma~\ref{lem:scaled-properties}(5) gives
\[
\E_{U,\xi}\bigl[\Psi_{m,q,U}(x_t^A)-\Psi_{m,q,U,\star}\bigr]\ge a_m^\alpha\ge 2\varepsilon
\qquad
\text{for all }t\le c_0 m/q.
\]
Consequently,
\[
\sup_{(f,G)\in \mathfrak H_{\varepsilon,\bar \sigma^2}}\E_\xi[f(x_t^A)-f_\star]
\ge
\E_{U,\xi}[\Psi_{m,q,U}(x_t^A)-\Psi_{m,q,U,\star}]
\ge
2\varepsilon.
\]
Since this holds for every algorithm \(A\), the fixed-budget risk
\(\mathcal E_t^{\mathrm{BV}}(\mathfrak H_{\varepsilon,\bar\sigma^2})\) is larger than \(\varepsilon\) for every
\(t\le c_0m/q\). Therefore, by~\eqref{eq:pair-family-minimax},
\[
T_\varepsilon^{\mathrm{BV}}(\mathfrak H_{\varepsilon,\bar \sigma^2})\ge c_0 m/q.
\]

Finally, since $m=a_m^{-(2-\alpha)}$,
\[
\frac{m}{q}
=
\frac{\bar \sigma^2}{4a_m^2}\,m
=
\frac{\bar \sigma^2}{4}a_m^{-(4-\alpha)}
\ge
\frac{\bar \sigma^2}{4}(C_\alpha^{\mathrm{amp}})^{-(4-\alpha)/\alpha}
\varepsilon^{-(4-\alpha)/\alpha}.
\]
This proves the lower bound with
\[
c_\alpha^{\mathrm{st}}:=\frac{c_0}{4}(C_\alpha^{\mathrm{amp}})^{-(4-\alpha)/\alpha}.
\]
The initial-gap claim follows from
\[
\Psi_{m,q,U}(0)-\Psi_{m,q,U,\star}\le 5a_m^\alpha\le 5C_\alpha^{\mathrm{amp}} \varepsilon.
\]
Since \(\mathfrak H_{\varepsilon,\bar \sigma^2}\subseteq \mathfrak P_{\mathrm{norm}}(\bar\sigma^2)\), enlarging the
pair class can only increase the worst-case risk. Hence the same lower bound holds for
\(\mathfrak P_{\mathrm{norm}}(\bar\sigma^2)\).
\end{proof}

\subsection{\texorpdfstring{Scaling to arbitrary $(L,\tau,\Delta_0,\sigma^2)$}{Scaling to arbitrary (L,tau,Delta0,sigma2)}}
\label{app:sto-scaling}

The normalized lower bound above has fixed smoothness and P\L{} constants \(L_0,\tau_0\). The next lemma records how
objective scaling and domain scaling change smoothness, the sublevel \(\alpha\)-P\L{} constant, the stochastic-gradient
variance, and the lower bound itself. It is the tool used to transfer Proposition~\ref{prop:normalized-stochastic-lb}
to arbitrary \(L,\tau,\sigma^2\).

\begin{lemma}[Stochastic scaling]
\label{lem:stochastic-scaling}
Let $(f,G)$ be a differentiable function-oracle pair on $\R^d$, and let $\lambda,\rho>0$.
Define
\[
(S_{\lambda,\rho}f)(x):=\lambda f(\rho x),
\qquad
(S_{\lambda,\rho}G)(x,\xi):=\lambda \rho\, G(\rho x,\xi).
\]
Then:
\begin{enumerate}[label=\arabic*.,leftmargin=1.5em]
    \item if $\nabla f$ is globally $L_f$-Lipschitz on $\R^d$, then $\nabla (S_{\lambda,\rho}f)$ is globally at most $(\lambda \rho^2 L_f)$-Lipschitz on $\R^d$;
    \item if
    \[
    f(y)-f_\star\le \tau_f\norm{\nabla f(y)}^\alpha
    \qquad
    (\forall y\in S_0(f)),
    \]
    then
    \[
    (S_{\lambda,\rho}f)(x)-(S_{\lambda,\rho}f)_\star
    \le
    \tau_f \lambda^{1-\alpha}\rho^{-\alpha}\norm{\nabla (S_{\lambda,\rho}f)(x)}^\alpha
    \qquad
    (\forall x\in S_0(S_{\lambda,\rho}f));
    \]
    \item if $G$ is unbiased with variance at most $\bar \sigma^2$, then $S_{\lambda,\rho}G$ is unbiased with variance at most $(\lambda \rho)^2\bar \sigma^2$;
    \item for any pair subfamily \(\mathfrak H\), define
    \[
    S_{\lambda,\rho}\mathfrak H
    :=
    \{(S_{\lambda,\rho}f,S_{\lambda,\rho}G):\ (f,G)\in\mathfrak H\}.
    \]
    Then, for every \(T\ge0\) and \(\varepsilon>0\),
    \[
    \mathcal E_T^{\mathrm{BV}}(S_{\lambda,\rho}\mathfrak H)
    \ge
    \lambda\,\mathcal E_T^{\mathrm{BV}}(\mathfrak H),
    \qquad
    T_\varepsilon^{\mathrm{BV}}(S_{\lambda,\rho}\mathfrak H)
    \ge
    T_{\varepsilon/\lambda}^{\mathrm{BV}}(\mathfrak H).
    \]
\end{enumerate}
\end{lemma}

\begin{proof}
We prove the four claims in order.

For item (1), let $x,y\in\R^d$.
Since
\[
\nabla (S_{\lambda,\rho}f)(x)=\lambda \rho \nabla f(\rho x),
\]
we obtain
\[
\norm{\nabla (S_{\lambda,\rho}f)(x)-\nabla (S_{\lambda,\rho}f)(y)}
\le
\lambda \rho L_f \norm{\rho x-\rho y}
=
\lambda \rho^2 L_f \norm{x-y}.
\]

For item (2), fix $x\in S_0(S_{\lambda,\rho}f)$ and set $y:=\rho x$.
Then $y\in S_0(f)$ and
\[
(S_{\lambda,\rho}f)(x)-(S_{\lambda,\rho}f)_\star
=
\lambda \bigl(f(y)-f_\star\bigr)
\le
\lambda \tau_f \norm{\nabla f(y)}^\alpha.
\]
Using $\nabla (S_{\lambda,\rho}f)(x)=\lambda \rho \nabla f(y)$ gives
\[
(S_{\lambda,\rho}f)(x)-(S_{\lambda,\rho}f)_\star
\le
\tau_f \lambda^{1-\alpha}\rho^{-\alpha}\norm{\nabla (S_{\lambda,\rho}f)(x)}^\alpha.
\]

For item (3), unbiasedness follows from
\[
\E[(S_{\lambda,\rho}G)(x,\xi)\mid x]
=
\lambda \rho \E[G(\rho x,\xi)\mid x]
=
\lambda \rho \nabla f(\rho x)
=
\nabla (S_{\lambda,\rho}f)(x).
\]
Moreover,
\[
\E\bigl[\norm{(S_{\lambda,\rho}G)(x,\xi)-\nabla (S_{\lambda,\rho}f)(x)}^2\mid x\bigr]
=
(\lambda \rho)^2 \E\bigl[\norm{G(\rho x,\xi)-\nabla f(\rho x)}^2\mid x\bigr]
\le
(\lambda \rho)^2 \bar \sigma^2.
\]

For item (4), fix an algorithm \(A\) for the scaled family \(S_{\lambda,\rho}\mathfrak H\). We construct an algorithm
\(\widetilde A\) for the base family \(\mathfrak H\) as follows. Whenever \(A\) queries \(x_t\), the simulator queries
\[
y_t:=\rho x_t
\]
to the base stochastic-gradient oracle and feeds back the synthetic scaled stochastic gradient
\[
(S_{\lambda,\rho}G)(x_t,\xi)=\lambda \rho\, G(y_t,\xi).
\]
If \(A\) outputs \(x_T\), the base algorithm outputs \(y_T=\rho x_T\). For every \((f,G)\in\mathfrak H\), the simulated
trajectory is exactly the trajectory of \(A\) on \((S_{\lambda,\rho}f,S_{\lambda,\rho}G)\), and
\[
(S_{\lambda,\rho}f)(x_T)-(S_{\lambda,\rho}f)_\star
=
\lambda\bigl(f(y_T)-f_\star\bigr).
\]
Taking the supremum over \((f,G)\in\mathfrak H\) and then the infimum over \(A\) gives
\[
\mathcal E_T^{\mathrm{BV}}(S_{\lambda,\rho}\mathfrak H)
\ge
\lambda\,\mathcal E_T^{\mathrm{BV}}(\mathfrak H).
\]
The displayed lower bound on \(T_\varepsilon^{\mathrm{BV}}\) follows directly from the definition
\eqref{eq:pair-family-minimax}.
\end{proof}

\begin{lemma}[Exact-gradient lower bound for randomized BV methods]
\label{lem:bv-deterministic-term}
There are constants $c_{\alpha,\mathrm{acc}}^{\mathrm{det}}>0$ and
$c_{\alpha,\mathrm{lb}}^{\mathrm{det}}>0$, depending only on $\alpha$, such that, for every
$L,\tau,\Delta_0,\sigma^2>0$ and every
\[
\varepsilon\le c_{\alpha,\mathrm{acc}}^{\mathrm{det}}
\min\{\Delta_0,L^{\alpha/(2-\alpha)}\tau^{2/(2-\alpha)}\},
\]
the bounded-variance stochastic minimax complexity satisfies
\begin{equation}
\label{eq:bv-deterministic-term}
T_\varepsilon^{\mathrm{BV}}(\alpha,L,\tau,\Delta_0,\sigma^2)
\ge
c_{\alpha,\mathrm{lb}}^{\mathrm{det}}
L\tau^{2/\alpha}\varepsilon^{-(2-\alpha)/\alpha}.
\end{equation}
This lower bound holds for randomized, adaptive, arbitrary-query, gradient-only methods.
\end{lemma}

\begin{proof}
Let $p:=\alpha/(2-\alpha)$, and let $L_0,\tau_0$ be the normalized constants in
Lemma~\ref{lem:scaled-properties}. Set
\[
\lambda:=
\left(\frac{\tau}{\tau_0}\right)^{2/(2-\alpha)}
\left(\frac{L}{L_0}\right)^{\alpha/(2-\alpha)},
\qquad
\rho:=\sqrt{\frac{L}{\lambda L_0}},
\qquad
\varepsilon':=\frac{\varepsilon}{\lambda}.
\]
Choose $c_{\alpha,\mathrm{acc}}^{\mathrm{det}}$ small enough that the displayed accuracy assumption implies
\[
\varepsilon'\le\varepsilon_\alpha^0,
\qquad
\left\lfloor(2\varepsilon')^{-1/p}\right\rfloor\ge
\max\{m_0,\lceil2/c_0\rceil\},
\qquad
5C_\alpha^{\mathrm{amp}}\varepsilon\le\Delta_0.
\]
This is possible because $\lambda$ is an $\alpha$-dependent constant multiple of
$L^{\alpha/(2-\alpha)}\tau^{2/(2-\alpha)}$.

Define
\[
m:=\left\lfloor(2\varepsilon')^{-1/p}\right\rfloor
\]
and use the normalized family~\eqref{eq:scaled-family-def} with $q=1$, any $d\ge d_{m,1}$, and
$U$ uniform on $\Ortho(d,2m)$. Since $z\sim\mathrm{Bernoulli}(1)$, we have $z=1$ almost surely. The definition of
$g_{m,q}$ then gives, for every $x$,
\[
g_{m,1}(x,z)=\nabla F_{2m}(x).
\]
Substituting this identity into~\eqref{eq:hatg-def} and then~\eqref{eq:scaled-family-def} yields
\[
G_{m,1,U}(x,z)=\nabla\Psi_{m,1,U}(x)
\qquad\text{almost surely}.
\]
Thus this is an exact-gradient oracle with variance zero, and hence it belongs to every variance class with budget
$\sigma^2>0$ after scaling.

The floor choice and $\varepsilon'\le\varepsilon_\alpha^0$ give
\[
m\le(2\varepsilon')^{-1/p},
\qquad
m\ge\frac12(2\varepsilon')^{-1/p},
\]
and therefore
\begin{equation}
\label{eq:q-one-amplitude-bounds}
a_m^\alpha=m^{-p}\ge2\varepsilon',
\qquad
m\ge2^{-1-1/p}(\varepsilon')^{-1/p}
=\Omega_\alpha\bigl((\varepsilon')^{-(2-\alpha)/\alpha}\bigr).
\end{equation}
Moreover, Lemma~\ref{lem:scaled-properties}(2) gives
\[
\Psi_{m,1,U}(0)-\Psi_{m,1,U,\star}
\le5a_m^\alpha
\le5C_\alpha^{\mathrm{amp}}\varepsilon'\le1.
\]
Hence every unscaled pair lies in the normalized pair class.

Now fix any randomized, adaptive, arbitrary-query gradient-only method $A$. Lemma~\ref{lem:scaled-properties}(5),
whose proof uses the corrected random-rotation Lemma~\ref{lem:random-rotation-progress}, gives
\[
\E_{U,A}\bigl[
\Psi_{m,1,U}(x_t^A)-\Psi_{m,1,U,\star}
\bigr]
\ge a_m^\alpha
\ge2\varepsilon'
\qquad
(t\le c_0m).
\]
There is no oracle randomness here because the $q=1$ oracle is exact; the expectation is over $U$ and the internal
randomness of $A$. Consequently,
\[
\sup_{U\in\Ortho(d,2m)}
\E_A\bigl[\Psi_{m,1,U}(x_t^A)-\Psi_{m,1,U,\star}\bigr]
\ge
\E_{U,A}\bigl[\Psi_{m,1,U}(x_t^A)-\Psi_{m,1,U,\star}\bigr]
\ge2\varepsilon'.
\]
Thus a genuine common distribution over $U$ is hard for every randomized method; no deterministic
resisting-oracle distribution and no interchange of an infimum with an expectation is used. The proof of
Lemma~\ref{lem:clipped-rotated}(5) already handles the output convention by appending the point returned after \(t\)
calls as the next query without using the response; its choice \(c_0=1/16\) absorbs that one-round shift. It follows from
\eqref{eq:q-one-amplitude-bounds} that the normalized exact-gradient family requires
\[
\Omega_\alpha\bigl((\varepsilon')^{-(2-\alpha)/\alpha}\bigr)
\]
calls.

Apply the stochastic scaling Lemma~\ref{lem:stochastic-scaling} to every pair in this family. The selected $\lambda$
and $\rho$ give smoothness $L$ and sublevel $\alpha$-P\L{} constant $\tau$. The scaled initial gap satisfies
\[
\lambda\bigl(\Psi_{m,1,U}(0)-\Psi_{m,1,U,\star}\bigr)
\le5C_\alpha^{\mathrm{amp}}\lambda\varepsilon'
=5C_\alpha^{\mathrm{amp}}\varepsilon
\le\Delta_0,
\]
and the scaled exact-gradient oracle still has variance zero. Finally,
\[
\lambda^{(2-\alpha)/\alpha}
=
\frac{L}{L_0}
\left(\frac{\tau}{\tau_0}\right)^{2/\alpha}.
\]
Therefore scaling the normalized lower bound gives
\[
T_\varepsilon^{\mathrm{BV}}(\alpha,L,\tau,\Delta_0,\sigma^2)
\ge
\Omega_\alpha\left(
\lambda^{(2-\alpha)/\alpha}\varepsilon^{-(2-\alpha)/\alpha}
\right)
=
\Omega_\alpha\left(
L\tau^{2/\alpha}\varepsilon^{-(2-\alpha)/\alpha}
\right),
\]
where the fixed factors $L_0^{-1}\tau_0^{-2/\alpha}$ are absorbed into a constant depending only on $\alpha$.
This proves~\eqref{eq:bv-deterministic-term}.
\end{proof}

\begin{theorem}[Matching bounded-variance lower bound]
\label{thm:stochastic-lb}
For every $\alpha\in[1,2)$ there exists $c_\alpha>0$ such that, for every $L,\tau,\Delta_0,\sigma^2>0$ and every
target accuracy
\[
\varepsilon\in\bigl(0,c_\alpha\min\{\Delta_0,L^{\alpha/(2-\alpha)}\tau^{2/(2-\alpha)}\}\bigr],
\]
\[
T_\varepsilon^{\mathrm{BV}}(\alpha,L,\tau,\Delta_0,\sigma^2)
\ge
c_\alpha\left(
L\tau^{2/\alpha}\varepsilon^{-(2-\alpha)/\alpha}
 + L\sigma^2\tau^{4/\alpha}\varepsilon^{-(4-\alpha)/\alpha}
\right)
\]
holds.
\end{theorem}

\begin{proof}
Let $(L_0,\tau_0)$ be the normalized constants from Lemma~\ref{lem:scaled-properties}.
We begin by choosing a scaling that transforms the normalized class into the target class.
Set
\[
\lambda:=\left(\frac{\tau}{\tau_0}\right)^{2/(2-\alpha)}\left(\frac{L}{L_0}\right)^{\alpha/(2-\alpha)},
\qquad
\rho:=\sqrt{\frac{L}{\lambda L_0}}.
\]
Then
\[
L_0\lambda \rho^2=L,
\qquad
\tau_0 \lambda^{1-\alpha}\rho^{-\alpha}=\tau.
\]
Thus Lemma~\ref{lem:stochastic-scaling}(1)--(2) sends normalized instances to functions with the desired smoothness and sublevel $\alpha$-P\L{} constants.

Next define the normalized variance level
\[
\bar \sigma^2:=\frac{\sigma^2}{(\lambda \rho)^2}=\frac{\sigma^2 L_0}{\lambda L},
\]
and the normalized accuracy target
\[
\varepsilon':=\frac{\varepsilon}{\lambda}.
\]
The objective scaling factor is a constant multiple of the natural gap scale:
\[
\lambda
=
\tau_0^{-2/(2-\alpha)}L_0^{-\alpha/(2-\alpha)}
L^{\alpha/(2-\alpha)}\tau^{2/(2-\alpha)}.
\]
Set $c_\alpha^{\mathrm{norm}}:=\varepsilon_\alpha^0$, where
$\varepsilon_\alpha^0$ is the explicit threshold in~\eqref{eq:normalized-accuracy-threshold}. Then
Proposition~\ref{prop:normalized-stochastic-lb} applies whenever
$\varepsilon'\le c_\alpha^{\mathrm{norm}}$ and the normalized noise condition below holds.

We claim that the condition
\[
\varepsilon'\le c_\alpha^{\mathrm{noise}}\bar \sigma^\alpha
\]
is equivalent to
\[
\varepsilon\le c_\alpha^{\mathrm{noise}}\frac{\tau}{\tau_0}\sigma^\alpha.
\]
Indeed,
\[
\lambda \bar \sigma^\alpha
=
\lambda \left(\frac{\sigma^2 L_0}{\lambda L}\right)^{\alpha/2}
=
\sigma^\alpha \lambda^{1-\alpha/2}\left(\frac{L_0}{L}\right)^{\alpha/2}.
\]
Using
\[
\tau=\tau_0\lambda^{1-\alpha/2}\left(\frac{L_0}{L}\right)^{\alpha/2},
\]
we obtain
\[
\lambda \bar \sigma^\alpha=\frac{\tau}{\tau_0}\sigma^\alpha,
\]
which proves the claim.

Therefore, whenever
\[
\varepsilon\le c_\alpha^{\mathrm{noise}}\frac{\tau}{\tau_0}\sigma^\alpha
\qquad\text{and}\qquad
\varepsilon\le
c_\alpha^{\mathrm{norm}}\tau_0^{-2/(2-\alpha)}L_0^{-\alpha/(2-\alpha)}
L^{\alpha/(2-\alpha)}\tau^{2/(2-\alpha)},
\]
Proposition~\ref{prop:normalized-stochastic-lb} yields, inside the normalized pair class
\eqref{eq:normalized-pair-class} and with the hard-subfamily definition
\eqref{eq:normalized-hard-subfamily} applied at accuracy \(\varepsilon'\), a subfamily
\[
\mathfrak H_{\varepsilon',\bar \sigma^2}\subseteq \mathfrak P_{\mathrm{norm}}(\bar \sigma^2)
\]
such that
\[
T_{\varepsilon'}^{\mathrm{BV}}(\mathfrak H_{\varepsilon',\bar \sigma^2})
\ge
c_\alpha^{\mathrm{st}}\bar \sigma^2(\varepsilon')^{-(4-\alpha)/\alpha},
\]
and every $(f,G)\in \mathfrak H_{\varepsilon',\bar \sigma^2}$ satisfies
\[
f(0)-f_\star\le 5C_\alpha^{\mathrm{amp}} \varepsilon'.
\]
If, in addition,
\[
\varepsilon\le \frac{\Delta_0}{5C_\alpha^{\mathrm{amp}}},
\]
then for every $(f,G)\in \mathfrak H_{\varepsilon',\bar \sigma^2}$,
\[
\lambda \bigl(f(0)-f_\star\bigr)\le \Delta_0.
\]
Hence the scaled family
\[
S_{\lambda,\rho}\mathfrak H_{\varepsilon',\bar \sigma^2}
:=
\{(S_{\lambda,\rho}f,S_{\lambda,\rho}G):\ (f,G)\in \mathfrak H_{\varepsilon',\bar \sigma^2}\}
\]
is contained in $\mathfrak P_{\alpha,L,\tau,\Delta_0,\sigma^2}(0)$ by Lemma~\ref{lem:stochastic-scaling}(1)--(3).

We now transfer the normalized lower bound to the scaled family.
Using the pair-family complexity notation from \eqref{eq:pair-family-minimax}, Lemma~\ref{lem:stochastic-scaling}(4)
gives
\[
T_\varepsilon^{\mathrm{BV}}(S_{\lambda,\rho}\mathfrak H_{\varepsilon',\bar \sigma^2})
\ge
T_{\varepsilon/\lambda}^{\mathrm{BV}}(\mathfrak H_{\varepsilon',\bar \sigma^2})
=
T_{\varepsilon'}^{\mathrm{BV}}(\mathfrak H_{\varepsilon',\bar \sigma^2}).
\]
Since this scaled family is a subclass of $\mathfrak P_{\alpha,L,\tau,\Delta_0,\sigma^2}(0)$, we conclude that
\[
T_\varepsilon^{\mathrm{BV}}(\alpha,L,\tau,\Delta_0,\sigma^2)
\ge
c_\alpha^{\mathrm{st}}\lambda^{(4-\alpha)/\alpha}\bar \sigma^2 \varepsilon^{-(4-\alpha)/\alpha}.
\]

It remains to rewrite the prefactor in terms of $(L,\tau,\sigma^2)$.
Using the identities above,
\[
L\sigma^2\tau^{4/\alpha}
=
(L_0\lambda \rho^2)\bigl(\bar \sigma^2 \lambda^2 \rho^2\bigr)\bigl(\tau_0^{4/\alpha}\lambda^{4(1-\alpha)/\alpha}\rho^{-4}\bigr)
=
L_0\tau_0^{4/\alpha}\bar \sigma^2 \lambda^{(4-\alpha)/\alpha}.
\]
Therefore
\[
T_\varepsilon^{\mathrm{BV}}(\alpha,L,\tau,\Delta_0,\sigma^2)
\ge
c_\alpha^{\mathrm{lb}} L\sigma^2\tau^{4/\alpha}\varepsilon^{-(4-\alpha)/\alpha},
\]
where
\[
c_\alpha^{\mathrm{lb}}:=c_\alpha^{\mathrm{st}}(L_0\tau_0^{4/\alpha})^{-1}.
\]

Define
\[
D_\varepsilon:=L\tau^{2/\alpha}\varepsilon^{-(2-\alpha)/\alpha},
\qquad
S_\varepsilon:=L\sigma^2\tau^{4/\alpha}\varepsilon^{-(4-\alpha)/\alpha},
\qquad
\kappa_\alpha:=\frac{c_\alpha^{\mathrm{noise}}}{\tau_0}.
\]
Their ratio is
\begin{equation}
\label{eq:stochastic-deterministic-ratio}
\frac{S_\varepsilon}{D_\varepsilon}
=
\sigma^2\tau^{2/\alpha}\varepsilon^{-2/\alpha}
=
\left(\frac{\tau\sigma^\alpha}{\varepsilon}\right)^{2/\alpha}.
\end{equation}
The stochastic construction above applies when
$\varepsilon\le\kappa_\alpha\tau\sigma^\alpha$. In this regime, both the stochastic lower bound just proved and
Lemma~\ref{lem:bv-deterministic-term} are available. Writing
$c_D:=c_{\alpha,\mathrm{lb}}^{\mathrm{det}}$ and $c_S:=c_\alpha^{\mathrm{lb}}$, we obtain
\[
T_\varepsilon^{\mathrm{BV}}
\ge
\max\{c_DD_\varepsilon,c_SS_\varepsilon\}
\ge
\frac{\min\{c_D,c_S\}}{2}
\bigl(D_\varepsilon+S_\varepsilon\bigr).
\]
The last inequality follows from
\[
\max\{c_DD,c_SS\}
\ge
\min\{c_D,c_S\}\max\{D,S\}
\ge
\frac12\min\{c_D,c_S\}(D+S),
\]
so the sum is obtained only up to a universal constant factor.

In the complementary regime $\varepsilon>\kappa_\alpha\tau\sigma^\alpha$,
\eqref{eq:stochastic-deterministic-ratio} gives
\[
S_\varepsilon<\kappa_\alpha^{-2/\alpha}D_\varepsilon.
\]
The exact-gradient randomized lower bound in Lemma~\ref{lem:bv-deterministic-term} therefore gives
\[
T_\varepsilon^{\mathrm{BV}}
\ge
c_DD_\varepsilon
\ge
\frac{c_D}{1+\kappa_\alpha^{-2/\alpha}}
\bigl(D_\varepsilon+S_\varepsilon\bigr).
\]

It remains only to collect the accuracy conditions. Choose the constant $c_\alpha$ in the theorem no larger than
the accuracy constant $c_{\alpha,\mathrm{acc}}^{\mathrm{det}}$, $(5C_\alpha^{\mathrm{amp}})^{-1}$, and
\[
c_\alpha^{\mathrm{norm}}
\tau_0^{-2/(2-\alpha)}L_0^{-\alpha/(2-\alpha)},
\]
and no larger than the two lower-bound prefactors after the fixed normalized constants have been absorbed. Then
\[
\varepsilon\le
c_\alpha\min\{\Delta_0,L^{\alpha/(2-\alpha)}\tau^{2/(2-\alpha)}\}
\]
ensures every initial-gap and normalized-accuracy requirement in both regimes. The crossover
$\tau\sigma^\alpha$ is used only to choose the active proof and is not part of the theorem's final accuracy range.
This proves the stated sum lower bound.
\end{proof}

\begin{remark}
The final accuracy range does not include the crossover \(\tau\sigma^\alpha\):
\[
\varepsilon\le c_\alpha' \min\{\Delta_0,L^{\alpha/(2-\alpha)}\tau^{2/(2-\alpha)}\}.
\]
The scale \(\tau\sigma^\alpha\) only determines which part of the proof is active: below it, the probabilistic
zero-chain gives the stochastic-noise term; above it, the deterministic lower bound already controls the full sum.
\end{remark}

\subsection{Completion of the stochastic argument}
\label{app:sto-completion}

\begin{proof}[Proof of Theorem~\ref{thm:stochastic-matching}]
Apply Theorem~\ref{thm:stochastic-lb}. Its proof uses the probabilistic construction below the crossover
$\varepsilon\asymp\tau\sigma^\alpha$ and the $q=1$ exact-gradient construction above it; the crossover does not
restrict the theorem's overall accuracy range. Relabeling the final $\alpha$-dependent constant gives the statement.
\end{proof}

\subsection{\texorpdfstring{Proof of the variance-dependent lower bound under global P\L{} geometry}{Proof of the variance-dependent lower bound under global PL geometry}}
\label{app:sto-endpoint-proof}

\begingroup
The proof combines two arguments. When $L\tau$ is larger than a universal threshold, we rescale the clipped,
randomly rotated probabilistic zero chain from Lemma~\ref{lem:clipped-rotated}. A one-dimensional Gaussian two-point
construction supplies the same lower bound over the remaining bounded range of $L\tau$. To place the zero-chain
instances in the global-P\L{} subclass, we first prove a quantitative P\L{} inequality on the whole ambient space.
We then summarize the normalized oracle properties and treat the two ranges of $L\tau$ in order.
\endgroup

\begingroup
The next lemma is the additional geometric ingredient needed for the strengthened endpoint theorem. Lemma~\ref{lem:clipped-rotated}
already gives a P\L{} estimate on the initial sublevel set. Here we show that its objective actually satisfies a global
P\L{} inequality, with a constant proportional to the chain length. The proportional dependence on $m$ is essential:
after scaling, it is what produces the factor $L\tau$ in the lower bound.

\begin{lemma}[Global regularity and P\L{} geometry of the endpoint instance]
\label{lem:endpoint-global-pl}
For every $m\ge m_0$, every $d\ge2m$, and every $U\in\Ortho(d,2m)$, the gradient of
$\widehat F_{m,U}$ is globally $113$-Lipschitz and
\[
\widehat F_{m,U}(x)-\widehat F_{m,U,\star}
\le
28{,}778m\norm{\nabla\widehat F_{m,U}(x)}^2
\qquad
(\forall x\in\R^d).
\]
\end{lemma}

\begin{proof}
Set $N:=2m$, recall that $\eta=1/20$, and note that
$R_m=16\sqrt m=8\sqrt2\,\sqrt N$. We first sharpen the scalar estimates used by the endpoint construction. We then
prove the smoothness and global P\L{} bounds separately.

\emph{Step 1: global scalar and chain estimates.}
Write
\[
F_N(x)=H_N(x)+\sum_{i=1}^N B(x_i),
\qquad
H_N(x):=\varphi(x_1)+\sum_{i=1}^{N-1}\Theta(x_i)\varphi(x_{i+1}).
\]
Because $0\le\Theta,\varphi\le1$, the $N$ summands in $H_N$ give $0\le H_N\le N$. Each coordinate of
$\nabla H_N$ is the sum of at most one term bounded by $\sup|\varphi'|<1$ and one term bounded by
$\sup|\Theta'|\sup\varphi\le3$. Hence
\begin{equation}
\label{eq:endpoint-H-bounds}
0\le H_N(x)\le N,
\qquad
\norm{\nabla H_N(x)}\le4\sqrt N.
\end{equation}

We also need three estimates for $B$. We verify them on every piece of its definition. If $t\le-1/2$, then
$B(t)=2+(t+1/2)^2/2$ and $B'(t)=t+1/2$, so
\[
B(t)-\left(\frac12t^2+2\right)
=\frac12t+\frac18\le0,
\qquad
\abs{B'(t)-t}=\frac12,
\]
and
\[
tB'(t)-\left(\frac12t^2-\frac{13}{2}\right)
=\frac12\left(t+\frac12\right)^2+\frac{51}{8}\ge0.
\]
If $-1/2<t\le1/2$, then $B(t)=2\le t^2/2+2$, $\abs{B'(t)-t}=\abs{t}\le1/2$, and
$tB'(t)=0\ge t^2/2-13/2$. If
$1/2<t<1$, then $0\le B(t)\le2$ and $B'(t)=-2\Theta'(t)$. Since $0<t<1$ and
$0\le\Theta'(t)\le3$,
\[
B(t)\le2\le\frac12t^2+2,
\qquad
\abs{B'(t)-t}\le7,
\qquad
tB'(t)\ge-6\ge\frac12t^2-\frac{13}{2}.
\]
Finally, if $t\ge1$, then $B(t)=(t-1)^2/2$ and $B'(t)=t-1$, whence
\[
B(t)\le\frac12t^2+2,
\qquad
\abs{B'(t)-t}=1,
\]
and
\[
tB'(t)-\left(\frac12t^2-\frac{13}{2}\right)
=\frac12(t-1)^2+6\ge0.
\]
Thus, globally,
\begin{equation}
\label{eq:endpoint-B-bounds}
B(t)\le\frac12t^2+2,
\qquad
\abs{B'(t)-t}\le7,
\qquad
tB'(t)\ge\frac12t^2-\frac{13}{2}.
\end{equation}

For $r:=\norm{x}$, equations~\eqref{eq:endpoint-H-bounds}--\eqref{eq:endpoint-B-bounds} imply
\begin{equation}
\label{eq:endpoint-Gamma-global-growth}
\Gamma_N(x)
\le
\left(\frac12+\frac{\eta}{2}\right)r^2+3N
=\frac{21}{40}r^2+3N
\end{equation}
and
\begin{equation}
\label{eq:endpoint-Gamma-radial}
\ip{\nabla\Gamma_N(x)}{x}
\ge
\left(\frac12+\eta\right)r^2-4\sqrt N\,r-\frac{13}{2}N
=\frac{11}{20}r^2-4\sqrt N\,r-\frac{13}{2}N.
\end{equation}

\emph{Step 2: the sharper smoothness bound.}
For any $y\in\R^N$ with $\norm{y}\le2R_m$, equations~\eqref{eq:endpoint-H-bounds}--\eqref{eq:endpoint-B-bounds}
give
\[
\norm{\nabla F_N(y)}
\le
\norm{\nabla H_N(y)}+\norm{B'(y)}
\le
\norm{y}+11\sqrt N
\le
3R_m,
\]
where the last inequality uses $\norm{y}\le2R_m$ and $11\sqrt N\le R_m$. For arbitrary $x,x'\in\R^d$, set
$y:=U^\top\rho_m(x)$ and $y':=U^\top\rho_m(x')$. The clipping map is $1$-Lipschitz,
$\sup_x\norm{J\rho_m(x)}_{\mathrm{op}}\le1$, and Lemma~\ref{lem:radial-clipping} gives
$\mathrm{Lip}(J\rho_m)\le11/R_m$. Therefore
\begin{align*}
&\norm{J\rho_m(x)^\top U\nabla F_N(y)-J\rho_m(x')^\top U\nabla F_N(y')}\\
&\qquad\le
79\norm{x-x'}+\frac{11}{R_m}(3R_m)\norm{x-x'}
=112\norm{x-x'}.
\end{align*}
Adding the gradient $\eta x$ of the quadratic regularizer shows that $\nabla\widehat F_{m,U}$ is globally
$(112+\eta)$-Lipschitz, and hence globally $113$-Lipschitz.

\emph{Step 3: a global P\L{} inequality for $\Gamma_N$.}
Let $x^\star$ be a global minimizer of $\Gamma_N$. It exists by Lemma~\ref{lem:Gamma-basic}(1), and
Lemma~\ref{lem:Gamma-outside-box} implies $x^\star\in\mathcal B_N$ because
$\nabla\Gamma_N(x^\star)=0$. If $x\in\mathcal B_N$, Lemma~\ref{lem:Gamma-box-strong} and the standard gradient
bound for strongly convex functions give
\[
\Gamma_N(x)-\Gamma_{N,\star}\le\norm{\nabla\Gamma_N(x)}^2.
\]

Suppose next that $x\notin\mathcal B_N$ and $r\le9\sqrt N$. The four cases in the proof of
Lemma~\ref{lem:Gamma-outside-box} actually give the sharper component bounds
$1/8$, $9/160$, $3/40$, and $1/8$, respectively. Hence
$\norm{\nabla\Gamma_N(x)}\ge9/160$. Meanwhile,
\eqref{eq:endpoint-Gamma-global-growth} and $\Gamma_{N,\star}\ge0$ give
\[
\Gamma_N(x)-\Gamma_{N,\star}
\le\frac{1821}{40}N
\le14{,}389N\norm{\nabla\Gamma_N(x)}^2.
\]

It remains to consider $x\notin\mathcal B_N$ and $r>9\sqrt N$. The difference between the right-hand side of
\eqref{eq:endpoint-Gamma-radial} and $r^2/40$ is
\[
\frac{21}{40}r^2-4\sqrt N\,r-\frac{13}{2}N.
\]
This expression is increasing for $r\ge9\sqrt N$ and equals $N/40>0$ at $r=9\sqrt N$. Therefore
$\ip{\nabla\Gamma_N(x)}{x}\ge r^2/40$, and Cauchy--Schwarz yields
$\norm{\nabla\Gamma_N(x)}\ge r/40$. Using~\eqref{eq:endpoint-Gamma-global-growth} and
$3N\le r^2/27$ once more,
\[
\Gamma_N(x)-\Gamma_{N,\star}
\le
\left(\frac{21}{40}+\frac{1}{27}\right)r^2
\le900\norm{\nabla\Gamma_N(x)}^2.
\]
Combining the three regions gives the global estimate
\begin{equation}
\label{eq:endpoint-Gamma-global-PL}
\Gamma_N(x)-\Gamma_{N,\star}
\le
C_\Gamma N\norm{\nabla\Gamma_N(x)}^2
\qquad(\forall x\in\R^N),
\quad
C_\Gamma:=14{,}389.
\end{equation}

\emph{Step 4: points inside the clipping ball.}
Assume $\norm{x}\le R_m$, so $\rho_m(x)=x$, and define
\[
y:=U^\top x,
\qquad
w:=x-Uy.
\]
The vectors $Uy$ and $w$ are orthogonal. Step~4 in the proof of Lemma~\ref{lem:clipped-rotated} establishes
$\widehat F_{m,U,\star}=\Gamma_{N,\star}$; it follows by minimizing $\Gamma_N$ and observing that its minimizer lies
strictly inside the clipping ball. Consequently,
\begin{equation}
\label{eq:endpoint-inside-decomposition}
\widehat F_{m,U}(x)
=\Gamma_N(y)+\frac{\eta}{2}\norm{w}^2,
\qquad
\norm{\nabla\widehat F_{m,U}(x)}^2
=\norm{\nabla\Gamma_N(y)}^2+\eta^2\norm{w}^2.
\end{equation}
Since $N=2m$, equations~\eqref{eq:endpoint-Gamma-global-PL}--\eqref{eq:endpoint-inside-decomposition} give
\[
\widehat F_{m,U}(x)-\widehat F_{m,U,\star}
\le
28{,}778m\norm{\nabla\Gamma_N(y)}^2+10\eta^2\norm{w}^2
\le
28{,}778m\norm{\nabla\widehat F_{m,U}(x)}^2.
\]

\emph{Step 5: points outside the clipping ball.}
Let $r:=\norm{x}>R_m$ and introduce
\begin{align*}
e&:=\frac{x}{r},
&
h&:=h_m(r),
&
a&:=\frac{h}{r},
&
z&:=U^\top e,\\
y&:=hz,
&
g&:=\nabla F_N(y),
&
\beta&:=\ip{g}{z}.
\end{align*}
The radial and tangential eigenvalues of $J\rho_m(x)$ are $h_m'(r)$ and $a$, respectively. Since
$\ip{Ug}{e}=\beta$, the two orthogonal gradient components are
\begin{equation}
\label{eq:endpoint-exterior-gradient}
\nabla\widehat F_{m,U}(x)
=
\bigl(h_m'(r)\beta+\eta r\bigr)e
+a\bigl(Ug-\beta e\bigr),
\end{equation}
and hence
\begin{equation}
\label{eq:endpoint-exterior-gradient-norm}
\norm{\nabla\widehat F_{m,U}(x)}^2
=
\bigl(h_m'(r)\beta+\eta r\bigr)^2
+a^2\bigl(\norm{g}^2-\beta^2\bigr).
\end{equation}
Applying~\eqref{eq:endpoint-H-bounds}--\eqref{eq:endpoint-B-bounds} to $F_N$ at $y$ also gives
\begin{equation}
\label{eq:endpoint-F-radial}
\ip{g}{y}
\ge
\frac12\norm{y}^2-4\sqrt N\norm{y}-\frac{13}{2}N.
\end{equation}

We now lower-bound the gradient in two cases. If $\beta\ge-\eta r/2$, then $0<h_m'(r)\le1$ implies
\[
h_m'(r)\beta+\eta r\ge\frac{\eta r}{2}.
\]
If instead $\beta<-\eta r/2$, then $\ip{g}{y}=h\beta<0$. Equation~\eqref{eq:endpoint-F-radial} therefore forces
\[
\norm{y}<(4+\sqrt{29})\sqrt N.
\]
Because $h\ge R_m=8\sqrt2\sqrt N$, we obtain
\[
\norm{z}\le s_0:=\frac{4+\sqrt{29}}{8\sqrt2}<1,
\qquad
s_0^2=\frac{45+8\sqrt{29}}{128}<\frac{7}{10}.
\]
Here $z\ne0$, since otherwise $\beta=0$. Cauchy--Schwarz then yields
\[
\norm{g}^2-\beta^2
\ge
\beta^2\left(\frac{1}{\norm{z}^2}-1\right)
\ge
\frac{3}{7}\beta^2.
\]
Moreover, $a|\beta|>\eta h/2\ge\eta R_m/2$. Thus the tangential component in
\eqref{eq:endpoint-exterior-gradient-norm} is bounded away from zero even if the radial component nearly cancels.

We finish by separating moderate and large radii. Set $K:=16$. If $R_m<r\le KR_m$, the two cases above, together
with $R_m^2=128N$, give
\begin{equation}
\label{eq:endpoint-annulus-gradient}
\norm{\nabla\widehat F_{m,U}(x)}^2
\ge
\frac{3}{7}\frac{\eta^2R_m^2}{4}
=\frac{6}{175}N.
\end{equation}
On the other hand, $\norm{y}\le h<2R_m$, $F_N(y)\le\norm{y}^2/2+3N$, and
$\widehat F_{m,U,\star}\ge0$, so
\[
\widehat F_{m,U}(x)-\widehat F_{m,U,\star}
\le
2R_m^2+3N+\frac{\eta}{2}K^2R_m^2
=\frac{5391}{5}N.
\]
Together with~\eqref{eq:endpoint-annulus-gradient}, this gives
\[
\widehat F_{m,U}(x)-\widehat F_{m,U,\star}
\le
31{,}448\norm{\nabla\widehat F_{m,U}(x)}^2.
\]

Finally, suppose $r>KR_m$ and write $s:=r/R_m\ge16$. The estimate proved in Step~2 gives
$\norm{g}\le3R_m$. The radial and tangential eigenvalues of
$J\rho_m(x)$ also give
\[
\norm{J\rho_m(x)}_{\mathrm{op}}
\le
\max\left\{e^{-(s-1)},\frac{2}{s}\right\}
\le\frac{2}{s},
\]
where $e^{s-1}\ge s$ for $s\ge1$. Since $s^2\ge16^2>240$,
\[
\norm{J\rho_m(x)^\top Ug}
\le
\frac{6R_m}{s}
\le
\frac{r}{40}
=\frac{\eta r}{2}.
\]
The quadratic regularizer therefore dominates:
\[
\norm{\nabla\widehat F_{m,U}(x)}\ge\frac{\eta r}{2}=\frac{r}{40}.
\]
Furthermore,
\[
\widehat F_{m,U}(x)-\widehat F_{m,U,\star}
\le259N+\frac{r^2}{40}
\le\frac{259}{16^2\cdot128}r^2+\frac{r^2}{40}
\le53\norm{\nabla\widehat F_{m,U}(x)}^2,
\]
where $r^2\ge16^2R_m^2=16^2\cdot128N$ was used in the second inequality.

The clipping ball, the annulus, and the large-radius region cover $\R^d$. Their bounds are all at most
$28{,}778m\norm{\nabla\widehat F_{m,U}(x)}^2$ because $m\ge m_0=5$. This proves the claimed global P\L{} bound.
\end{proof}
\endgroup

Recall that $m_0=5$ and that $d_{m,q}$ is the dimension threshold in~\eqref{eq:dmq-threshold}. Fix a chain length
$m\ge m_0$, a reveal probability $q\in(0,1]$, and a dimension $d\ge d_{m,q}$. For each fixed
$U\in\Ortho(d,2m)$, Lemma~\ref{lem:clipped-rotated} constructs the objective-oracle pair
$(\widehat F_{m,U},\widehat g_{m,q,U})$. Using the sharper endpoint smoothness and global P\L{}
bounds from Lemma~\ref{lem:endpoint-global-pl}, define the constants needed below by
\begingroup
\begin{equation}
\label{eq:endpoint-base-constants}
L_{\rm b}:=113,
\qquad
C_{\mathrm{gPL}}:=28{,}778,
\qquad
C_\Delta:=5,
\qquad
C_{\rm var}:=4,
\qquad
c_0:=\frac1{16}.
\end{equation}
Here $L_{\rm b}$ is the sharper endpoint smoothness bound from Lemma~\ref{lem:endpoint-global-pl},
$C_{\mathrm{gPL}}m$ bounds the global P\L{} constant,
$C_\Delta m$ bounds the initial gap, $C_{\rm var}/q$ bounds the oracle variance, and $c_0m/q$ is the number of
calls for which slow coordinate revelation produces an objective barrier. More precisely, for every randomized
algorithm $A$, if $U$ is uniform on $\Ortho(d,2m)$, then
\[
\E_{U,\xi,A}\bigl[
\widehat F_{m,U}(x_t^A)-\widehat F_{m,U,\star}
\bigr]
\ge m
\qquad
\left(t\le c_0\frac{m}{q}\right).
\]
The expectation is over the random rotation $U$, the oracle randomness $\xi$, and the internal randomness of $A$.
In particular, this is an average-over-$U$ guarantee, not a pointwise claim for every fixed rotation.
\endgroup

\begingroup
We first handle the range in which a chain of length at least $m_0$ can be fitted to the prescribed value of
$L\tau$. Define the numerical threshold
\begin{equation}
\label{eq:endpoint-kappa-threshold}
\kappa_0:=C_{\mathrm{gPL}}L_{\rm b}m_0=16{,}259{,}570.
\end{equation}
\endgroup

\begingroup
\begin{proposition}[Endpoint scaling of the probabilistic zero chain]
\label{prop:endpoint-large-condition}
Suppose $L\tau>\kappa_0$ and $0<\varepsilon\le\Delta_0/10$. Then
\[
T_\varepsilon^{\mathrm{BV,glob}}(L,\tau,\Delta_0,\sigma^2)
\ge
\frac{25c_0}{72C_{\rm var}C_{\mathrm{gPL}}^2L_{\rm b}}
\frac{L\tau^2\sigma^2}{\varepsilon}.
\]
\end{proposition}
\endgroup

\begin{proof}
\begingroup
Write $\kappa:=L\tau$. The chain length must be small enough that, after imposing smoothness $L$, the scaled global
P\L{} constant does not exceed $\tau$. We therefore choose
\begin{equation}
\label{eq:endpoint-scaling-choice}
m:=\left\lfloor\frac{\kappa}{C_{\mathrm{gPL}}L_{\rm b}}\right\rfloor.
\end{equation}
Because $\kappa>\kappa_0$, we have $m\ge m_0$ and
\begin{equation}
\label{eq:endpoint-m-bounds}
\frac{5\kappa}{6C_{\mathrm{gPL}}L_{\rm b}}
<m
\le
\frac{\kappa}{C_{\mathrm{gPL}}L_{\rm b}},
\qquad
\kappa<\frac65C_{\mathrm{gPL}}L_{\rm b}m.
\end{equation}
Indeed, if $s:=\kappa/(C_{\mathrm{gPL}}L_{\rm b})$, then $s>m_0=5$, so $m=\lfloor s\rfloor\ge5$ and
$s<m+1\le6m/5$. These inequalities give all three bounds above.
\endgroup

Next choose the objective scale
\[
\lambda:=\frac{2\varepsilon}{m}.
\]
The normalized expected residual lower bound $m$ will therefore become $\lambda m=2\varepsilon$. Finally, choose the
domain scale
\[
\rho^2:=\frac{L}{\lambda L_{\rm b}}.
\]
This choice makes the smoothness bound of the scaled objective equal to $L$.

For a value of $q\in(0,1]$ chosen below, take any dimension $d\ge d_{m,q}$ from
\eqref{eq:dmq-threshold}, draw $U$ uniformly from $\Ortho(d,2m)$, and define the scaled function-oracle pair
\begin{equation}
\label{eq:endpoint-scaled-pair}
f_{m,U}(x):=\lambda\widehat F_{m,U}(\rho x),
\qquad
G_{m,q,U}(x,z):=\lambda\rho\,\widehat g_{m,q,U}(\rho x,z).
\end{equation}

\begingroup
We now verify membership in the global-P\L{} class for every fixed $U$. Differentiability and lower boundedness
follow from the corresponding properties of $\widehat F_{m,U}$ and from $\lambda,\rho>0$; moreover,
$(f_{m,U})_\star=\lambda\widehat F_{m,U,\star}$. The gradient is globally Lipschitz with constant at most
\[
\mathrm{Lip}(\nabla f_{m,U})
\le
\lambda\rho^2L_{\rm b}=L.
\]
Thus $f_{m,U}$ is globally $L$-smooth. For every $x\in\R^d$, Lemma~\ref{lem:endpoint-global-pl} and
$\nabla f_{m,U}(x)=\lambda\rho\nabla\widehat F_{m,U}(\rho x)$ give
\begin{align*}
f_{m,U}(x)-(f_{m,U})_\star
&\le
\frac{C_{\mathrm{gPL}}m}{\lambda\rho^2}\norm{\nabla f_{m,U}(x)}^2\\
&=
\frac{C_{\mathrm{gPL}}mL_{\rm b}}{L}\norm{\nabla f_{m,U}(x)}^2
\le
\tau\norm{\nabla f_{m,U}(x)}^2.
\end{align*}
The inequality holds on the whole space, so the global P\L{} constant is at most $\tau$. Finally, the initial
gap satisfies
\begin{equation}
\label{eq:endpoint-scaled-initial-gap}
f_{m,U}(0)-(f_{m,U})_\star
\le
C_\Delta\lambda m
=
10\varepsilon
\le
\Delta_0.
\end{equation}
This verifies all function-class requirements.
\endgroup

It remains to choose the reveal probability so that the scaled oracle respects the variance budget. Its variance is
at most $C_{\rm var}\lambda L/(L_{\rm b}q)$, so the value that exactly saturates the budget is
\begin{equation}
\label{eq:endpoint-qstar}
q_\star:=\frac{C_{\rm var}\lambda L}{L_{\rm b}\sigma^2}.
\end{equation}

\emph{Case 1: $q_\star\le1$.}
Set $q=q_\star$, so the probabilistic oracle uses the entire variance budget. Lemma~\ref{lem:clipped-rotated}(1)--(2)
and the scaling in
\eqref{eq:endpoint-scaled-pair} show that $G_{m,q,U}$ is unbiased and has conditional variance at most
\[
(\lambda\rho)^2\frac{C_{\rm var}}q
=
\frac{C_{\rm var}\lambda L}{L_{\rm b}q}
=
\sigma^2.
\]
Thus the scaled pair is admissible. Revealing order $m$ chain coordinates requires order $m/q$ calls; using
$\lambda=2\varepsilon/m$ and the lower bound on $m$ in
\eqref{eq:endpoint-m-bounds},
\begingroup
\begin{align}
\frac{m}{q}
&=
\frac{mL_{\rm b}\sigma^2}{C_{\rm var}\lambda L}
=
\frac{m^2L_{\rm b}\sigma^2}{2C_{\rm var}\varepsilon L}
\notag\\
&\ge
\frac{25}{72C_{\rm var}C_{\mathrm{gPL}}^2L_{\rm b}}
\frac{L\tau^2\sigma^2}{\varepsilon}.
\label{eq:endpoint-low-q-time}
\end{align}
\endgroup
The last expression has exactly the claimed dependence on $(L,\tau,\sigma^2,\varepsilon)$.

\emph{Case 2: $q_\star>1$.}
The variance-saturating value is now larger than a probability, so set $q=1$. Since $q$ is the reveal probability,
the variable $z\sim\mathrm{Bernoulli}(1)$ equals one almost surely. From the definition of the base oracle,
\[
g_{m,1}(x,z)=\nabla F_{2m}(x)
\qquad\text{almost surely}.
\]
Equations~\eqref{eq:hatF-def}--\eqref{eq:hatg-def} and~\eqref{eq:endpoint-scaled-pair} therefore imply
\[
G_{m,1,U}(x,z)=\nabla f_{m,U}(x)
\qquad\text{almost surely}.
\]
This oracle is exact, so its variance is zero and it is admissible for every positive variance budget. The condition
$q_\star>1$, followed by the last bound in~\eqref{eq:endpoint-m-bounds}, gives
\begingroup
\begin{align*}
\frac{L\tau^2\sigma^2}{\varepsilon}
&<
\frac{2C_{\rm var}(L\tau)^2}{mL_{\rm b}}
\le
\frac{72}{25}C_{\rm var}C_{\mathrm{gPL}}^2L_{\rm b}m.
\end{align*}
Consequently,
\begin{equation}
\label{eq:endpoint-q-one-time}
m
\ge
\frac{25}{72C_{\rm var}C_{\mathrm{gPL}}^2L_{\rm b}}
\frac{L\tau^2\sigma^2}{\varepsilon}.
\end{equation}
\endgroup
Hence the order-$m$ exact-gradient barrier is already at least a constant multiple of the target lower bound.

We now convert the average zero-chain barrier into a minimax lower bound; this step is the same in both cases. Fix an
arbitrary randomized, adaptive, arbitrary-query, gradient-only method $A$ in dimension $d$. To simulate its run on
the scaled pair using the normalized pair, map every query $x_s$ to $\rho x_s$ and multiply the returned normalized
stochastic gradient by $\lambda\rho$. The simulated transcript has exactly the same law as the transcript generated
by $G_{m,q,U}$. Lemma~\ref{lem:clipped-rotated}(5) therefore gives
\[
\E_{U,\xi,A}\bigl[f_{m,U}(x_t^A)-(f_{m,U})_\star\bigr]
\ge
\lambda m
=
2\varepsilon
\qquad
\left(t\le c_0\frac{m}{q}\right).
\]
Here the expectation is over the random rotation $U$, the oracle randomness $\xi$, and the internal randomness of
$A$; when $q=1$, the oracle itself is deterministic. The output-after-$t$-calls convention causes no additional
issue: as verified in the proof of Lemma~\ref{lem:clipped-rotated}(5), the output can be appended as the next adaptive
query without using the resulting oracle response.

For fixed $A$, $t$, and $U$, define its expected risk by
\[
R_A(U):=
\E_{\xi,A}\bigl[f_{m,U}(x_t^A)-(f_{m,U})_\star\bigr].
\]
Then
\[
\sup_{U\in\Ortho(d,2m)}R_A(U)
\ge
\E_U R_A(U)
\ge2\varepsilon.
\]
Thus, for each algorithm $A$, at least one fixed rotation $U$ has expected error at least $2\varepsilon$. This uses
only $\sup_U R_A(U)\ge\E_U R_A(U)$: it neither invokes Yao's principle nor interchanges the infimum over algorithms
with the expectation over $U$. Since every resulting pair belongs to the global-P\L{} class and the dimension-free minimax class contains
the chosen dimension $d\ge d_{m,q}$, every integer $t\le c_0m/q$ fails to guarantee error $\varepsilon$. Therefore
\begingroup
\[
T_\varepsilon^{\mathrm{BV,glob}}(L,\tau,\Delta_0,\sigma^2)
\ge
c_0\frac{m}{q}.
\]
\endgroup
Combining this with~\eqref{eq:endpoint-low-q-time} in Case~1 and
\eqref{eq:endpoint-q-one-time} in Case~2 proves the proposition.
\end{proof}

The zero-chain argument requires $m\ge m_0$, so Proposition~\ref{prop:endpoint-large-condition} applies only when
$L\tau>\kappa_0$. The next one-dimensional argument covers the remaining bounded range of $L\tau$.

\begingroup
\begin{lemma}[Gaussian two-point patch for bounded $L\tau$]
\label{lem:endpoint-gaussian-two-point}
Let $L,\tau,\Delta_0,\sigma^2>0$ satisfy $L\tau\ge1/2$. If
$0<\varepsilon\le\Delta_0/8$, then
\[
T_\varepsilon^{\mathrm{BV,glob}}(L,\tau,\Delta_0,\sigma^2)
\ge
\frac{\tau\sigma^2}{32\varepsilon}.
\]
\end{lemma}
\endgroup

\begin{proof}
Set the curvature and separation parameters to
\[
\mu:=\frac1{2\tau},
\qquad
r:=4\sqrt{\frac{\varepsilon}{\mu}},
\]
and, for $v\in\{-r,r\}$, define on $\R$
\begin{equation}
\label{eq:endpoint-two-point-pair}
f_v(x):=\frac{\mu}{2}(x-v)^2,
\qquad
G_v(x,\xi):=\mu(x-v)+\xi,
\qquad
\xi\sim\mathcal N(0,\sigma^2),
\end{equation}
with a fresh independent Gaussian sample at each oracle call. We first verify that both pairs are admissible. The
assumption $L\tau\ge1/2$ is equivalent to $\mu=(2\tau)^{-1}\le L$. Therefore the quadratic $f_v$ is differentiable,
lower bounded, and globally $L$-smooth; it attains its minimum value zero at $x=v$. Moreover,
\[
f_v(x)-(f_v)_\star
=
\frac{\mu}{2}(x-v)^2
=
\tau\abs{f_v'(x)}^2.
\]
Thus $f_v$ satisfies the global P\L{} inequality with constant $\tau$ and belongs to the endpoint
global-P\L{} class. At the
initial point $0$, its gap is
\[
f_v(0)-(f_v)_\star=\frac{\mu r^2}{2}=8\varepsilon\le\Delta_0.
\]
Finally, conditional on any query $x$,
\[
\E[G_v(x,\xi)\mid x]=\mu(x-v)=f_v'(x),
\qquad
\E\bigl[\abs{G_v(x,\xi)-f_v'(x)}^2\mid x\bigr]=\sigma^2.
\]
Hence both function-oracle pairs belong to the required global-P\L{}, gradient-only bounded-variance
class.

We now compare the information supplied by these two oracles. Fix any randomized adaptive algorithm making $T$
oracle calls, and include its internal random seed in the transcript. Let $P_+^T$ and $P_-^T$ denote the full
transcript laws under $v=r$ and $v=-r$, respectively. Conditional on the same past transcript, the seed and previous
responses determine the algorithm's next query $x_t$. The next oracle response is Gaussian with variance $\sigma^2$;
its two possible means are $\mu(x_t-r)$ and $\mu(x_t+r)$. Their difference has magnitude $2\mu r$, which does not
depend on $x_t$. Thus adaptivity does not change the conditional KL divergence contributed by a query. The chain rule
for relative entropy gives
\begin{equation}
\label{eq:endpoint-transcript-kl}
\mathrm{KL}(P_+^T\Vert P_-^T)
=
T\frac{(2\mu r)^2}{2\sigma^2}
=
\frac{32T\mu\varepsilon}{\sigma^2}.
\end{equation}
The random seed contributes zero relative entropy because it has the same distribution under the two instances.

To convert indistinguishability into objective error, let $V$ be uniform on $\{-1,1\}$ and set $v=Vr$. Let
$\widehat x_T$ be the algorithm's output, and use its sign as a test: set $\widehat V=1$ when
$\widehat x_T\ge0$ and $\widehat V=-1$ otherwise. If $\widehat V\ne V$, then
$\abs{\widehat x_T-v}\ge r$, so
\[
f_v(\widehat x_T)-(f_v)_\star
\ge
\frac{\mu r^2}{2}
=8\varepsilon.
\]
For any such test, Le Cam's inequality~\citep{Yu1997} followed by Pinsker's inequality gives
\[
\Pbb(\widehat V\ne V)
\ge
\frac{1-\norm{P_+^T-P_-^T}_{\rm TV}}{2}
\ge
\frac{1-\sqrt{\mathrm{KL}(P_+^T\Vert P_-^T)/2}}{2}.
\]
If $T\le\sigma^2/(64\mu\varepsilon)$, then~\eqref{eq:endpoint-transcript-kl} is at most $1/2$, and hence the testing
error is at least $1/4$. Multiplying the objective loss on the testing-error event by this probability yields
\[
\frac12\sum_{v\in\{-r,r\}}
\E_v\bigl[f_v(\widehat x_T)-(f_v)_\star\bigr]
\ge
\frac{\mu r^2}{2}\cdot\frac14
=
2\varepsilon.
\]
The worst-case expected error over the two admissible pairs is at least their average. Therefore every integer
$T\le\sigma^2/(64\mu\varepsilon)$ fails to guarantee error at most $\varepsilon$, and
\begingroup
\[
T_\varepsilon^{\mathrm{BV,glob}}(L,\tau,\Delta_0,\sigma^2)
\ge
\frac{\sigma^2}{64\mu\varepsilon}
=
\frac{\tau\sigma^2}{32\varepsilon}.
\]
\endgroup
Because the calculation used the full transcript, including the internal seed and adaptively selected queries, it
applies to every randomized adaptive algorithm in the model.
\end{proof}

\begingroup
\begin{proof}[Proof of Theorem~\ref{thm:stochastic-pl-endpoint}]
Set
\[
c_{\mathrm{acc}}^{\mathrm{end}}:=\frac1{10}
\]
and suppose $0<\varepsilon\le\Delta_0/10$. Let $\kappa:=L\tau\ge1/2$.

If $\kappa>\kappa_0$, Proposition~\ref{prop:endpoint-large-condition} gives
\[
T_\varepsilon^{\mathrm{BV,glob}}(L,\tau,\Delta_0,\sigma^2)
\ge
\frac{25c_0}{72C_{\rm var}C_{\mathrm{gPL}}^2L_{\rm b}}
\frac{L\tau^2\sigma^2}{\varepsilon}.
\]
If $1/2\le\kappa\le\kappa_0$, then $\varepsilon\le\Delta_0/10<\Delta_0/8$, so
Lemma~\ref{lem:endpoint-gaussian-two-point} applies. In this bounded range,
\[
\frac{\tau\sigma^2}{\varepsilon}
\ge
\frac{1}{\kappa_0}
\frac{L\tau^2\sigma^2}{\varepsilon},
\]
and the lemma therefore yields
\[
T_\varepsilon^{\mathrm{BV,glob}}(L,\tau,\Delta_0,\sigma^2)
\ge
\frac{\tau\sigma^2}{32\varepsilon}
\ge
\frac{1}{32\kappa_0}
\frac{L\tau^2\sigma^2}{\varepsilon}.
\]
The two cases cover every $L\tau\ge1/2$. Taking
\[
c_{\mathrm{end}}
:=
\min\left\{
\frac{1}{32\kappa_0},
\frac{25c_0}{72C_{\rm var}C_{\mathrm{gPL}}^2L_{\rm b}}
\right\},
\]
with the constants from~\eqref{eq:endpoint-base-constants} and~\eqref{eq:endpoint-kappa-threshold}, proves the
theorem.
\end{proof}
\endgroup

\end{document}